\documentclass[11pt, reqno]{amsart}
\usepackage[final]{showkeys}
\usepackage{amssymb}
\usepackage{amsmath}
\usepackage{amsthm}
\usepackage{verbatim}
\usepackage{color}
\usepackage{url}
\usepackage[all]{xy}

\usepackage{graphicx,amsmath,amsfonts,latexsym,amssymb,amsthm,color}
\usepackage{latexsym}
\usepackage{verbatim}
\usepackage{enumerate}
\usepackage{enumitem}
\usepackage{tikz-cd,mathtools}

\usepackage{tikz}
\usepackage{pgfplotstable}
\usepackage[makeroom,thicklines]{cancel}
\usetikzlibrary{shapes,arrows,calendar,matrix,backgrounds,folding,hobby,calc,positioning,patterns,fadings,shadows,spy,hobby,fit}
\usepackage{tkz-euclide}
\usepackage[font=small,labelfont=bf]{caption}
\usetikzlibrary{patterns.meta}

\definecolor{blau}{rgb}{0.05,0.2,0.7}
\definecolor{auchblau}{rgb}{0.03,0.3,0.7}
\usepackage[pdftex,linkbordercolor={0.8 0.5 0.2}]{hyperref} 

\newtheorem{theorem}{Theorem}[section]
\newtheorem*{theorem*}{Theorem}
\newtheorem{prop}[theorem]{Proposition}
\newtheorem{lem}[theorem]{Lemma}
\newtheorem{cor}[theorem]{Corollary}
\newtheorem{defn}[theorem]{Definition}

\newtheorem{rem}[theorem]{Remark}
\definecolor{light-gray}{RGB}{176, 176, 176}

\newcommand{\R}{\mathbb{R}}

\newcommand{\N}{\mathbb{N}}

\renewcommand{\Re}{\operatorname{Re}}
\renewcommand{\Im}{\operatorname{Im}}

\newcommand{\Op}{\operatorname{Op}}
\newcommand{\WF}{\operatorname{WF}_h}
\newcommand{\Ellip}{\operatorname{ell}_h}

\DeclareMathOperator{\supp}{supp}
\newcommand{\der}{\mathrm{d}}
\newcommand{\rmi}{\mathrm{i}}

\newcommand{\tr}{\mathrm{Tr}}

\newcommand{\sgn}{\operatorname{sgn}}

\newcommand{\inner}{{_{\mathcal{I}}}}
\newcommand{\exterior}{{_{\mathcal{O}}}}
\newcommand{\Eucl}{\exterior}
\newcommand{\tangent}{\mathrm{t}}

\newcommand{\vol}{{\mathrm{vol}}}
\newcommand{\comp}{\mathrm{comp}}

\newcommand{\loc}{\mathrm{loc}}

\newcommand{\calG}{\mathcal{G}}

\newcommand{\calR}{\mathcal{R}}

\renewcommand{\div}{\operatorname{div}}
\newcommand{\e}{\varepsilon}

\numberwithin{equation}{section}
\newcommand{\blue}[1]{{{#1}}}

\title[Surface Plasmons in Metamaterial cavities]{Surface plasmons in metamaterial cavities
}
\author{Yan-Long Fang}
\address{Department of Mathematics, University College London, 25 Gordon Street, London, WC1H 0AY, UK}
\email{yanlong.fang@ucl.ac.uk} 
\author{Jeffrey Galkowski} \address{Department of Mathematics, University College London, 25 Gordon Street, London, WC1H 0AY, UK}
\email{j.galkowski@ucl.ac.uk} 

\begin{document}

	\begin{abstract}
    We study scattering by metamaterials with negative indices of refraction, which are known to support \emph{surface plasmons} -- long-lived states that are highly localized at the boundary of the cavity. This type of state has found uses in a variety of modern technologies. In this article, we study surface plasmons in the setting of non-trapping cavities; i.e. when all billiard trajectories outside the cavity escape to infinity. We characterize the indices of refraction which support surface plasmons, show that the corresponding resonances lie super-polynomially close to the real axis, describe the localization properties of the corresponding resonant states, and give an asymptotic formula for their number. 
		
	\end{abstract}
    	\maketitle
	
	
	\vspace{-1cm}
	\section{Introduction}

	We consider resonance phenomena for metamaterial cavities which exhibit a negative index of refraction or negative wave speed. These structures are known, in some contexts, to support surface plasmons -- long lived states that are highly localized to the surface of the metamaterial~\cite{Maier}. These surface plasmons offer strong light enhancement and are central to a range of modern technologies~\cite{Sanno}. Although negative index of refraction metamaterials have attracted some mathematical interest (see~\cite{BoChCi:12,BoChCl:13,BoChCi:14,MR3614756,MR3686000,BoHaMo:21,CaMo:23,DaBaCaMo:24,CdV:25}) their asymptotic behavior has remained largely unexplored. Under a relatively mild assumption on the metamaterial scatterer, we study plasmon resonances in a scalar model (i.e. in the Transverse Electric or Transverse Magnetic polarization). In this article, we give an accurate description of the asymptotic behavior of surface plasmons. We characterize the existence and absence of surface plasmons, accurately describe their localization properties, and provide an asymptotic formula for their number. 
	
	Let $d\ge 2$ and $\Omega_\inner\subset\mathbb{R}^d$ be a bounded open domain with smooth boundary and connected complement. 
    Define $\Omega_\exterior:=\R^d\backslash\overline{\Omega_\inner}$. We denote the shared boundary of $\Omega_{\inner}$ and $\Omega_{\exterior}$ by $\partial \Omega$ and the outward pointing normal of $\Omega_\inner$ by $\nu$. Although we work in the more general setting of negative wave speeds below (see section~\ref{s:general}), we state first a simple consequence of our main theorem. Let the \emph{index of refraction} 
    $$
    n\in C^\infty(\overline{\Omega_{\inner}};(0,\infty)),\quad\text{with}\quad|n|_{\partial\Omega}-1|>0.
    $$
    We call $\lambda\in \mathbb{C}\setminus \rmi (-\infty,\infty)$ a \emph{resonance} if there is a  $0\neq (u_{\inner},u_{\exterior})\in H^2(\Omega_{\inner})\oplus H^2_{\loc}(\Omega_{\exterior})$ satisfying
    \begin{equation}
    \label{e:nProblem}
    \begin{cases}
    (\div n^{-1}\nabla -\lambda^2)u_{\inner}=0&\text{in }\Omega_{\inner},\\
    (-\Delta-\lambda^2)u_{\exterior}=0&\text{in }\Omega_{\exterior},\\
    u_{\exterior}=u_{\inner}&\text{on }\partial\Omega,\\
    \partial_{\nu}u_{\exterior}=-n^{-1}\partial_{\nu}u_{\inner}&\text{on }\partial\Omega,\\
    u_{\exterior}\text{ is }\lambda\text{-outgoing}.
    \end{cases}
    \end{equation}
    Here, we say that $u$ is $\lambda$-outgoing if there is $g\in L^2_{\comp}(\mathbb{R}^d)$ such that $u(x)=[R_0(\lambda)g](x)$ for $|x|\gg 1$, where $R_0(\lambda)$ is the free, outgoing resolvent -- when $\lambda$ is real, this outgoing condition becomes the usual Sommerfeld radiation condition and $R_0(\lambda)$ is $L^2$-bounded for $\Im(\lambda)>0$.
    
    The system \eqref{e:nProblem} was studied for $\lambda\in\mathbb{R}$ in \cite{CaMo:23} and is a special case of our more general setting below (see \eqref{e:mainProblem} and Remark \ref{special case}). We call $u_{\lambda}=u_{\inner}1_{\Omega_{\inner}}+u_{\exterior}1_{\Omega_{\exterior}}$ a resonant state for $\lambda$ and write $\mathcal{R}(n,\Omega_{\inner})$ for the set of resonances. Surface plasmons are a special type of long-lived resonant state that asymptotically concentrate at $\partial\Omega$. 
    \blue{In our next definition and below, we use the notation 
    $
    A=O(b^{\infty})
    $ if for all $N$ there is $C_N>0$ such that $|A|\leq C_N b^{ N}$ for $0<b<1$ and $A=O(B^{-\infty})$ if for all $N$ there is $C_N>0$ such that $|A|\leq C_NB^{-N}$ for $1<B<\infty$. }
\begin{defn}    We call a sequence of resonances $\{\lambda_j\}_{j=1}^\infty\subset\mathcal{R}$ with $|\lambda_j|\to \infty$ a \emph{surface plasmon} if $|\Im\lambda_j|=O(|\lambda_j|^{-\infty})$ and for any $\psi\in C_c^\infty(\mathbb{R}^d)$ with $\supp \psi\cap \partial\Omega=\emptyset$ and any sequence of resonant states $u_{\lambda_j}$, 
    \begin{equation}
    \label{e:plasmonic}
    \frac{\|\psi u_{\lambda_j}\|_{L^2(\mathbb{R}^d)}}{\|u_{\lambda_j}\|_{L^2(\partial\Omega)}}=O(\lambda_j^{-\infty}).
    \end{equation} 

    \end{defn}

    Throughout the text, we will assume that $\Omega_{\inner}$ is \emph{non-trapping} (See Figure~\ref{f:nt}). That is, all billiard trajectories (or more precisely generalized broken bicharacteristics; see e.g.~\cite[Section 24.3]{hor3} for a definition) escape any compact set in finite time. This condition guarantees that resonant states corresponding to propagating modes cannot approach the real axis. 

    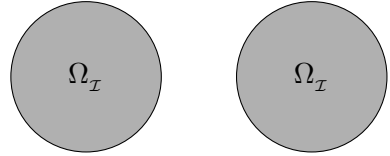
\begin{figure}[htbp]
\begin{tikzpicture}
\begin{scope}[scale=.5]
 \fill[fill=white] (-3,0) circle (2cm);
    \fill[fill=white] (3,0) circle (2cm);
  \draw[fill=light-gray]
      plot[domain=0:360, samples=500, variable=\t]
      ({cos(\t)*(3+sin(3*\t))}, {sin(\t)*(3+sin(3*\t))})
      -- cycle;
      \node at (0,0){$\Omega_{\inner}$};
      \node at (0,4){Non-trapping domain};
      \end{scope}
\end{tikzpicture}
\hspace{3cm}
\begin{tikzpicture}
\begin{scope}[scale=.5]
\fill[fill=white]
      plot[domain=0:360, samples=500, variable=\t]
      ({cos(\t)*(3+sin(3*\t))}, {sin(\t)*(3+sin(3*\t))})
      -- cycle;
  \draw[fill=light-gray] (-3,-1) circle (2cm);
  \node at(-3,-1){$\Omega_{\inner}$};
    \draw[fill=light-gray] (3,-1) circle (2cm);
      \node at(3,-1){$\Omega_{\inner}$};
      \node at (0,4){Trapping domain};
      \end{scope}
\end{tikzpicture}
\caption{\label{f:nt}Examples of trapping and non-trapping domains.}
\end{figure}
    
    We first determine conditions on the index of refraction such that there are no resonances close to the real axis.
    \begin{theorem}
    \label{t:simple}
     Suppose that $\Omega_{\inner}$ is non-trapping and $n\in C^\infty(\overline{\Omega}_{\inner};(0,\infty))$ satisfies $n|_{\partial\Omega}<1$. Then for all $M>0$ there is $C>0$ such that 
     $$
     \mathcal{R}(n,\Omega_{\inner})\cap \{ |\Re \lambda|>C\}\subset \{\Im \lambda<-M\}.
     $$
    \end{theorem}
    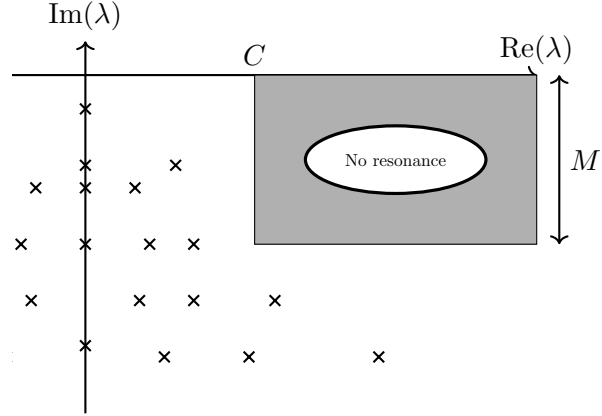
\begin{figure}[hbtp]
    \begin{center}
\begin{tikzpicture}
\begin{scope}[scale=1.5]
\clip (-.65,-3) rectangle (4.8,.9);
\fill[white] (-4.8,-.3) rectangle(4.8,.9);
	\filldraw[fill=light-gray] (1.5,-1.5) rectangle ++(2.5 ,1.5);
		\draw[->,thick] (-4, 0) -- (4, 0) node[above] {$\operatorname{Re}(\lambda)$};
	\draw[->,thick] (0, -3) -- (0, 0.3) node[above] {$\operatorname{Im}(\lambda)$};
	\draw(0,-1) node[solid,thick, cross out,draw=black, scale=0.5] {};
	\draw(0.8,-0.8) node[solid,thick, cross out,draw=black, scale=0.5] {};
	\draw(-0.8,-0.8) node[solid,thick, cross out,draw=black, scale=0.5] {};
	\foreach \x in {1,...,4} \draw(0,-0.1*\x^2-0.2*\x) node[solid,thick, cross out,draw=black, scale=0.5] {};
	\foreach \x in {2,...,5} \draw(0.1*\x^2+0.02*\x,-0.5*\x) node[solid,thick, cross out,draw=black, scale=0.5] {};
	\foreach \x in {2,...,5} \draw(-0.1*\x^2-0.02*\x,-0.5*\x) node[solid,thick, cross out,draw=black, scale=0.5] {};
	\foreach \x in {3,...,5} \draw(0.05*\x^2+0.04*\x,-0.5*\x) node[solid,thick, cross out,draw=black, scale=0.5] {};
	\foreach \x in {3,...,5} \draw(-0.05*\x^2-0.04*\x,-0.5*\x) node[solid,thick, cross out,draw=black, scale=0.5] {};
	\foreach \x in {4,...,5} \draw(0.02*\x^2+0.04*\x,-0.5*\x) node[solid, thick,cross out,draw=black, scale=0.5] {};
	\foreach \x in {4,...,5} \draw(-0.02*\x^2-0.04*\x,-0.5*\x) node[solid,thick, cross out,draw=black, scale=0.5] {};

    \draw node at (2.75,-.75){Resonance free};
	\draw (1.5,0) node[above] {$C$};
	\draw[<->,thick] (4.2, 0) -- (4.2, -1.5) node[midway, right] {$M$};
    \end{scope}
	\end{tikzpicture}
    \end{center}
    \caption{
\label{f:1}The figure shows the resonances for~\eqref{e:nProblem} with $n|_{\partial\Omega}<1$ as x's. The resonance free region is determined by Theorem~\ref{t:simple} or~\ref{t:noStates}.}
    \end{figure}
    
        Next, in the complementary case, we describe the region in which resonances may lie and show that any such resonances are surface plasmons.
    \begin{theorem}
    \label{t:simple2}
     Suppose that $\Omega_{\inner}$ is non-trapping and $n\in C^\infty(\overline{\Omega}_{\inner};(0,\infty))$ satisfies $n|_{\partial\Omega}>1$. Then for all $M>0$, $N>0$ there is $C$ such that 
     $$
     \mathcal{R}(n,\Omega_{\inner})\cap\{|\Re \lambda|>C\} \subset \{\Im \lambda<-M\text{ or }-|\Re\lambda|^{-N}<\Im \lambda<0\}.
     $$
     Moreover, any sequence $\{\lambda_j\}_{j=1}^\infty \subset \mathcal{R}(n,\Omega_{\inner})$ with $|\Re \lambda_j|\to \infty$ and $|\Im \lambda_j|$ bounded is a surface plasmon.
    \end{theorem}
    \begin{rem}
    We prove much more than that the resonant states with $|\Im\lambda_j|$ bounded are surface plasmons in the sense of~\eqref{e:plasmonic} (see Lemma~\ref{l:plasmonic} and Figure~\ref{f:whatIsAPlasmon}). We are, in fact, able to describe their localization properties modulo $|\lambda_j|^{-\infty}$. For instance, one can see that for any $\chi\in C_c^\infty(\mathbb{R}^d)$, $\|\chi \partial_x^\alpha u_{\lambda_j}(x)\|_{L^2(\mathbb{R}^d)}\leq C|\lambda_j|^{-\frac{1}{2}+|\alpha|}$, $|\alpha|\leq 2$. 
    \end{rem}
    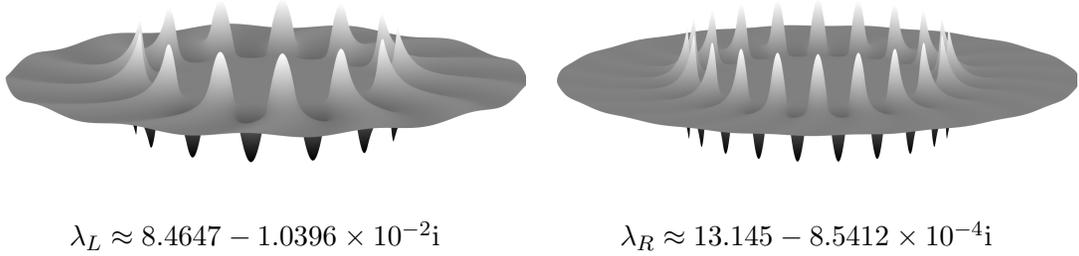
\begin{figure}
     \begin{tikzpicture}
    \clip (-3.6,-2.5)rectangle(3.6,1.5);
    \node at (0,0){\includegraphics[width=.55\textwidth]{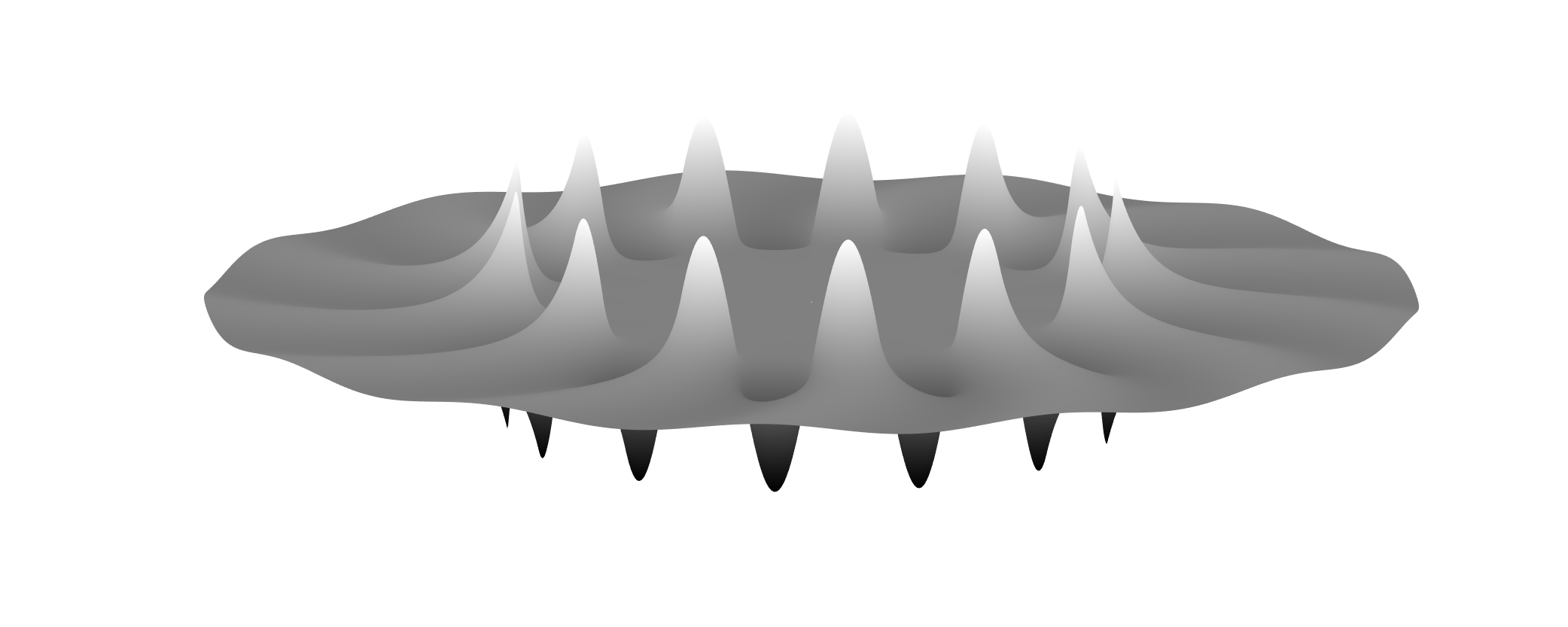}};
    \node at (0,-2){$\lambda_L \approx 8.4647-1.0396\times 10^{-2}\rmi$};
    \end{tikzpicture}
    \begin{tikzpicture}
    \clip (-3.6,-2.5)rectangle(3.6,1.5);
    \node at (0,0){\includegraphics[width=.55\textwidth]{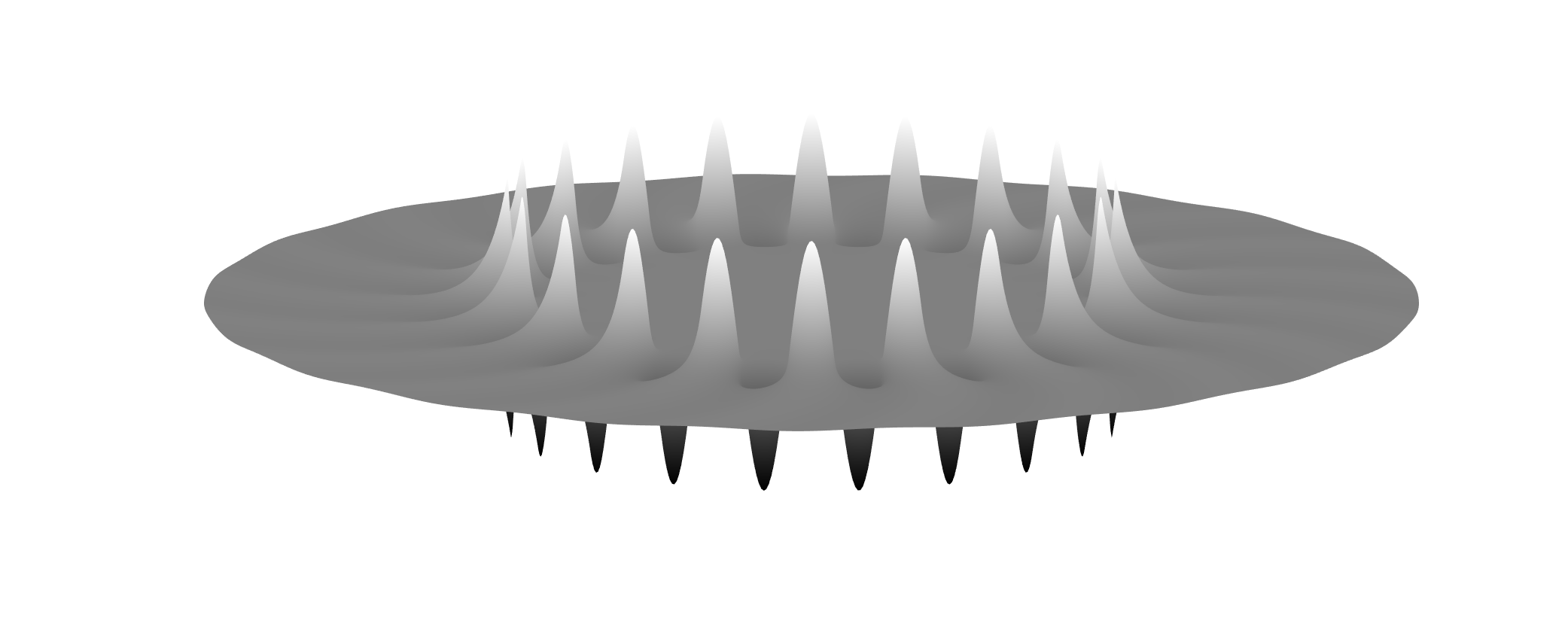}};
    \node at (0,-2){$\lambda_R\approx 13.145-8.5412\times 10^{-4}\rmi$};
    \end{tikzpicture}
    \caption{\label{f:whatIsAPlasmon} Lemma~\ref{l:plasmonic} in fact shows that, modulo $O(|\lambda_j|^{-\infty})$, all surface plasmons are as pictured here (with $\Omega_{\inner}=B(0,1)$). These plasmons concentrate in a $|\lambda_j|^{-1}$ neighborhood of the boundary, $\partial\Omega$ and oscillate at frequency $\sim |\lambda_j|$ in $\partial\Omega$. The functions plotted here are the real parts of the resonant state corresponding to $n|_{B(0,1)}\equiv 3$ with resonance $\lambda_L \approx 8.4647\times 10^0-1.0396\times 10^{-2}\rmi$ (on the left) and $\lambda_R\approx 13.145-8.5412\times 10^{-4}\rmi$ (on the right).}
    \end{figure}
\noindent The resonance free regions of Theorems~\ref{t:simple} and~\ref{t:simple2} are pictured in Figures~\ref{f:1} and~\ref{f:2} respectively.

    Finally, we determine the asymptotic number of surface plasmons, counted with multiplicity.
    \begin{theorem}
    \label{t:simple4}
     Suppose that $\Omega_{\inner}$ is non-trapping and $n\in C^\infty(\overline{\Omega}_{\inner};(0,\infty))$ satisfies $n|_{\partial\Omega}>1$. Then for all $M>0$,
     \begin{align*}
     \#\{ \lambda_j\in \mathcal{R}(n,\Omega_{\inner})\,:\, 0<\Re \lambda_j\leq \lambda,\,\Im \lambda \geq -M\}=\frac{\lambda^{d-1}}{(2\pi)^{d-1}}\vol_{T^*\partial\Omega}\left(\mathcal{V}\right)+o(\lambda^{d-1}),
     \end{align*}
     where 
     $$
     \mathcal{V}:=\left\{ (x',\xi')\in T^*\partial\Omega\,:\, |\xi'|^2_{g_{0}}\leq 1+\frac{1}{n(x')-1}\right\},
     $$
     with $g_{0}$ the metric induced on $\partial\Omega$ by the Euclidean metric on $\mathbb{R}^d$
    \end{theorem}    
     \begin{rem} By Theorem~\ref{t:simple2}, Theorem~\ref{t:simple4} provides an asymptotic formula for the number of surface plasmons when $n|_{\partial\Omega}>1$. Similarly, Theorem~\ref{t:count} provides an asymptotic formula for the number of surface plasmons. 
   \end{rem}
   \begin{rem} Note that Theorem \ref{t:simple} \ref{t:simple2}, and \ref{t:simple4} are special cases of Theorem~\ref{t:noStates}, Theorem~\ref{t:states}, Theorem~\ref{t:plasmonic}, and Theorem~\ref{t:count} respectively.
   \end{rem}

   Theorems~\ref{t:simple} to~\ref{t:simple4} (and their analogs below) give a precise description of resonances near the real axis for a wide class of negative index of refraction scattering problems. They determine when such resonances exist, how many there are, and describe the asymptotic properties of the corresponding resonant states. While it is often possible to obtain asymptotic upper bounds on the number of resonances near the real axis (see e.g.~\cite{SjZw:07,DaDy:13,DyGa:17,Dy:19} and references therein), it is very rare to be able to give an asymptotic count of these resonances -- celebrated examples include scattering in one dimension~\cite{Zw:87}, in elasticity~\cite{SjVo:97}, by convex obstacles~\cite{SjZw:99}, by convex transparent obstacles~\cite{CaPoVo:01}, and with normally hyperbolic trapping~\cite{Dy:15}. Theorem~\ref{t:simple4} and its more general analog Theorem~\ref{t:count} provide another such example.
    
    \noindent{\bf{Relation with previous work on negative index of refraction metamaterials:}}

    To the best of the authors' knowledge, the first mathematical paper considering negative index of refraction scattering is~\cite{COSTABEL1985367}, where the authors study materials with constant index of refraction (i.e. $n|_{\Omega_{\inner}}\equiv c_{\inner}$) and show that the problem~\eqref{e:nProblem} is Fredholm under certain conditions on $\lambda$ and $c_{\inner}$.  The works~\cite{BoChCi:12,BoChCl:13,BoChCi:14} build on this theory, allowing $n$ to be variable, and study the problem in lower regularity. In Appendix~\ref{a:blackBox}, we give a different proof inspired by~\cite{CdV:25} to show that~\eqref{e:nProblem} (or indeed the more general problem~\eqref{e:mainProblem}) is Fredholm when $n|_{\partial\Omega}$ avoids $1$.
    
     We study scattering resonances in the context of negative index of refraction metamaterials. As far as the authors are aware, the only previous works in this context are~\cite{CaMo:23,DaBaCaMo:24,CdV:25}. In~\cite{CaMo:23,DaBaCaMo:24}, the authors study the case of $d=2$ with $\Omega_{\inner}$ having smooth boundary and show, without further assumptions on the geometry, that there are many resonances near the real axis in the case $n|_{\partial\Omega}>1$ and many negative eigenvalues in the case $n|_{\partial\Omega}<1$. In fact, the authors construct a sequence of quasimodes, $u_j$, with quasi-eigenvalue $\lambda_j\to \infty$ associated to~\eqref{e:nProblem} so that $u_j$ are highly localised near $\partial\Omega$. However, they do not show that the true resonant states are highly localized. 
    When $n|_{\partial\Omega}>1$ is constant,
    $$
    \lambda_j=\frac{2\pi j(n-1)}{n|\partial\Omega|},
    $$
    and, if $\Omega_{\inner}$ is non-trapping, we confirm from Theorem~\ref{t:simple4} that this sequence of quasi-eigenvalues captures most resonances near the real axis. Moreover, Theorem~\ref{t:simple2} shows that all corresponding resonant states are highly localized. Very recently, and independently from our work,~\cite{CdV:25} considers the higher dimensional analog of~\cite{CaMo:23}.

    The present article differs from and strengthens these earlier works in two substantial ways. First, under an additional natural dynamical assumption, we describe the location of all possible resonances near the real axis and show that if there are any, they must correspond to highly localized resonant states, and second, we determine how many such resonances there are. 

\begin{rem}There are variety transmission problems, including by positive index of refraction materials, which are much better developed in the mathematical literature (See e.g.~\cite{PoVo:99,MR1724834,CaPoVo:99,CaPoVo:01,Ga:19b,Ga:19,MoSp:19}).
\end{rem}

\begin{rem}
The adjective plasmon -- meaning a state that is highly localized to an interface -- has appeared in other wave-type contexts in mathematical literature; see e.g.,~\cite{Gr:14,AmRuYuZh:16,AmMiRuZh:17,LiZh:19}, where plasmons are studied in the context of scattering by obstacles that are much smaller than the wavelength.  
\end{rem}

\blue{
\noindent{\bf Comments on other mathematical and physical models:} \
The model used in this article for a negative index of refraction metamaterial deserves a few comments.
\begin{enumerate}
    \item As with most materials, the refraction index of negative index of refraction metamaterials depends on the frequency of the incident field. Specifically, these materials exhibit negative index of refraction behavior in a specific range of frequencies rather than for all frequencies simultaneously~\cite{SSS,Valentine2008ThreeDimensional}. The model  considered here simplifies the situation by assuming that the material has a negative index for all frequencies. The question then becomes one of scale: is the wavelength at which the material exhibits negative index small compared to the size of the obstacle or not? In many cases, these phenomena occur at a low number of wavelengths in the obstacle and require the use of the quasi-static regime \cite{MR3837172}.  However, in some situations it is possible to construct materials that exhibit a negative index at frequencies large relative to the size of the object~\cite{SSS}.
    \item Because of energy losses in propagation, true materials have index of refraction with non-trivial, albeit small, imaginary part. Although this would likely not require significant modification of the analysis here, we have chosen not to pursue this generalization in order to keep the model as simple as possible.
    \item The mathematical foundation for constructing materials with effective negative index of refraction is usually studied in the context of homogenization theory, which relies heavily on the notion of two-scale convergence developed by Nguetseng \cite{Nguetseng1989,LukkassenNguetsengWall2002,Nguetseng2003,Nguetseng2004}. Since the focus of this paper is on scattering by negative index metamaterials rather than the homogenization theory used in their construction, we refer the reader to \cite{MR3667288,Davies_2025,CKS:25} for further details.
    \item This paper focuses on obstacles with non-trapping geometry. As a result, all resonances close to the real axis are caused by the jump in the index of refraction rather than propagating modes outside the obstacle. Indeed, in trapping situations there may be resonances close to the real axis whose corresponding resonant states concentrate near trapped billiard trajectories (see e.g.~\cite{MR1932033}). The study of plasmons with trapping geometries in the quasistatic regime can be found in \cite{MR3797722,YuAmmari2019}. One of physically interesting cases is then when two or more symmetric obstacles are close to each other. This is the case of plasmon hybridization~\cite{hybridization,YuAmmari2019}, where a resonant state arises that concentrates in the narrow gap between the obstacles \cite{MR3824167}. 
\end{enumerate}
}
 
\begin{figure}
\begin{center}
\begin{tikzpicture}
\begin{scope}[scale=1.5]
\clip (-.65,-3) rectangle (4.8,.9);
\fill[white] (-4.8,-.3) rectangle(4.8,.9);
	\def\s{.8};
	\def\m{.65};
	\draw[->,thick] (-4, 0) -- (4, 0) node[above] {$\operatorname{Re}(\lambda)$};
	\draw[->,thick] (0, -3) -- (0, 0.3) node[above] {$\operatorname{Im}(\lambda)$};
	\draw(0,-1) node[solid,thick, cross out,draw=black, scale=0.5] {};
	\draw(0.8,-0.8) node[solid, thick,cross out,draw=black, scale=0.5] {};
	\draw(-0.8,-0.8) node[solid,thick, cross out,draw=black, scale=0.5] {};
	\foreach \x in {1,...,4} \draw(0,-0.1*\x^2-0.2*\x) node[solid,thick, cross out,draw=black, scale=0.5] {};
	\foreach \x in {2,...,5} \draw(0.1*\x^2+0.02*\x,-0.5*\x) node[solid, thick,cross out,draw=black, scale=0.5] {};
	\foreach \x in {2,...,5} \draw(-0.1*\x^2-0.02*\x,-0.5*\x) node[solid, thick,cross out,draw=black, scale=0.5] {};
	\foreach \x in {3,...,5} \draw(0.05*\x^2+0.04*\x,-0.5*\x) node[solid, thick,cross out,draw=black, scale=0.5] {};
	\foreach \x in {3,...,5} \draw(-0.05*\x^2-0.04*\x,-0.5*\x) node[solid, thick,cross out,draw=black, scale=0.5] {};
	\foreach \x in {4,...,5} \draw(0.02*\x^2+0.04*\x,-0.5*\x) node[solid, thick,cross out,draw=black, scale=0.5] {};
	\foreach \x in {4,...,5} \draw(-0.02*\x^2-0.04*\x,-0.5*\x) node[solid,thick, cross out,draw=black, scale=0.5] {};

	\filldraw[fill=light-gray] (1.5,-1.5) -- plot [smooth ] coordinates {(1.5,-\s) (2.125,-\s*\m) (2.75,-\s*\m*\m) (3.375,-\s*\m*\m*\m) (4,-\s*\m*\m*\m*\m )} --(4,-1.5)--cycle;
    \draw node at(2.75,-1.05){Resonance free};
	\draw (1.5,0) node[above] {$C$};
	\draw[<->,thick] (4.2, 0) -- (4.2, -1.5) node[midway, right] {$M$};
	\draw[dashed] (1.5,0.05) --(1.5,-0.6);
	\foreach \x in {1,...,7} \draw(0.5*\x,{-2*exp(.5*(-\x-0.8))-.02}) node[solid, thick,circle ,draw=black, scale=0.3] {};
	\foreach \x in {1,...,7} \draw(-0.5*\x,{-2*exp(.5*(-\x-0.8))-.02}) node[solid, thick,circle ,draw=black, scale=0.3] {};
    \end{scope}
\end{tikzpicture}
\end{center}
    \caption{\label{f:2}The figure shows the resonances for~\eqref{e:nProblem} with $n|_{\partial\Omega}>1$. Resonances that are not surface plasmons are denoted with x's and surface plasmons are indicated with o's. The resonance free regions are those determined by Theorem~\ref{t:simple2} or~\ref{t:states}. Theorems~\ref{t:simple4} and~\ref{t:count} determine the asymptotic number of surface plasmons.}
\end{figure}
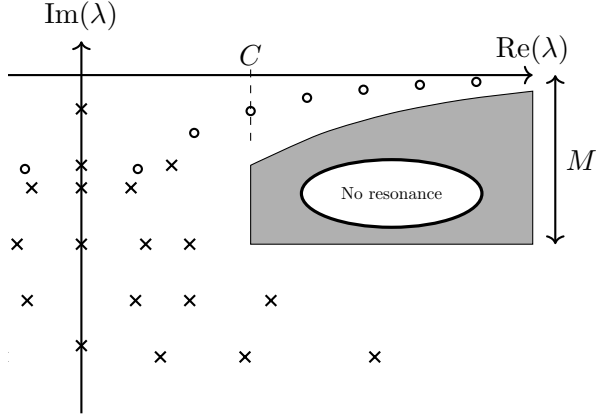

\subsection{Inhomogeneous metrics}
\label{s:general}
In this article, we study the general situation of a cavity with a negative definite Laplacian. This corresponds to a metamaterial where the material properties are not homogoneous and exhibit an effective negative wave-speed. To this end, let $g_{\exterior}$ be a smooth Riemannian metric on $\overline{\Omega}_{\exterior}$ and $g_{\inner}$ a smooth Riemannian metric on $\overline{\Omega_{\inner}}$. Let also $\rho_{\inner}\in C^\infty(\overline{\Omega_{\inner}};(0,\infty))$ and $\rho_{\exterior}\in C^\infty(\overline{\Omega}_{\exterior};(0,\infty))$. We assume that the geometry is Euclidean near infinity. That is
    \begin{equation*}
    \label{e:euclidenn}
    g_{\exterior}^{ij}(x)= \delta^{ij},\qquad \rho_{\exterior}(x)=1,\qquad \text{for }|x|\gg 1.
    \end{equation*}

    For a metric $g$ and positive function $\rho$, we define the operator 
    $$
    \Delta_{g,\rho}u:=\frac{1}{\rho\sqrt{|g|}}\partial_{x^i}(g^{ij}\sqrt{|g|}\rho\partial_j u(x)), \quad |g|:=|\det g_{ij}|,
    $$
    and note that $\Delta_{g,\rho}$ is symmetric on $L^2(\rho \der \vol_g)$. We then define the unbounded operator $P:L^2(\Omega_{\inner},\rho_{\inner} \der\vol_{g_{\inner}})\oplus L^2(\Omega_{\exterior},\rho_{\exterior}\der \vol_{g_{\exterior}})\to  L^2(\Omega_{\inner},\rho_{\inner}\der\vol_{g_{\inner}})\oplus L^2(\Omega_{\exterior},\rho_{\exterior}\der\vol_{g_{\exterior}})$ given by 
    $$
    P(u_{\inner},u_{\exterior}):=(\Delta_{g_{\inner},\rho_{\inner}},-\Delta_{g_{\exterior},\rho_{\exterior}})
    $$
    with domain 
    \begin{equation}
    \label{e:domain1}
    \mathcal{D}(P):=\left\{\begin{aligned}u=(u_{\inner},u_{\exterior})\in H^1(\mathbb{R}^d)&\cap \left(H^2(\Omega_{\inner})\oplus H^2(\Omega_{\exterior})\right)\,:\\
    &\rho_{\exterior}\partial_{\nu_{\exterior}}u_{\exterior}\der \vol_{g_{\exterior},\partial\Omega}=\rho_{\inner}\partial_{\nu_{\inner}}\der \vol_{g_{\inner},\partial\Omega}\end{aligned}\right\},
    \end{equation}
    where $\partial_{\nu_{\inner}}$ and $\partial_{\nu_{\exterior}}$ denote respectively the outward unit normal to $\Omega_{\inner}$ with respect to $g_{\inner}$ and the outward unit normal to $\Omega_{\exterior}$ with respect to $g_{\exterior}$. 

    The induced boundary volume forms from the exterior and interior metrics are not in general the same, i.e. $\der \vol_{g_{\exterior}}$ is not necessarily the same as $\der \vol_{g_{\inner}}$. It is thus natural to define a function on $\partial \Omega$ to measure their difference. Hence, we define $\tau\in C^\infty(\partial\Omega)$ 
    \begin{equation*}
    \label{rel density}
    \tau\der \vol_{g_\exterior,\partial \Omega}:=\der \vol_{g_\inner,\partial \Omega}.
    \end{equation*}
    Notice that $\tau$ is positive on $\partial \Omega$.
    
    We assume throughout the text that 
    \begin{equation}
    \label{e:thereIsAJump}
    \tau^2 \rho_{\inner}^2|\xi'|_{g_{\inner}}^2-\rho_{\exterior}^2|\xi'|^2_{g_{\exterior}}\neq 0,\qquad (x,\xi')\in T^*\partial\Omega,
    \end{equation}
    \blue{which generalizes the condition for the refraction index $n\neq 1$ in \eqref{e:nProblem}.} Under assumption~\eqref{e:thereIsAJump}, the operator $P$ is self-adjoint and is a black box Hamiltonian in the sense of~\cite[Section 4.1]{DyZw:19} (See Appendix~\ref{a:blackBox}).

    We then let $R_P(\lambda):=(P-\lambda^2)^{-1}:L^2(\Omega_{\inner})\oplus L^2(\Omega_{\exterior})\to  L^2(\Omega_{\inner})\oplus L^2(\Omega_{\exterior})$. By~\cite[Theorem 4.4]{DyZw:19}, $R_P(\lambda):L^2(\Omega_{\inner})\oplus L^2(\Omega_{\exterior})\to \mathcal{D}(P)$ is meromorphic in $\Im \lambda>0$ and has a meromorphic continuation to $\mathbb{C}$ for $d$ odd and the logarithmic cover of $\mathbb{C}\setminus\{0\}$ for $d$ even as an operator $R_{P}(\lambda):L^2(\Omega_{\inner})\oplus L^2_{\comp}(\Omega_{\exterior})\to \mathcal{D}_{\loc}(P)$, where 
    $$
    \mathcal{D}_{\loc}(P):=\left\{u=(u_{\inner},u_{\exterior})\in H_{\loc}^1(\mathbb{R}^d)\cap\left( H^2(\Omega_{\inner})\oplus H_{\loc}^2(\Omega_{\exterior})\right)\,:\,\rho_{\exterior}\partial_{\nu_{\exterior}}u_{\exterior}=\tau\rho_{\inner}\partial_{\nu_{\inner}}u_{\inner}\right\}.
    $$

    Defining  $(u_{\inner},u_{\exterior}):=R_P(\lambda)(f_{\inner},f_{\exterior})$, $(u_{\inner},u_{\exterior})$ satisfies 
    \begin{equation}
    \label{e:mainProblem}
    \begin{cases}(\Delta_{g_{\inner},\rho_{\inner}}-\lambda^2)u_{\inner}=f_{\inner}&\text{ in }\Omega_{\inner},\\
    (-\Delta_{g_{\exterior},\rho_{\exterior}}-\lambda^2)u_{\exterior}=f_{\exterior}&\text{ in }\Omega_{\exterior},\\
    u_\inner=u_{\exterior}&\text{ on $\partial\Omega$},\\
    \rho_{\exterior}\partial_{\nu_{\exterior}}u_{\exterior}-\tau\rho_{\inner}\partial_{\nu_{\inner}}u_{\inner}=0&\text{ on $\partial\Omega$},\\
    u_{\exterior}\text{ is $\lambda$-outgoing}.
    \end{cases}
    \end{equation}
    Define the set of \emph{resonances of $P$} by
    $$
    \mathcal{R}(P):=\{\lambda\,:\, \lambda\text{ is a pole of $R_P(\lambda)$}\}.
    $$
    \begin{rem}
    \label{special case}
    Note that \eqref{e:nProblem} is a special case of our general setting \eqref{e:mainProblem} with $g_\exterior^{ij}=\delta^{ij}$, $g_\inner^{ij}=n^{-1}\delta^{ij}$, $\rho_\exterior=1$, $\rho_\inner=n^{-\frac{d}{2}}$, and $\tau=n^{\frac{d-1}{2}}$.
    \end{rem}
    
    We now state the analogs of Theorems~\ref{t:simple} to~\ref{t:simple4} in the more general setting of a negative wave speed. We begin in the case
    \begin{equation}
    \label{e:noPlasmons}
    0>\rho_{\exterior}^2|\xi'|^2_{g_{\exterior}}-\tau^2\rho_{\inner}^2|\xi'|_{g_{\inner}}^2, \qquad (x',\xi')\in T^*\partial\Omega,
    \end{equation}
    where we show that there are no resonances close to the real axis.
	\begin{theorem}
		\label{t:noStates}
		Suppose that $(\Omega_{\exterior},g_{\exterior})$ is a non-trapping domain and~\eqref{e:noPlasmons} holds. Then for all $M>0$, there is $C>0$ such that 
		$$
		\calR(P)\cap \{ |\Re \lambda|>C\}\subset \{\Im \lambda \leq -M\}. 
		$$
        Moreover, for $\chi\in C_c^\infty(\mathbb{R}^d)$,
        $$
        \|\chi R_P(\lambda)\chi\|_{L^2(\R^d)\to L^2(\R^d)}\leq C|\lambda|^{-1},\qquad \lambda\in \{ \Re \lambda>C,\, \Im \lambda>-M\}.
        $$
	\end{theorem}
	
	Our next three theorems consider the opposite case: \begin{equation}
    \label{e:plasmons}
    0<\rho_{\exterior}^2|\xi'|^2_{g_{\exterior}}-\tau^2\rho_{\inner}^2|\xi'|_{g_{\inner}}^2, \qquad (x',\xi')\in T^*\partial\Omega.
    \end{equation}
    The combination of the next three theorems shows that there are many resonances resonances superpolynomially close to the real axis, all of which are surface plasmons and, moreover, any sequence of resonances that is not superpolynomially close to the real axis must have imaginary part whose absolute value tends to infinity.

The first theorem provides a resonances free region.
	\begin{theorem}
		\label{t:states}
		Suppose that $(\Omega_{\exterior},g_{\exterior})$ is a non-trapping domain and~\eqref{e:plasmons} holds. Then for all $M>0$, $N>0$ there is $C>0$ such that 
		$$
		\mathcal{R}(P)\cap \{ |\Re \lambda|>C\}\subset \{\Im \lambda \leq -M\}\cup \{-|\lambda|^{-N}<\Im \lambda<0\} . 
		$$
        Moreover, 
        $$
        \|\chi R_P(\lambda)\chi\|_{L^2(\R^d)\to L^2(\R^d)}\leq C|\Im \lambda|^{-1}|\lambda|^{-1},\qquad \lambda\in \{\Re \lambda>C,\, -M<\Im \lambda<(\Re \lambda)^{-N}\}.
        $$
\end{theorem}
Next, we show that any resonances with bounded imaginary parts are necessarily surface plasmons.
\begin{theorem}
\label{t:plasmonic}
Suppose that $(\Omega_{\exterior},g_{\exterior})$ is a non-trapping domain and~\eqref{e:plasmons} holds. Then for any $\{\lambda_j\}_{j=1}^\infty\subset \calR(P)$ with $|\Re \lambda_j|\to \infty$ and $\sup|\Im \lambda_j |<\infty$, any $0\neq u_{\lambda_j}\in\mathcal{D}_{\loc}(P)$ satisfying~\eqref{e:mainProblem} with $(f_{\inner},f_{\exterior})=0$, and any $\psi\in C_c^\infty(\mathbb{R}^d)$ with $\supp \psi\cap \partial\Omega=\emptyset$, we have
$$
\frac{\|\psi u_{\lambda_j}\|_{L^2(\mathbb{R}^d)}}{\|u_{\lambda_j}\|_{L^2(\partial\Omega)}}=O(\lambda_j^{-\infty}).
$$
In particular, combining this with Theorem~\ref{t:states} shows that all such sequences of resonances are surface plasmons.
\end{theorem}

Finally, we give an asymptotic formula for the number of surface plasmons.
\begin{theorem}
\label{t:count}
Suppose that $(\Omega_{\exterior},g_{\exterior})$ is a non-trapping domain and~\eqref{e:plasmons} holds. Then for all $M>0$, 
$$
\#\{ \lambda_j\in \calR(P)\,:\, 0< \Re \lambda_j \leq \lambda\,:\Im \lambda_j\geq -M\}= \frac{\lambda^{d-1}}{(2\pi )^{d-1}}\vol_{T^*\partial\Omega}(\mathcal{V}) +o(\lambda^{d-1}),
$$
where
$$
\mathcal{V}:=\left\{ (x',\xi')\in T^*\partial\Omega\,:\rho^2_{\exterior}(x')|\xi'|_{g_{\exterior}}^2-\rho^2_{\inner}(x')^2\tau^2(x')|\xi'|_{g_{\inner}}^2 \leq\rho_{\exterior}(x')^2+\rho_{\inner}^2(x')\tau^2(x') \right\}.
$$
\end{theorem}
	

\subsection{Outline and ideas from the proof}

We start by making a semiclassical rescaling, setting \blue{$\lambda =h^{-1}z$} and $0<h<1$. The goal of this article can then be rephrased in the following way. 

Let $\chi \in C_c^\infty(\mathbb{R}^d)$ with $\chi \equiv 1$ on $\Omega_{\inner}$, $L^2(\mathbb{R}^d)\ni f=(f_\inner,f_\exterior )\in L^2(\Omega_{\inner})\oplus L^2(\Omega_{\exterior})$. Our main goal will be to prove estimates for the solution, $(u_\inner,u_\exterior )$ to 
	\begin{equation}
		\label{e:transmission1}
		\begin{cases}
			(P_{\inner}-z^2)u_{\inner}:=(h^2\Delta_{g_\inner,\rho_{\inner}} -z^2)u_\inner=h f_\inner&\text{in }\Omega_{\inner},\\
		(P_{\exterior}-z^2)u_{\exterior}:=(-h^2\Delta_{g_{\exterior},\rho_{\exterior}}-z^2)u_\exterior =h\chi f_\exterior &\text{in }\Omega_{\exterior},\\
			u_\exterior =u_\inner&\text{on }\partial\Omega,\\			\rho_{\exterior}h\partial_{\nu_{\exterior}}u_\exterior -\tau\rho_{\inner}h\partial_{\nu_{\inner}}u_{\inner}=0&\text{on }\partial\Omega,\\
            u_\exterior \text{ is }z/h\text{-outgoing}.
		\end{cases}
	\end{equation}

In section~\ref{Reformulation of the problem as an exterior problem}, using the solution of the Dirichlet problem in $\Omega_{\inner}$ and the outgoing Dirichlet problem in $\Omega_{\exterior}$, we reduce these estimates to the study of 
	\begin{equation}
		\label{e:transmission2Temp}
		\begin{cases}
			(P_{\exterior}-z^2)v_\exterior =0&\text{in }\Omega_{\exterior},\\
			\rho_{\exterior}h\partial_{\nu_{\exterior}}v_\exterior -\tau\Lambda_{\inner}(z)v_{\exterior} =:g\in H_h^{\frac{1}{2}}&\text{on }\partial\Omega,\\
            v_\exterior \text{ is }z/h\text{-outgoing},
		\end{cases}
	\end{equation}
    where $\Lambda_{\inner}(z)$ is a certain Dirichlet-to-Neumann map associated with the inner problem. More precisely, we define by $\Lambda_{\inner/\exterior}(z)w=\rho_{\inner/\exterior}h\partial_{\nu_{\inner/\exterior}}u_{\inner/\exterior}$, where $u_{\inner/\exterior}\in H^2_{\loc}(\overline{\Omega}_{\inner/\exterior})$ solves
    $$
    \begin{cases}
    (P_{\inner/\exterior}-z^2)u_{\inner/\exterior}=0&\text{in }\Omega_{\inner/\exterior},\\
    u_{\inner/\exterior}=w&\text{on }\partial\Omega,\\
    u_{\exterior}\text{ is }z/h\text{ outgoing}.
    \end{cases}
    $$

    After the reduction to~\eqref{e:transmission2Temp}, it becomes natural to study the Dirichlet-to-Neumann map for operators of the form
    $$
    P(\omega;g,L):=-h^2\Delta_g+hL-\omega,
    $$
    where $L=\sum_{i=1}^dL^i(x)hD_{x^i}$, and $|\omega-\omega_0|\leq Ch$. \blue{As we will often take $\omega=z^2$ or $\omega=-z^2$ in later analysis, it is sufficient to consider the case when $|z-z_0|<Ch$.}

    Our next theorem yields a parametrix for the Dirichlet-to-Neumann map in the elliptic region.
    \begin{theorem}
    \label{t:dtN}
    Let $\omega_0\in \mathbb{R}$, $L$ be a smooth vector field, and $(M,g)$ a Riemannian manifold with boundary. Then for all $\e>0$, $C>0$, and $|\omega-\omega_0|\leq Ch$ there is $E_\omega\in \Psi^1_h(\partial\Omega)$ with $\sigma(E_\omega)=\sqrt{|\xi'|^2_g-\omega_0}$ such that for any $s\geq \frac{1}{2}$, $X\in \Psi^0_h(\partial M)$ with $\WF(X)\subset \{|\xi'|_g^2>\omega_0\}$, $\delta>0$, $N>0$,  there is $C_1>0$ such that for all $0<h<1$,
    $$
    \|X(h\partial_{\nu_g}u-E_\omega (u|_{\partial M}))\|_{H_h^s(\partial M)}\leq C_1\left(h^{-\frac{1}{2}}\|P(\omega;g,L)u\|_{H_h^{s-\frac{1}{2}}(\partial M_\delta)}+h^N\|u\|_{H_h^1(\partial M_\delta)}\right),
    $$
    where
    $$
    \partial M_\delta:=\{ x\in M\,:\, d(x,\partial M)<\delta\}.
    $$
    \end{theorem}

    Notice that \blue{if $z_0\in \R$ and $|z-z_0|\leq Ch$,} Theorem~\ref{t:dtN} implies that $\Lambda_\inner(z)\in \Psi^1_h(\partial\Omega)$ with principal symbol 
    \begin{equation}
    \label{e:inDtN}
    \sigma(\Lambda_{\inner}(z))=\rho_{\inner}\sqrt{|\xi'|^2_{g_{\inner}}+z_0^2}
    \end{equation}
    and, moreover, for $X\in \Psi^0_h(\partial\Omega)$ with $\WF(X)\subset \{|\xi'|_{g_{\exterior}}>|z_0|\}$, $X\Lambda_{\exterior}(z)\in \Psi^1_h(\partial\Omega)$ with \begin{equation}
    \label{e:outDtN}
    \sigma(X\Lambda_{\exterior}(z))=\rho_\inner\sigma(X)\sqrt{|\xi'|^2_{g_{\exterior}}-z_0^2}.
    \end{equation}

    The distinction between~\eqref{e:noPlasmons} and~\eqref{e:plasmons} can be seen from~\eqref{e:inDtN} and~\eqref{e:outDtN}. Indeed, 
    $$
    \sigma( \blue{X}\Lambda_{\exterior} -\tau \Lambda_{\inner})=\rho_\inner\sigma(X)\sqrt{|\xi'|^2_{g_{\exterior}}-z_0^2}-\tau \rho_{\inner}\sqrt{|\xi'|^2_{g_{\inner}}+z_0^2},\qquad |\xi'|_{g_{\exterior}}>|z_0|,
    $$
    and this symbol does not vanish in the case of~\eqref{e:noPlasmons}, while it does in the case of~\eqref{e:plasmons}. One can also see why the existence of $(x',\xi')\in T^*\partial\Omega$ such that 
    $$
    \tau^2\rho_{\inner}^2|\xi'|^2_{g_{\inner}}=\rho_{\exterior}^2|\xi'|_{g_{\exterior}}^2
    $$
    may cause problems with self-adjointness. Indeed, in this case, the symbol is not uniformly elliptic as $|\xi'|\to \infty$ and hence, standard elliptic regularity results will fail.

In the case of~\eqref{e:noPlasmons}, the knowledge of the symbol of $\Lambda_{\inner}$ and that of $\Lambda_{\exterior}$ at high frequency is sufficient. However, in the case of~\eqref{e:plasmons} we need one more subtle piece of information about the Dirichlet-to-Neumann maps.
\begin{theorem}
\label{t:imagPart}
Let $M>0$. Then there is $c>0$ such that for all $0<h<1$, $|z-1|\leq Mh$, and $u\in H_h^{\frac{1}{2}}(\partial\Omega)$,
$$
-\operatorname{sgn}(\Im z^2)\Im \langle \tau \Lambda_{\inner}u,u\rangle_{L^2(\partial\Omega,\der\vol_{g_{\exterior}})} \leq -c|\Im z^2|\|u\|_{L^2(\partial\Omega)}^2.
$$
Furthermore, for all $X\in \Psi^{\comp}_h$ with $\WF(X)\subset\{ |\xi'|_{g_{\exterior}}>1\}$, and $N>0$, there are $c>0$ and $C_N>0$ such that for all $0<h<1$, $|z-1|<Mh$, and $u\in L^2(\partial\Omega)$, we have
$$
\operatorname{sgn}(\Im z^2)\Im \langle  \Lambda_{\exterior}Xu,Xu\rangle_{L^2(\partial\Omega,\der \vol_{g_{\exterior}})} \leq -\left(c|\Im z^2|-C_Nh^N\right)\|u\|_{L^2(\partial\Omega)}^2.
$$
\end{theorem}
Provided that $|\Im z|>h^N$ for some $N$, Theorem~\ref{t:imagPart} allows us to obtain estimates where $\Lambda_{\exterior}-\tau\Lambda_{\inner}$ fails to be elliptic. 

In order to finish the proofs of Theorem~\ref{t:noStates} and~\ref{t:states}, we need to obtain estimates on $|\xi'|_{g_{\exterior}}\leq 1$. For this, we employ defect measure arguments similar to those in~\cite{Bu:02,GaSpWu:20,GaMaSp:21}.

The proof of Theorem~\ref{t:count} relies on Theorem~\ref{t:states} and a contour integration. Let $V(h):=[1-2\e,1+2\e]\times \rmi [-h,h]$. First, using a complex absorbing potential to reduce to operators of trace class, one can find a compactly microlocalized pseudodifferential operator, $X$, with $\WF(X)\subset \{|\xi'|_{g_{\exterior}}>1+2\e\}$ such that 
$$
\sum_{\lambda_j\in V(h)}\tilde{\psi}(\lambda_j)=\frac{1}{2\pi i} \int_{\partial V(h)}\tilde{\psi}(z)(\Lambda_{\exterior}(z)-\tau\Lambda_{\inner}(z))^{-1}\partial_z (\Lambda_{\exterior}(z)-\Lambda_{\inner}(z))X\der z+O(h^\infty),
$$
where $\psi\in C_c^\infty((1-2\e,1+2\e))$ and $\tilde{\psi}$ is an almost analytic extension of $\psi$.

Then, since on $\WF(X)$, $\Lambda_{\exterior}-\tau\Lambda_{\inner}$ is a pseudodifferential operator with principal symbol
$$
\rho_{\exterior}\sqrt{|\xi'|_{g_{\exterior}}^2-z^2}-\tau \rho_{\inner}\sqrt{|\xi'|_{g_{\inner}}+z^2},
$$
we will argue in this sketch as though the whole operator was such a pseudodifferential operator. In particular, we can find $E$ an elliptic pseudodifferential operator such that 
$$
\Lambda_{\exterior}-\tau\Lambda_{\inner}= E(B(z)-z^2),
$$
where $B(z)\in \Psi^2_h$ with 
$$
B(z)=B_0+hB_1(z),\qquad B_i\in \Psi^{2-i}_h,
$$
and 
$$
\sigma(B_0)=\frac{\rho_{\exterior}^2|\xi'|_{g_{\exterior}}^2-\rho_{\inner}^2\tau^2|\xi'|_{g_{\inner}}^2}{\rho_{\exterior}^2+\rho_{\inner}^2\tau^2}.
$$

Next, we closely follow the proof of the Weyl law for self-adjoint pseudodifferential operators. In particular, for $\mp \Im z>0$, we write
$$
(B-z^2)^{-1}=\frac{i}{h}\int_0^{\pm \infty} U(t)e^{\frac{i}{h}tz^2}\der t,
$$
where
$$
(hD_t-B(z))U(t)=0,\qquad U(0)=I.
$$
Then, 
$$
\sum_{\lambda_j\in V(h)}\tilde{\psi}(\lambda_j)=\sum_{\pm}\pm\frac{1}{2\pi h} \int_0^{\pm\infty}\int_{\partial V(h)}\tilde{\psi}(z)U(t)e^{-\frac{i}{h}tz^2}E^{-1}\partial_z (\Lambda_{\exterior}(z)-\Lambda_{\inner}(z))W\der z+O(h^\infty),
$$ 
and, integrating by parts in $z$, we are able to replace the integral to time infinity by a finite integral; i.e. for $\chi\in C_c^\infty(-1,1)$ with $\chi \equiv 1$ near $0$, we have
$$
\sum_{\lambda_j\in V(h)}\tilde{\psi}(\lambda_j)=\frac{1}{2\pi h} \int_{-\infty}^\infty\int_{\partial V(h)}\tilde{\psi}(z)\chi(t)U(t)e^{-\frac{i}{h}tz^2}E^{-1}\partial_z (\Lambda_{\exterior}(z)-\Lambda_{\inner}(z))W\der z+O(h^\infty).
$$ 
At this point we can use an oscillatory integral approximation of $U(t)$ to compute the integrals and then approximate $1_{[1-\e,1+\e]}$ by cutoff functions $\psi$, thereby finishing the proof of theorem.

\subsection{Structure of the paper}

Section~\ref{s:preliminaries} contains a review of some preliminary material including basic notation for semiclassical operators, and defect measures as well as propagation of defect measure results. In Section~\ref{Reformulation of the problem as an exterior problem} we reformulate the problem as a scattering problem in the exterior of the obstacle $\Omega_{\inner}$ with a non-standard boundary condition. Next, in Section~\ref{Microlocal description of the Dirichlet-to-Neumann maps}, we prove Theorem~\ref{t:dtN} by implementing a factorization scheme for the Laplace-Beltrami operator near the boundary. We apply these methods specifically to $P_{\exterior}-z^2$ and $P_{\inner}-z^2$ and prove Theorem~\ref{t:imagPart} in Sections~\ref{LU exterior problem} and~\ref{LU interior problem}. In Section~\ref{s:proofOfTheorems}, we prove Theorems~\ref{t:noStates},~\ref{t:states}, and~\ref{t:plasmonic}. Finally, in Section~\ref{s:counting}, we prove Theorem~\ref{t:count}. Appendix~\ref{a:blackBox} shows that $P$ is a black-box Hamiltonian.

\bigskip

\noindent {\bf{Acknowledgements:}} The authors would like to thank Zo\"{i}s Moitier and Monique Dauge for introducing them to this problem and Yves Colin de Verdi\`{e}re for interesting discussion around~\cite{CdV:25}. Thanks also to Janosch Preuss for help with numerical experimentation \blue{and to the anonymous referees for their comments and careful reading which improved the exposition}. JG was supported by ERC Synergy Grant: PSINumScat - 101167139, Leverhulme Research Project Grant: RPG-2023-325, EPSRC Early Career Fellowship: EP/V001760/1, and EPSRC Standard Grant: EP/V051636/1. YF was supported by EPSRC Standard Grant: EP/V051636/1.

\section{Preliminaries}
\label{s:preliminaries}
\subsection{Semiclassical rescaling and pseudodifferential operators}
\label{Semiclassical rescaling and pseudodifferential operators}	
In order to prove our estimates, we reformulate our problem in semiclassical language; i.e. let $0<h<1$, $z=z(h)\in\mathbb{C}$ with $|z-1|\leq Mh$ and set $\lambda= h^{-1}z$.  We will also need semiclassical Sobolev spaces defined on a Riemannian manifold $(M,g)$, for $k\in\mathbb{N}$ by the norm
$$
\|u\|_{H_h^k(M)}^2:=\sum_{|\alpha|\leq k}\|(hD)^\alpha u\|_{L^2(M)}^2.
$$
We then define $H_h^s(M)$ for $s\geq 0$ by interpolation and $H_h^{-s}(M)$ by duality (Notice that when $M$ has a boundary $H_h^{-s}(M)$ is the space of supported distributions).

 We then write $f\in H^s_{h,\loc}(M)$ if for all $\chi\in C_c^\infty(M)$, $\chi f\in H^s_{h}(M)$. We write $f\in H_{h,\comp}^s$ if $f\in H_h^s(M)$ and $f$ is compactly supported.
	
	We use the language of semiclassical pseudodifferential operators frequently in this paper. We now briefly recall the concepts and notation (see~\cite{Zw:12} and\cite[Appendix E]{DyZw:19} for a complete treatment). We will define pseudodifferential operators on $\mathbb{R}^d$, the definitions on manifolds being similar and refer the reader to~\cite[Appendix E]{DyZw:19} for the precise definitions on a manifold.

\noindent{\bf{Semiclassical Pseudodifferential Operators on $\mathbb{R}^d$}}
    We say $a\in C^\infty(T^*\mathbb{R}^d)$ is a symbol of order $m$ and write $a\in S^m(T^*\mathbb{R}^d)$ if for all $\alpha,\beta\in\mathbb{N}^d$, there is $C_{\alpha\beta}>0$ such that 
    $$
    |\partial_x^\alpha\partial_{\xi}^\beta a(x,\xi)|\leq C_{\alpha\beta}\langle \xi\rangle^{m-|\beta|},\qquad \langle \xi\rangle:=(1+|\xi|^2)^{\frac{1}{2}}.
    $$
    We then quantize $a\in S^m(T^*\mathbb{R}^d)$ using the quantization
    $$
    [\Op(a)u](x):=\frac{1}{(2\pi h)^d}\int e^{\frac{i}{h}\langle x-y,\xi\rangle}a(x,\xi)u(y)\der y \der \xi,
    $$
    and define the class of semiclassical pseudodifferential operators of order $m$, 
    $$
    \Psi^m_h(\mathbb{R}^d):=\{ \Op(a)\,:\, a\in S^m(T^*\mathbb{R}^d)\}.
    $$
    We write $\Psi_h^\infty(\mathbb{R}^d):=\cup_{m}\Psi_h^m(\mathbb{R}^d)$ and $\Psi_h^{-\infty}(\mathbb{R}^d):=\cap _m \Psi_h^m(\mathbb{R}^d)$. 

        We next recall a few technical lemmas and definitions from the calculus of semiclassical pseudodifferential operators. The first gives the basic elements of the calculus.
    \begin{lem}[Theorem 9.5~\cite{Zw:12}]
    \label{l:compose}
    Let $a\in S^{m_1}(\mathbb{R}^d)$ and $b\in S^{m_2}(\mathbb{R}^d)$. Then,
    \begin{gather*}
    h^{-1}(\Op(a)\Op(b)-\Op(ab))\in \Psi_h^{m_1+m_2-1}(\mathbb{R}^d),\\
    h^{-1}(\Op(a)^*-\Op(\bar{a}))\in \Psi_h^{m_1-1}(\mathbb{R}^d),\\
    h^{-2}([\Op(a),\Op(b)]+hi\Op(\{a,b\}))\in \Psi_h^{m_1+m_2-2}(\mathbb{R}^d).
    \end{gather*}
    \end{lem}
    The next defines the principal symbol.
    \begin{lem}[Principal Symbol Map, Proposition E.14~\cite{DyZw:19}]
    There is a map $\sigma_m:\Psi_h^m(\mathbb{R}^d)\to S^m(T^*\mathbb{R}^d)$ so that 
    $$
    A-\Op(\sigma(A))\in h\Psi^{m-1}_h(\mathbb{R}^d). 
    $$
    \end{lem}

We write $\overline{T^*\mathbb{R}^d}=T^*\mathbb{R}^d\sqcup \mathbb{R}^d\times S^{d-1}$ for the fiber radially compactified cotangent bundle i.e. the cotangent bundle with the sphere at infinity in $\xi$ attached. 

We can now define the notion of the elliptic set.
\begin{defn}
Let $A\in \Psi_h^m(\mathbb{R}^d)$. For $(x_0,\xi_0)\in \overline{T^*M}$, we say that $A$ is elliptic at $(x_0,\xi_0)$ and write $(x_0,\xi_0)\in \Ellip(A)$ if there is a neighborhood, $U$ of $(x_0,\xi_0)$ and $c>0$ such that 
$$
|\sigma(A)(x,\xi)|\geq c\langle \xi\rangle^m,\qquad (x,\xi)\in U.
$$
\end{defn}
Next, we define the wavefront set of a pseudodifferential operator.
\begin{defn}
Let $A\in \Psi_h^m(\mathbb{R}^d)$. For $(x_0,\xi_0)\in \overline{T^*M}$ we say that $(x_0,\xi_0)$ is not in the wavefront set of $A$ and write $(x_0,\xi_0)\notin \WF(A)$ if there is $a\in S^m$ such that 
$$A=\Op(a)+O(h^\infty)_{\Psi_h^{-\infty}}$$
and 
$
(x_0,\xi_0)\notin \supp a.
$
\end{defn}
    The next lemma gives the so-called elliptic parametrix construction.
\begin{lem}[Proposiion E.32~\cite{DyZw:19}]
Suppose that $A\in \Psi_h^{m_1}(\mathbb{R}^d)$ and $B\in \Psi_h^{m_2}(\mathbb{R}^d)$ and $\WF(A)\subset \Ellip (B)$. Then there is $E\in \Psi_h^{m_1-m_2}(\mathbb{R}^d)$ with $\WF(E)\subset \WF(A)$ such that 
$$
A=EB+O(h^\infty)_{\Psi_h^{-\infty}}.
$$
\end{lem}
    The final lemma concerns boundedness of pseudodifferential operators. 
    \begin{lem}[Proposition E.19~\cite{DyZw:19}]
    Let $A\in \Psi_h^m(\mathbb{R}^d)$. Then, for all $s\in\mathbb{R}$,  there is $C>0$ such that for all $u\in H_h^{s+m}(\mathbb{R}^d)$ and $0<h<1$, 
    $$
    \|Au\|_{H_h^{s}(\mathbb{R}^d)}\leq C\|u\|_{H_h^{s+m}(\mathbb{R}^d)}.
    $$
    \end{lem}

\medskip

\noindent{\bf{Tangential Pseudodifferential operators}}
    We will also have occasion to use tangential pseudodifferential operators on a domain $\Omega\subset \mathbb{R}^d$ with smooth boundary. Once again, we make these definition in local coordinates $\mathbb{R}_{x^1}\times \mathbb{R}^{d-1}_{x'}$, we $\Omega=\{x^1>0\}$. 

    We say that $a\in C^\infty (\mathbb{R}\times T^*\mathbb{R}^{d-1})$ is a tangential symbol of order $m$ and write $a\in S_{\tangent}^m$ if $a\in C^\infty (\mathbb{R}_{x^1};S^m(\mathbb{R}^{d-1}))$. We then define the class of tangential pseudodifferential operators of order $m$ by 
    $$
    \Psi_{\tangent,h}^{m}:=\{ \Op(a)\,:\, a\in S^m_{\tangent,h}\}.
    $$
    We also write $\Psi_{\tangent,h}^\infty:=\cup_{m}\Psi_{\tangent, h}^m(\mathbb{R}^d)$ and $\Psi_{\tangent,h}^{-\infty}(\mathbb{R}^d):=\cap _m \Psi_{\tangent,h}^m(\mathbb{R}^d)$.
    
    Notice that for any $A\in \Psi_{\tangent,h}^m$ and $y\in \mathbb{R}$ $A$ can be restricted to an operator on $\{x^1=y\}$ and this operator lies in $\Psi_h^m(\{x^1=y\})$. 

    \subsection{The operator in Fermi Normal Coordinates}
	\label{Laplace operator around boundary}
	In Fermi normal coordinates $x=(x^1,x')$, where $x^1$ is the signed distance to the boundary, $\overline{\Omega}$ is given by $x^1\ge0$ and the metric is of the form 
	\begin{equation*}
		g=(\der x^1)^2+\sum_{\alpha,\beta=2}^d g_{\alpha\beta}(x) \der x^\alpha \der x^\beta.
	\end{equation*}
	Then, 
	\begin{equation}
		\label{LaplaceonBoundary}
		-h^2\Delta_{g,\rho}-z^2
		=(hD_{x^1})^2+ha(x)hD_{x^1}-R(x^1,x',hD_{x'}).
	\end{equation}
	Here, $a$ is a smooth function given by $a=(\sqrt{|g|}\rho)^{-1}D_{x^1}\sqrt{|g|}\rho$ with $\sqrt{|g|}$ being the Riemannian density function. Moreover, $R$ is a tangential differential operator of order $2$. The semiclassical principal symbol of $R$ is given by $\sigma(R)=r(x^1,x',\xi')$ with $r(0,x',\xi')=1-|\xi'|^2_{g}$.

    	\subsection{Semiclassical defect measures} \label{Semiclassical defect measures}
	Semiclassical defect measures are measures associated with a sequence (possibly subsequence) of functions $\{u(h)\}_{0<h<h_0}$. Some well-known existence theorem of semiclassical defect measures can be found in \cite[Chapter 5]{Zw:12} or \cite[Appendix E.3]{DyZw:19}. We will summarise them in the following. 
	\begin{itemize}
		\item If $u_j:=u(h_j)$ satisfies
		\begin{equation}
			\label{L2 loc condition}
			\|\chi u_j\|_{L^2(\Omega)}\le C_\chi
		\end{equation}
		for some constant $C_\chi$ depending on $\chi\in C^\infty_{c}(\overline{\Omega})$ but not $j$, then there is a subsequence $j_n$ and a non-negative Radon measure $\mu$ on $T^*\Omega$ such that
		\begin{equation}
			\label{interior measure}
			\langle \Op_h(a)(x,h_{j_n}D) u_{j_n}, u_{j_n} \rangle \to \int_{T^*\Omega} a(x,\xi) \der \mu \quad \text{for} \quad \forall a \in C^\infty_{c}(T^*\Omega).
		\end{equation}
		\item If $u_j$ satisfies
		\begin{equation}
			\label{L2 Dirichlet condition}
			\|u_j\|_{L^2(\partial \Omega)}\le C_D,
		\end{equation}
		then there is a subsequence $j_n$ and a non-negative Radon measure $\upsilon_D$ on $T^*\partial \Omega$ such that
		\begin{equation*}
			\label{Dirichlet measure}
			\langle \Op_h(b)(x',h_jD') u_j, u_j \rangle \to \int_{T^*\partial \Omega} b(x',\xi') \der \upsilon_D \quad \text{for} \quad \forall b \in C^\infty_{c}(T^*\partial \Omega).
		\end{equation*}
		\item If $u_j$ satisfies
		\begin{equation}
			\label{L2 Neumann condition}
			\|hD_\nu u_j\|_{L^2(\partial \Omega)}\le C_N,
		\end{equation}
		then there exists a non-negative Radon measure $\upsilon_N$ on $T^*\partial \Omega$ such that
		\begin{equation*}
			\label{Neumann measure}
			\langle \Op_h(b)(x',h_jD') hD_\nu u_j, hD_\nu u_j \rangle \to \int_{T^*\partial \Omega} b(x',\xi') \der \upsilon_N \quad \text{for} \quad \forall b \in C^\infty_{c}(T^*\partial \Omega).
		\end{equation*}
		\item If $u_j$ satisfies \eqref{L2 Dirichlet condition} and \eqref{L2 Neumann condition}
		then in addition to measures $\upsilon_D$ and $\upsilon_N$, there exists another Radon measure $\upsilon_{DN}$ on $T^*\partial \Omega$ such that
		\begin{equation*}
			\label{Dirichlet Neumann measure}
			\langle \Op_h(b)(x',h_jD')u_j, hD_\nu u_j \rangle \to \int_{T^*\partial \Omega} b(x',\xi') \der \upsilon_{DN} \quad \text{for} \quad \forall b \in C^\infty_{c}(T^*\partial \Omega).
		\end{equation*}
		\item Let $u_j$ satisfy
		\begin{equation}
			\label{PB_n condition}
			\begin{cases}
				(P_{\exterior}-z^2) u_j=h_jf_j  &\text{in } \Omega,\\
				(h_jD_{\nu_{\exterior}}+\Lambda)u_j=g_j  &\text{on } \partial \Omega,
			\end{cases}
		\end{equation}
		for some $f_j\in L^2_{\comp}(\Omega)$, $\|f_j\|_{L^2(\Omega)}\leq C$ and $g_j\in H_h^{\frac{1}{2}}(\partial\Omega)$, $\|g_j\|_{H_h^{\frac{1}{2}}(\partial \Omega)}\leq C$, and $\Lambda\in \Psi_h^1(\partial\Omega)$. Then there exists Radon measures $\mu_f$ on $T^*\Omega$ and $\sigma_g$ on $T^*\partial \Omega$ such that
		\begin{equation*}
			\label{PB_n measures}
			\left\{
			\begin{aligned}
				&\langle \Op_h(a)(x,h_jD) u_j, f \rangle \to \int_{T^*\Omega} a(x,\xi) \der \mu_f \quad \text{for} \quad \forall a \in C^\infty_{c}(T^*\Omega),\\
				&\langle \Op_h(b)(x',h_jD') u_j, g \rangle \to \int_{T^*\partial \Omega} b(x',\xi') \der \upsilon_g \quad \text{for} \quad \forall b \in C^\infty_{c}(T^*\partial \Omega).
			\end{aligned}
			\right.
		\end{equation*}
		If $u_j$ further satisfies \eqref{L2 loc condition}, then $\supp(\mu)\cap T^*\Omega \subset \Sigma_p:=\{p=0\}$.
	\end{itemize}
	Notice that if $u_j$ satisfies \eqref{L2 Dirichlet condition}, \eqref{L2 Neumann condition} and \eqref{PB_n condition} then
	\begin{equation}
		\label{boundary measures relation}	
		\upsilon_g=\upsilon_{DN}+\overline{\sigma(\Lambda)}\upsilon_D.
	\end{equation}
	To obtain relationships between the interior defect measures and boundary defect measures one uses the following integration by parts formula.
	\begin{lem}[Integration by parts]
		\label{Integration by parts}
		Suppose that 
		$$
		P=(hD_{x^1})^2+ha(x)hD_{x^1}-R(x,hD_{x'})
		$$
		is formally self adjoint with respect to the density $\rho \der x$. Let $B=B_0+B_1hD_{x^1}$ with $B_i\in C^\infty_c((-2\delta,2\delta)_{x^1};\Psi^{\ell_i}_h(\R^{d-1}))$ for $i=1,2$. Moreover, $\Omega$ is defined for $x^1>0$. Then we have
		\begin{multline*}
			\label{IBPeqn}
			\frac{\rmi}{h}\langle [P,B]u,u \rangle_{L^2(\Omega)}=-\frac{2}{h} \Im\left(\langle Bu, Pu \rangle_{L^2(\Omega)}\right)+\frac{\rmi}{h}\langle Pu, (B-B^*)u \rangle_{L^2(\Omega)}
			\\
			-\Big(\langle B_0u,hD_{x^1}u\rangle_{L^2(\partial \Omega)}+\langle B_1hD_{x^1}u,hD_{x^1}u\rangle_{L^2(\partial \Omega)}+\langle B_1Ru,u \rangle_{L^2(\partial \Omega)}
			\\
			+\langle hD_{x^1}B_0 u, u\rangle_{L^2(\partial \Omega)}+\langle [hD_{x^1},B_1] hD_{x^1}u, u\rangle_{L^2(\partial \Omega)}+h\langle [a,B_1]hD_{x^1}u,u\rangle_{L^2(\partial\Omega)}+h\langle aB_0u,u\rangle_{L^2(\partial\Omega)}\Big),
		\end{multline*}
        where $a=\rho^{-1} D_{x^1}\rho$.
	\end{lem}
	\begin{proof}
		Using the measures $\rho \der x$ in $\{x^1>0\}=\Omega$ and $\rho \der x'$ on $\{x^1=0\}=\partial\Omega$, from expression \eqref{LaplaceonBoundary}, we have
		\begin{equation*}
			\begin{aligned}
				\langle PBu,u \rangle_{L^2(\Omega,\rho \der x)}&=\langle Bu, Pu \rangle_{L^2(\Omega)}+\rmi h\left(\langle hD_{x^1}Bu,u \rangle_{L^2(\partial \Omega)}+\langle Bu,hD_{x^1}u \rangle_{L^2(\partial \Omega)}\right),\\
				\langle BPu,u \rangle_{L^2(\Omega)}&=\langle Pu, Bu \rangle_{L^2(\Omega)}+\langle Pu, (B^*-B)u \rangle_{L^2(\Omega)}+\rmi h\langle Pu,B_1^*u \rangle_{L^2(\partial \Omega)}.
			\end{aligned}
		\end{equation*}
		One also has
		\begin{multline*}
			\langle hD_{x^1}B u, u\rangle_{L^2(\partial \Omega)}=\langle (hD_{x^1})^2 u, B_1^*u\rangle_{L^2(\partial \Omega)}
			\\
			+\langle hD_{x^1}B_0 u, u\rangle_{L^2(\partial \Omega)}+\langle [hD_{x^1},B_1] hD_{x^1}u, u\rangle_{L^2(\partial \Omega)}.
		\end{multline*}
		Therefore, the boundary contributions of $\langle [P,B]u,u \rangle_{L^2(\Omega)}$ is given by $\rmi h$ multiplying with
		\begin{multline*}
			\langle Bu,hD_{x^1}u \rangle_{L^2(\partial \Omega)}+\langle hD_{x^1}B u, u\rangle_{L^2(\partial \Omega)}-\langle Pu,B_1^*u \rangle_{L^2(\partial \Omega)}
			\\
			=\langle Bu,hD_{x^1}u \rangle_{L^2(\partial \Omega)}+\langle (R-ah^2D_{x^1})u,B_1^*u \rangle_{L^2(\partial \Omega)}
			\\
			+\langle hD_{x^1}B_0 u, u\rangle_{L^2(\partial \Omega)}+\langle [hD_{x^1},B_1] hD_{x^1}u, u\rangle_{L^2(\partial \Omega)}
			\\
			=\langle Bu,hD_{x^1}u\rangle_{L^2(\partial \Omega)}+\langle B_1Ru,u \rangle_{L^2(\partial \Omega)}+\langle hD_{x^1}B_0 u, u\rangle_{L^2(\partial \Omega)}
			\\
			+\langle [hD_{x^1},B_1] hD_{x^1}u, u\rangle_{L^2(\partial \Omega)}+h\langle [a,B_1]hD_{x^1}u,u\rangle_{L^2(\partial\Omega)}+h\langle aB_0u,u\rangle_{L^2(\partial\Omega)}.
		\end{multline*}
		This completes the proof.
	\end{proof}

	When $B_0=0$, we have
	\begin{multline*}
		\frac{\rmi}{h}\langle [P,B_1hD_{x^1}]u,u \rangle_{L^2(\Omega)}=
		\\
		-\frac{2}{h} \Im\left(\langle B_1hD_{x^1}u, Pu \rangle_{L^2(\Omega)}\right)+\frac{\rmi}{h}\langle Pu, (B_1hD_{x^1}-(hD_{x^1})^*B_1^*)u \rangle_{L^2(\Omega)}
		\\
		-\Big(\langle B_1hD_{x^1}u,hD_{x^1}u\rangle_{L^2(\partial \Omega)}+\langle B_1Ru,u \rangle_{L^2(\partial \Omega)}
		\\
		+h\langle [a,B_1]hD_{x^1}u,u\rangle_{L^2(\partial\Omega)}+\langle [hD_{x^1},B_1] hD_{x^1}u, u\rangle_{L^2(\partial \Omega)}\Big).
	\end{multline*}
	When $B_1=0$, we have
	\begin{multline*}
		\frac{\rmi}{h}\langle [P,B_0]u,u \rangle_{L^2(\Omega)}=-\frac{2}{h} \Im\left(\langle B_0u, Pu \rangle_{L^2(\Omega)}\right)+\frac{\rmi}{h}\langle Pu, (B_0-B_0^*)u \rangle_{L^2(\Omega)}
		\\
		-\left(\langle B_0u,hD_{x^1}u\rangle_{L^2(\partial \Omega)}+\langle hD_{x^1}B_0 u, u\rangle_{L^2(\partial \Omega)}+h\langle aB_0u,u\rangle_{L^2(\partial\Omega)}\right).
	\end{multline*}
	Using Lemma \ref{Integration by parts}, one can obtain the results of~\cite{Miller} (see also~\cite{GaLaSp}) on boundary defect measures. \blue{In the next theorem, we use the notation $\delta_{\partial\Omega}\otimes \upsilon$ for a measure $\upsilon$ on $T^*\partial\Omega$, that acts on $f\in C_c^\infty({}^bT^*\Omega)$ as 
        $$
        \langle \delta_{T^*\partial\Omega}\otimes \upsilon, f\rangle :=\int f|_{T^*\partial\Omega}\der\upsilon,
        $$}
	\begin{theorem}
		\label{measure invariant thm}
		Let $u_j$ satisfy \eqref{L2 loc condition}, \eqref{L2 Dirichlet condition}, \eqref{L2 Neumann condition}, and \eqref{PB_n condition}, \blue{then for any subsequence $j_n$ satisfying~\eqref{interior measure} for a Radon measure $\mu$, we have} $\supp \mu\subset \{
    \sigma(P-1)=0\}\cap T^*\Omega$ and
		\begin{multline}
			\label{interior boundary measure eqn}
			\mu(H_p a)=-2\Im\mu_f(\Re a)-2\rmi \Im\mu_f(\Im a)
			\\
			-2\Re\upsilon_g(a_{\operatorname{even}})-2|\sigma( \Lambda)|\Im\upsilon_D(a_{\operatorname{even}})-\upsilon_N(a_{\operatorname{odd}})-\upsilon_D(ra_{\operatorname{odd}}),
		\end{multline}
		where $a_{\operatorname{even}}=\frac{1}{2}(a(x,r^{\frac{1}{2}},\xi')+a(x,-r^{\frac{1}{2}},\xi'))$, $a_{\operatorname{odd}}=\frac{1}{2r^{\frac{1}{2}}}(a(x,r^{\frac{1}{2}},\xi')-a(x,-r^{\frac{1}{2}},\xi'))$ and $r=\sigma(R)$ in \eqref{LaplaceonBoundary}. Let $\pi: T^*\Omega \to {}^bT^*\Omega$ define as $\pi(x^1,x',\xi_1,\xi')=(x^1,x',x^1\xi_1,\xi')$. If $\partial \Omega$ is nowhere tangent to $H_p$ to infinite order. Then, for $q\in C^\infty_{c}({}^bT^*\Omega;\R)$, we have
		\begin{equation}
			\label{measure invariant eqn}
			\pi_*\mu(q\circ \varphi_t)-\pi_*\mu(q)=\int_0^t\left(-2\Im\pi_*\mu_f-2\delta_{\blue{T^*}\partial \Omega}\otimes (\Re\upsilon_{DN})+\upsilon_\calG\right)(q\circ \varphi_s) \der s,
		\end{equation}
		where $\varphi_t$ is the bicharacteristic flow and the measure $\upsilon_\calG$ is only supported on the glancing set. Moreover, $\mu_f=0$ if $f=o(1)$ and, similarly, $\Re\upsilon_{DN}=\upsilon_\calG=0$ if $g=o(1)$.
	\end{theorem}
	\begin{proof}
		We will briefly mention the proof. Identity \eqref{interior boundary measure eqn} follows from identity \eqref{boundary measures relation}, Lemma \ref{Integration by parts} and the fact that $(H_p a)|_{S^*\R^d}=H_p (a|_{S^*\R^d})$. Here, $H_p$ is the Hamiltonian vector field generated by $p=\sigma(P)$. The proof of \eqref{measure invariant eqn} can be found in \cite[Section 2]{GaLaSp}. It is clear that $f=o(1)$ implies $\mu_f=0$. For $g=o(1)$, we know from \eqref{boundary measures relation} that $\upsilon_{DN}=-\overline{\sigma(\rmi \tau \Lambda_\inner)}\upsilon_D$, whose real part is zero as $\overline{\sigma(\rmi \tau \Lambda_\inner)}$ is purely imaginary (see Proposition \ref{symbol of DtN interior}). Finally, $\upsilon_\calG=0$ follows immediately from $\Re\upsilon_{DN}=0$.
	\end{proof}

    \subsection{An estimate for the Neumann Trace}
    \begin{lem}
    \label{l:neumannBounds}
    Suppose that 
		$$
		P=(hD_{x^1})^2+ha(x)hD_{x^1}-R(x,hD_{x'})
		$$
		is formally self adjoint with respect to the density $\rho \der x$. Then, for $s\leq \frac{1}{2}$ and $\e>0$, there is $C>0$ such that for $u\in H_h^1$ with $Pu\in L^2$, 
        $$
        \|hD_{x^1}u|_{x^1=0}\|_{H_h^{s}(\partial\Omega)}\leq Ch^{-1}\|Pu\|_{L^2(0,\e)}+C\|u\|_{H_h^1(0,\e)}+C\|u|_{x^1=0}\|_{H_h^{s+1}(\partial\Omega)}.
        $$
    \end{lem}
    \begin{proof}
    Let $E\in \Psi_h^{s}(\partial\Omega)$ elliptic, $\chi\in C_c^\infty([0,\e))$ with $\chi\equiv 1$ near $0$, set $B_1:=\chi(x^1)E$, $B_0=0$, $B:=E^*E\chi(x^1)hD_{x^1}$. Then, by Lemma~\ref{Integration by parts},
    \begin{align*}
    \|hD_{x^1}u|_{x^1=0}\|_{H_h^{s}(\partial\Omega)}^2&\leq \langle EhD_{x^1}u,EhD_{x^1}u\rangle_{L^2(\partial\Omega)}\\
    &=\frac{\rmi}{h}\langle [P,B]u,u\rangle_{L^2(\Omega)}-\frac{2}{h}\Im \langle Bu, Pu\rangle_{L^2(\Omega)}-\langle ERu,Eu\rangle_{L^2(\partial\Omega)}\\
    &\qquad+h\langle [a,\chi(x^1)E^*E]hD_{x^1}u,u\rangle_{L^2(\partial\Omega)}\\
    &\leq C\|u\|_{H_h^1(0,\e)}^2+Ch^{-1}\|u\|_{H_h^1(0,\e)}\|Pu\|_{L^2(0,\e)}+\|u|_{x^1=0}\|_{H_h^{s+1}(\partial\Omega)}^2\\
&\qquad+Ch^2\|hD_{x^1}u|_{x^1=0}\|_{H_h^s(\partial\Omega)}\|u|_{x^1=0}\|_{H_h^{s-1}(\partial\Omega)}\\
&\leq C\|u\|_{H_h^1(0,\e)}^2+\delta \|u\|_{H_h^1(0,\e)}^2+Ch^{-2}\delta^{-1}\|Pu\|_{L^2(0,\e)}\\
&\qquad +C(1+\delta^{-1})\|u|_{x^1=0}\|_{H_h^{s+1}(\partial\Omega)}^2+h^2\delta\|hD_{x^1}u|_{x^1=0}\|^2_{H_h^s(\partial\Omega)}.
    \end{align*}
  Taking $\delta>0$ small enough completes the proof of the lemma.
    \end{proof}
	
		\section{Reformulation of negative index of refraction problem as an exterior problem}
	\label{Reformulation of the problem as an exterior problem}
In this section, we reformulate the estimates for $R_P(z)$ in terms of an exterior scattering problem. To do this, we first review estimates for more classical resolvents.
    \subsection{Review of estimates for the Dirichlet resolvent in $\Omega_{\exterior}$}\label{s:exeteriorResolve}
	Since $\Omega_{\exterior}$ is connected, $P_{\exterior}$ is self-adjoint with domain $H_0^1\cap H^2$, and $g_{\exterior}^{ij}(x)=\delta^{ij}$, $\rho(x) \equiv 1$ for $|x|$ large enough, the theory of black-box scattering~\cite[Chapter 4]{DyZw:19} implies that there is a meromorphic family of operators $R_{\exterior}(z):L^2_{\comp}(\Omega_{\exterior})\to H^2_{h,\loc}\cap H_{0,\loc}^{1}(\Omega_{\exterior})$ satisfying
	$$
	(P_{\exterior}-z^2)R_{\exterior}(z)f=f\text{ in }\Omega_{\exterior},\,\, R_{\exterior}(z/h)f\text{ is }z/h\text{-outgoing}.
	$$
	Moreover, since $g_{\exterior}$ is non-trapping on $\Omega_{\exterior}$, combining~\cite[Theorem 1.3]{Bu:02} with elliptic regularity, we have for any $M>0$, there is $h_0>0$ such that  $R_{\exterior}(z)$ is  analytic in, $|z-1|\leq Mh$ and for all $\chi \in C_c^\infty(\overline{\Omega}_{\exterior})$ there is $C>0$ such that 
	\begin{equation}
    \label{e:exteriorResolve}
	\|\chi R_{\exterior}(z)\chi\|_{L^2(\Omega_{\exterior})\to H_h^2(\Omega_{\exterior})}\leq Ch^{-1},\,\,\, 0<h<h_0,\, \, |z-1|<Mh.
	\end{equation}
	 By Lemma~\ref{l:neumannBounds}, this implies
	    and hence also
    \begin{equation}
    \label{e:normalTraceBoundExterior}
	\|h\partial_\nu  R_{\exterior}(z)\chi\|_{L^2(\Omega_{\exterior})\to H_h^{\frac{1}{2}}(\partial\Omega)}\leq Ch^{-1}+\|\chi R_{\exterior}(z)\chi\|_{L^2(\partial\Omega)\to H_h^1(\Omega_{\exterior})}\leq Ch^{-1}.
    \end{equation}
	
	Moreover, letting $E_{\exterior}:H_h^{\frac{3}{2}}(\partial\Omega)\to H_{h,\comp}^{2}(\Omega_{\exterior})$, be an extension operator, the operator $G_{\exterior}(z):H_{h}^{\frac{3}{2}}(\partial\Omega)\to H^2_{h,\loc}(\Omega_{\exterior})$ defined by
    $$
    G_{\exterior}(z)v:=E_{\exterior}v-R_{\exterior}(z)E_{\exterior}v,
    $$
     is a meromorphic family of operators  satisfying
	$$
	(P_{\exterior}-z^2)G_{\exterior}(z)v=0\text{ in }\Omega_{\exterior},\,\, G_{\exterior}(z)g|_{\partial\Omega}=v,\,\, G_{\exterior}(z)v\text{ is }z/h\text{-outgoing}.
	$$
    We now obtain an improved version of \cite[Theorem 3.5]{BaSpWu:16}.
      The following proposition is an improvement of \cite[Theorem 3.5]{BaSpWu:16}, where we combine the method used in the proof of \cite[Theorem 3.5]{BaSpWu:16} and Lemma \ref{l:gainNormal}.
    \begin{prop}
    \label{G bound thm}
    Let $G_{\exterior,h}$ be the Dirichlet map for $(P_{\exterior}-z^2)$ satisfying $z/h$ outgoing condition. Then for all $M>0$, $\chi \in C_c^\infty(\overline{\Omega}_{\exterior})$, and $j=0,1$ there are $C,h_0>0$ such that for $0<h<h_0$,
    \begin{equation}
    \label{e:dirichletResolveNontrap}
    \|\chi G_{\exterior,h}(z)\|_{H^{\frac{1}{2}+j}_h(\partial \Omega)\to H^{1+j}_h(\Omega_\exterior)}\le C , \quad \text{for} \quad |z-1|\leq Mh,
    \end{equation}
    and 
    	\begin{equation}
        \label{e:dirichletReesolveNeumann}
	\|h\partial_{\nu_{\exterior}}  G_{\exterior}(z)\chi\|_{H_h^{\frac{1}{2}+j}(\partial\Omega)\to H_h^{-\frac{1}{2}+j}(\partial\Omega)}\leq C.
	\end{equation}
    \end{prop}
	\begin{proof}
	Let $g\in H_h^{\frac{1}{2}}(\partial\Omega)$ and $w$ be a solution to 
    \begin{equation*}
        \left\{
        \begin{aligned}
            &-h^2\Delta_{g_{\exterior},\rho_{\exterior}} w_+- (1+\rmi h)^2 w_+=0 & \text{in} \quad \Omega_\exterior,\\
            &w_+=g & \text{on} \quad \partial \Omega.
        \end{aligned}
        \right.
    \end{equation*}
     Then Green's identity implies
    \begin{equation}
        \begin{aligned}
            &h\|w_+\|_{L^2(\Omega_\exterior)}^2\le C h|\langle h \partial_\nu w_+, g\rangle_{\partial \Omega}|
            \label{e:green1},\\
            &\|w_+\|^2_{H_h^1(\Omega_\exterior)}\le C\left(\|w_+\|_{L^2(\Omega_\exterior)}^2+h|\langle h \partial_\nu w_+, g\rangle_{\partial \Omega}|\right),
        \end{aligned}
        \end{equation}
    where the boundary contribution at infinity is zero since $h>0$ implies $w_+\in H^1$.

    Applying Lemma~\ref{l:neumannBounds}, one has for $s\leq \frac{3}{2}$,
    \begin{equation}
    \label{e:normalEst}
    \|h\partial_{\nu_{\exterior}}w_+\|_{H_h^{s-1}(\partial\Omega)}\leq \|w_+\|_{L^2(\Omega_\exterior)} \|w_+\|_{H_h^1(\Omega_\exterior)}+\|g\|_{H_h^{s}(\partial\Omega)}.
    \end{equation}
    Hence, using~\eqref{e:normalEst} with $s=\frac{1}{2}$ and~\eqref{e:green1}, we obtain
    \begin{equation*}
        \begin{aligned}
        &\|w_+\|_{H_h^1(\Omega_\exterior)}\le C \e  \|w_+\|_{H_h^1(\Omega_{\exterior})}+C(\e+\e^{-1})\|g\|_{H^{\frac{1}{2}}_h(\partial \Omega)}.
        \end{aligned}
    \end{equation*}
In particular,  taking $\e=\frac{1}{2C}$, we have
    \begin{equation}
        \label{G bound eqn 1}
        \|w_+\|_{H^1_h(\Omega_\exterior)}\le C\|g\|_{H^{\frac{1}{2}}_h(\partial \Omega)}.
    \end{equation}
Note also that if $g\in H_h^{\frac{3}{2}}(\partial \Omega)$, then by elliptic regularity for $-h^2\Delta_{g_{\exterior},\rho_{\exterior}}+1$,
\begin{equation}
\label{e:h2}
\|w_+\|_{H^2_h(\Omega_\exterior)}\leq C(\|w_+\|_{H_h^1(\Omega_\exterior)}+\|g\|_{H_h^{\frac{3}{2}}(\partial\Omega)})\leq C\|g\|_{H_h^{\frac{3}{2}}(\partial\Omega)}.
\end{equation}

Let $\chi_1,\chi_2\in C^\infty_0(\overline{\Omega_\exterior})$ $\chi_1\equiv 1$ near $\partial\Omega$, and $\supp (1-\chi_2)\cap \supp \chi_1=\emptyset$. Then, notice that 
\begin{align*}
G_{\exterior}(z)g&=\chi_1 w_+ -R_{\exterior}(z)(P_{\exterior}-z^2)\chi_1 w_+\\
&=\chi_1 w_+ -R_{\exterior}(z)\chi_2[-h^2\Delta_{g_\exterior,\rho_{\exterior}},\chi_1]w_+-((1+ih)^2-z^2)R_{\exterior}(z)\chi_2\chi_1 w_+.
\end{align*}
The estimate~\eqref{e:dirichletResolveNontrap} now follows from~\eqref{G bound eqn 1},~\eqref{e:h2}, and the estimates 
$$
\|\chi R_{\exterior}\chi_2\|_{L^2(\Omega_{\exterior})\to H_h^2(\Omega_{\exterior})}\leq Ch^{-1},\qquad \|[-h^2\Delta_{g_\exterior,\rho_{\exterior}},\chi_1]\|_{H_h^1(\Omega_{\exterior})\to L^2(\Omega_{\exterior})}\leq Ch.
$$
The estimate~\eqref{e:dirichletReesolveNeumann} now follows from Lemma~\ref{l:neumannBounds}.
	\end{proof}


	\subsection{Review of estimates for the Dirichlet resolvent in $\Omega_{\inner}$}\label{s:interiorResolve}
	Since $P_{\inner}$ is self-adjoint with domain $H_0^1(\Omega_{\inner})\cap H^2(\Omega_{\inner})$ and $P_{\inner}\leq Ch$, for $|z-1|\leq C_0 h$, there is a holomorphic family of operators $R_{\inner}(z):=(P_\inner -z^2)^{-1}:L^2(\Omega_{\inner})\to H^2_{h}(\Omega_{\inner})\cap H_{0}^{1}(\Omega_{\inner})$ satisfying
	\begin{equation*}
		\label{HTP}
		(P_{\inner} -z^2)R_{\inner}(z)f=f\text{ in }\Omega_{\inner}.
	\end{equation*}
	Moreover, 
	$
	\| R_{\inner}(z)\|_{L^2(\Omega_\inner)\to L^2(\Omega_{\inner})}\leq C.
    $
    Hence by elliptic regularity estimates (see e.g.~\cite[Theorem 4.18]{Mc:00}), 
    \begin{equation}
    \label{e:interiorResolve}
    \| R_{\inner}(z)\|_{L^2(\Omega_\inner)\to H_h^2(\Omega_\inner)}\leq C
    \end{equation}
    and by Lemma~\ref{l:neumannBounds}
    \begin{equation}
    \label{e:normalTraceBoundInner}
    \| h\partial_\nu  R_{\inner}(z)\|_{L^2(\Omega_\inner)\to H_h^{\frac{1}{2}}(\partial\Omega)}\leq C.
\end{equation}

	In addition, letting $E_{\inner}:H_h^{\frac{3}{2}}(\partial\Omega)\to H_h^2(\Omega_{\inner})$ be an extension operator with 
    $$
    \|E_{\inner}\|_{H_h^{\frac{3}{2}}(\partial\Omega)\to H_h^{2}(\Omega_{\inner})}\leq Ch^{\frac{1}{2}},
    $$
    and defining $G_{\inner}(z):H_{h}^{\frac{3}{2}}(\partial\Omega)\to H^2_{h}(\Omega_{\inner})$ by 
    $$
    G_{\inner}(z)v:=E_{\inner}v-R_{\inner}(z)E_{\inner}v.
    $$
    $G_{\inner}(z)$ satisfies
	$$
	(P_{\inner} -z^2)G_{\inner}(z)g=0\text{ in }\Omega_{\inner},\,\, G_{\inner}(z)g|_{\partial\Omega}=g, 
	$$
    and for any $M>0$ there is $h_0>0$ such that for $0<h<h_0$, $|z-1|\leq Mh$, and $j=0,1$
	\begin{equation}
    \label{e:innerEstimates}
	\|G_{\inner}(z)g\|_{H_h^{\frac{1}{2}+j}(\partial \Omega)\to H_h^{1+j}(\Omega_\inner)}+ \|h\partial_\nu  G_{\inner}(z)\chi\|_{H_h^{\frac{1}{2}+j}(\partial \Omega)\to H_h^{-\frac{1}{2}+j}(\Omega_\inner)}\leq C.
	\end{equation}
	
\subsection{Reformulation of the negative index of refraction problem}	
	
	It is convenient to reduce our problem to studying the case of $f\equiv 0$ at the cost of placing inhomogeneous data in the boundary condition. To do this, define $v_\exterior :=u_\exterior -R_{\exterior}(z)h\chi f_\exterior $ and $v_\inner:=u_\inner-R_{\inner}(z)hf_\inner $.
	
	Then, using~\eqref{e:transmission1} we obtain
	\begin{equation}
		\label{e:transmission2}
		\begin{cases}
			(P_\inner-z^2)v_\inner=0&\text{in }\Omega_{\inner},\\
			(P_{\exterior}-z^2)v_\exterior =0&\text{in }\Omega_{\exterior},\\
			v_\exterior =v_\inner &\text{on }\partial\Omega,\\
			\rho_{\exterior}h\partial_{\nu_{\exterior}}v_\exterior -\tau\rho_{\inner}h\partial_{\nu_{\inner}} v_\inner=:g&\text{on }\partial\Omega,\\
            v_\exterior \text{ is }z/h\text{-outgoing}.\\
		\end{cases}
	\end{equation}
    Notice that by~\eqref{e:exteriorResolve},
    \begin{equation}
    \label{e:extDiff}
    \|\chi(v_{\exterior}-u_{\exterior})\|_{H_h^2(\Omega_{\exterior})}\leq Ch^{-1}\|f_{\exterior}\|_{L^2(\Omega_{\exterior})}, 
    \end{equation}
    by~\eqref{e:interiorResolve}
    \begin{equation}
    \label{e:inDiff}
    \|v_{\inner}-u_{\inner}\|_{H_h^2(\Omega_{\inner})}\leq C\|f_{\inner}\|_{L^2(\Omega_{\inner})}.
    \end{equation}
    Next, by~\eqref{e:normalTraceBoundExterior} and~\eqref{e:normalTraceBoundInner},
    \begin{equation}
    \label{e:traceData}
    \|g\|_{H_h^{\frac{1}{2}}(\partial \Omega)}\leq C\left(h^{-1}\|f_{\exterior}\|_{L^2(\Omega_{\exterior})}+\|f_{\inner}\|_{L^2(\Omega_{\inner})}\right).
    \end{equation}
    Finally, observe that 
    \begin{equation}
    \label{e:inEstimate}
    \|v_{\inner}\|_{H_h^2(\Omega_\inner)}\leq C\|v_{\exterior}\|_{H_h^{\frac{3}{2}}(\partial\Omega)}.
    \end{equation}
    Using~\eqref{e:extDiff},~\eqref{e:inDiff},~\eqref{e:traceData}, and~\eqref{e:inEstimate}, Theorem~\ref{t:noStates} will follow from the estimate
    \begin{equation}
    \label{e:noStatesToProve}
    \|\chi v_{\exterior}\|_{H_h^2(\Omega_{\exterior})}+\|v_{\exterior}\|_{H_h^{\frac{3}{2}}(\partial\Omega)}\leq  C\|g\|_{H_h^{\frac{1}{2}}(\partial\Omega)}, \quad |\Im z|<Mh,
    \end{equation}
    and Theorem~\ref{t:states} will follow from the estimate
    \begin{equation}
    \label{e:statesToProve}
    \|\chi v_{\exterior}\|_{H_h^2(\Omega_{\exterior})}+\|v_{\exterior}\|_{H_h^{\frac{3}{2}}(\partial\Omega)}\leq  C|\Im z|^{-1}\|g\|_{H_h^{\frac{1}{2}}(\partial\Omega)},\quad -Mh<\Im z<-h^N.
    \end{equation}

	Since our goal is now to prove~\eqref{e:noStatesToProve} and~\eqref{e:statesToProve}, we can now reduce~\eqref{e:transmission2} to an exterior scattering problem with a non-standard Robin-type boundary condition. For this, let
	$ \Lambda_{\inner}(z):H_h^s(\partial\Omega)\to H_h^{s-1}(\partial\Omega)$ be the Dirichlet-to-Neumann (DtN) map defined as follows. $\Lambda_{\inner}(z)u_0:=h\rho_{\inner}\partial_{\nu_\inner}G_{\inner}u_0$, where $\nu_\inner$ is outward normal with respect to the metric $g_\inner$. We then rewrite~\eqref{e:transmission2} as
	\begin{equation}
		\label{e:transmission}
		\begin{cases}
			(P_{\exterior}  -z^2)v_\exterior =0&\text{in }\Omega_{\exterior},\\
			\rho_{\exterior}h\partial_{\nu_{\exterior}}v_\exterior -\tau \Lambda_{\inner}(z)v_\exterior |_{\partial\Omega}=g&\text{on }\partial\Omega,\\
            v_\exterior \text{ is }z/h\text{-outgoing}.
		\end{cases}
	\end{equation}
	The proof of Theorems~\ref{t:noStates} and~\ref{t:states} will consist of proving estimates on the solution to~\eqref{e:transmission}. Since we have eliminated $v_\inner$ from~\eqref{e:transmission}, we will abuse notation slightly and simply write $v_\exterior =v$ from now on.

	\section{Microlocal description of the Dirichlet-to-Neumann map}
    \label{Microlocal description of the Dirichlet-to-Neumann maps}

    In addition to the DtN map $\Lambda_{\inner}(z)$, we will use the outgoing DtN map, $\Lambda_{\exterior}(z):H_h^{s}(\partial\Omega)\to H_h^{s-1}(\partial\Omega)$ defined as $h\rho_{\exterior}\partial_{\nu_{\exterior}} G_{\exterior}(z)$. In this section we provide a full description of $\Lambda_{\inner}(z)$ as a pseudodifferential operator and a microlocal description of $\Lambda_{\exterior}(z)$ on $|\xi'|_{g_{\exterior}}>1$. In particular, we prove Theorem~\ref{t:dtN}. 
    
    In fact, we generalize our situation slightly, defining for $L:=\sum_{i=1}^d L^i(x)hD_{x^i}$ and $\omega=\omega_0+h\omega_1+o(h^2)$, 
    \begin{equation}\label{P(w;g,L)}
    P(\omega;g,L):=-h^2\Delta_g+hL-\omega.
    \end{equation}
    Notice that 
    \begin{equation}
    \label{e:pExtRewrite}
    \begin{aligned}P_{\exterior}-\omega&= -\frac{h^2}{\rho \sqrt{|g|}}\partial_i g^{ij}_\exterior \sqrt{|g|} \rho \partial_j-\omega\\
    &= -h^2\Delta_{g_\exterior}-h\rho^{-1}(\partial_i\rho)g^{ij}_\exterior h\partial_j-\omega=P(\omega;g_{\exterior},L_{\exterior}),
    \end{aligned}
    \end{equation}
    with 
    $$
    L_{\exterior}:=-\rho^{-1}(\partial_i\rho)g^{ij}_\exterior h\partial_j.
    $$
    Similarly, there is $L_{\inner}$ such that 
    \begin{equation}
    \label{e:pInRewrite}
    -P_{\inner}+\omega=P(-\omega;g_{\inner},L_{\inner}).
    \end{equation}

	\subsection{Semiclassical Lee-Uhlmann constructions}
	\label{LUconstruction}
To the best of authors' knowledge, the earliest paper showing DtN map as a classical (i.e. non-semiclassical) pseudodifferential operator and providing a way of calculating the full classical symbol expression of DtN map can be found in the work of Sylvester and Uhlmann \cite{SylvGunther1988}. The method of Sylvester and Uhlmann is based on the study of Calder\'{o}n projector. A more direct approach to DtN map via factorization modulo smoothing operators can be found in the work of Lee--Uhlmann~\cite{LU1989}, and their method allows one to calculate its full classical asymptotic expansions in a simpler and more intuitive way. In this section, we give a semiclassical version of Lee--Uhlmann approach. While the results in this sections are considered folk-lore knowledge, we were unable to find a reference in the literature.
\begin{rem}
In the simplest form of factorization problem, say solving the equation 
\begin{equation}
    \label{e:factoreqn}
    p(x;\lambda):=x^2-\lambda^2=0
\end{equation}
for some unknown number $x$, we can factorize $p(x;\lambda)$ as $(x-\lambda)(x+\lambda)$ and obtain $x=\lambda$ or $x=-\lambda$ as solutions to this toy problem. To further determine which solution to be valid, we would need more information about $x$. Lee--Uhlmann's idea is essentially solving an analogue of \eqref{e:factoreqn} for $p(x;\lambda)$ being the Laplace-Beltrami operator, $x$ being an unknown operator with some given classical pseudodifferential operator $\lambda$ and, furthermore, the right-hand-side of \eqref{e:factoreqn} is replaced by some smoothing operator. To determine which solution of $x$ to be the right candidate boils down to choosing which $\pm\lambda$ that gives the well-posedness of the heat equation. In fact, it is due to the nature that the heat operator $e^{\pm \lambda t}$ is only well-posed in positive time $t$ for $-\lambda$ if we assume $\lambda>0$ (See \cite[Proposition 1.2]{LU1989}). 

Our approach is essentially a semiclassical version of Lee-Uhlmann's method, i.e. $p(x;\lambda)$, $x$ and $\lambda$ are now semiclassical pseudodifferential operators. Note that Lemmas \ref{Elliptic estimate lemma} and \ref{l:lopa} to be proved in this section imply that depending on the sign of principal symbol of $\lambda$, the operator $x-\lambda$ enjoys different microlocal estimates. In other words, the operator $x+\lambda$ and $x-\lambda$ satisfy different microlocal estimates if we fix $\lambda$. In this way, we replaces the fulfillment of the well-posedness of the heat equation, as in Lee--Uhlmann's construction, by microlocal estimates. See Remark \ref{Epm comparison} for further details.
\end{rem}
	\blue{As in Section \ref{Semiclassical defect measures}, we will denote a domain in $\mathbb{R}^d$ with smooth boundary by $\Omega$.} The following lemma gives a semiclassical factorisation of the semiclassical Laplace-Beltrami with potentials.
	\begin{lem}
		\label{Pdecompmicrolocal}
		Let $\omega_0\in\mathbb{R}$ and suppose that $|\omega_0-\omega|<Ch$. Then, for $X\in \Psi^0_{\tangent,h}(\Omega)$ such that $\WF(X)\subset\left\{(x,\xi'):\left||\xi'|_g^2-\omega_0\right|>0\right\}$, we have, in the boundary normal coordinates,
		\begin{equation*}
			\label{Laplace factorisation microlocal}
			X P(\omega;g,L)=X\left(hD_{x^1}+h \tilde{a}-\rmi E_{\pm}(x,hD_{x'})\right)\left(hD_{x^1}+\rmi E_{\pm}(x,hD_{x'})\right) +O(h^\infty)_{ \Psi^{-\infty}_{\tangent,h}}
		\end{equation*}
		in the boundary normal coordinates, where $\tilde{a}=a+L^1$, $E_{\pm}\in \Psi^1_{\tangent,h}(\Omega)$ and its principal symbols is given by
		\begin{equation}
			\label{Epm}
			\sigma(E_{\pm})=e_1:=\begin{cases}\pm \rmi \sqrt{\omega_0-|\xi'|_g^2}&|\xi'|_g^2-\omega_0<0,\\
				\pm\sqrt{|\xi'|_g^2-\omega_0}&|\xi'|_g^2-\omega_0>0.
			\end{cases}
		\end{equation} 
	\end{lem}
	\begin{proof}
		Our strategy is to show $E_{\pm}=\sum_{j\le 1}h^{1-j}\Op_h(e_j)$ for some $e_j \in S^{j}_{1,0}(T^*\partial \Omega)$. First, set $E_{\pm}=\Op_h(e_1)$ with $e_1$ as in~\eqref{Epm}. Then it follows from equation \eqref{LaplaceonBoundary} and definition of $P(\omega,g,L)$ that
		\begin{equation*}
			X_+\left(hD_{x^1}+h \tilde{a}-\rmi E_{\pm}(x,hD_{x'})\right)\left(hD_{x^1}+\rmi E_{\pm}(x,hD_{x'})\right)=X_+ P(\omega;g,L)+hX_+F_1,
		\end{equation*}
		where $F_1\in \Psi^0_{\tangent,h}$. This proves our first induction step. Suppose that we have proved the $k$-th induction step, i.e.
		\begin{equation*}
			X_+\left(hD_{x^1}+h \tilde{a}-\rmi E_{+,k}(x,hD_{x'})\right)\left(hD_{x^1}+\rmi E_{+,k}(x,hD_{x'})\right)=X_+ P(\omega;g)+h^kX_+F_k,
		\end{equation*}
		where $E_{+,k}=\sum_{-(k-1)< j\le 1}h^{1-j}\Op_h(e_j)$ and $F_k\in \Psi^{1-k}_{\tangent,h}$. Then we set $e_{-(k-1)}=-\frac{1}{2}f_ke_1^{-1}$ with $f_k=\sigma(F_k)$.
		\begin{multline*}
			X_+\left(hD_{x^1}+h \tilde{a}-\rmi E_{+,k+1}\right)\left(hD_{x^1}+\rmi E_{+,k+1}\right)
			\\
			=X_+\left(hD_{x^1}+h \tilde{a}-\rmi E_{+,k}-\rmi h^k\Op_h\left(e_{-\left(k-1\right)}\right)\right)\left(hD_{x^1}+\rmi E_{+,k}+\rmi h^k\Op_h\left(e_{-\left(k-1\right)}\right)\right)
			\\
			=X_+\Big( P(\omega;g)+h^kX_+F_k+h^k\left(E_{+,k}\Op_h\left(e_{-\left(k-1\right)}\right)+\Op_h\left(e_{-\left(k-1\right)}\right)E_{+,k}\right)
			\\+\rmi h^k\left[hD_{x^1},\Op_h\left(e_{-\left(k-1\right)}\right)\right] \Big)+O(h^{k+1})_{\Psi^{-k-1}_{\tangent,h}}
			=X_+ P(\omega;g)+h^{k+1}X_+F_{k+1},
		\end{multline*}
		which proves the $(k+1)$-th induction step and hence completes the proof.
	\end{proof}
	When $\omega_0<0$, Lemma \ref{Pdecompmicrolocal}  gives a full factorization for $P(\omega;g)$. 
	\begin{cor}
		\label{Pdecomp}
		Let $\omega_0<0$ and suppose that $|\omega-\omega_0|<Ch$. Consider operator $P(\omega;g,L)$. Then $P$ is strongly elliptic and we have, in the boundary normal coordinates,
		\begin{equation}
			\label{Laplace factorisation}
			P(\omega;g)=\left(hD_{x^1}+h \tilde{a}-\rmi E_{\pm}(x,hD_{x'})\right)\left(hD_{x^1}+\rmi E_{\pm}(x,hD_{x'})\right)+O(h^\infty)_{ \Psi^{-\infty}_{\tangent,h}},
		\end{equation}
		where $E_{\pm}$ is a first order semiclassical operator whose principal symbol can be chosen as $\sigma(E_\pm)=\pm\sqrt{|\xi'|_{g}^2-\omega_0}$. 
	\end{cor}
	\begin{proof}
		The proof follows immediately from Proposition \ref{Pdecompmicrolocal} with $X=I$. 
	\end{proof}
    \subsection{Energy estimates for first order operators}
    \label{Energy estimates for first order operators}
	The first estimate is the following basic energy estimate, which can be applied to $E_\pm$.
	\begin{lem}
		\label{Hyperbolic estimate lemma}
		Let $\Lambda\in \Psi^1_{\tangent,h}(\Omega)$ with $\Im\sigma(\Lambda)\le 0$. Then, for all $s\in \R$, and $t_0<t_1$, we have
		\begin{equation*}
			\label{Hyperbolic estimate eqn}
			\|v|_{x^1=t_0}\|_{H^s_h(\partial\Omega)}\le C \left(h^{-1} \|(hD_{x^1}-\Lambda )v\|_{L^2\left((t_0,t_1); H^s_h(\partial \Omega)\right)}+\|v\|_{L^2\left((t_0,t_1); H^s_h(\partial \Omega)\right)}\right).
		\end{equation*}
	\end{lem}
	\begin{proof}
		Let $A=\langle hD' \rangle^s$. We work in the boundary normal coordinates as before and start with the derivative of $\|v(x^1,\cdot)\|_{L^2(\R^{d-1})}$ in $x^1$. Omitting arguments $x^1$ and $\R^{d-1}$, we have
		\begin{multline}
			\label{dA eqn}
			\frac{h}{2} \partial_{x^1}\|Av\|_{L^2}^2=-\Im\langle hD_{x^1} Av, Av \rangle
			=-\Im\langle (hD_{x^1} -\Lambda)Av, Av \rangle-\Im\langle \Lambda Av, Av \rangle
			\\
			=-\Im\langle A(hD_{x^1} -\Lambda)v, Av \rangle-\Im\left\langle \left(\Lambda +[\Lambda,A]A^{-1}\right)Av, Av \right\rangle
			\\
			\ge -\|A(hD_{x^1} -\Lambda)v\|_{L^2}\|AXv\|_{L^2}-Ch \|Av\|_{L^2}^2
			\\
			\ge -h^{-1}\|A(hD_{x^1} +\Lambda)v\|_{L^2}^2-Ch \|Av\|_{L^2}^2,
		\end{multline}
		where G\aa rding's inequality is used for $\Im\left\langle \left(\Lambda +[\Lambda,A]A^{-1}\right)Av, Av \right\rangle$. In other words,
		\begin{equation*}
			\partial_{x^1}\|v\|_{H^s_h}^2\ge -C\left(h^{-2}\|(hD_{x^1} +\Lambda)v\|_{H^s_h}^2 +\|v\|_{H^s_h}^2\right).
		\end{equation*}
		Let $t_-<t_0$ and $\varphi \in C^\infty_c(t_-,t_2)$ with $\varphi\ge 0$ and $\varphi=1$ in a neighbourhood of $x^1=t_0$. Then we have
		\begin{multline*}
			\|v(t_0)\|_{H^s_h}^2=-\int_{t_0}^\infty \partial_{\blue{\tau}}\left( \varphi(\blue{\tau})\|v(\blue{\tau})\|_{H^s_h}^2\right) \der \blue{\tau}
			\\
			\le C \int_{t_0}^{t_1}\|v(\blue{\tau})\|_{H^s_h}^2 \der \blue{\tau}- \int_{t_0}^\infty \varphi(\blue{\tau}) \partial_{\blue{\tau}}\|v(\blue{\tau})\|_{H^s_h}^2 \der \blue{\tau}
			\\
			\le C \left(h^{-2}\|(hD_{x^1} -\Lambda)v(\blue{\tau})\|_{L^2\left((t_0,t_1)_{\blue{\tau}}; H^s_h(\partial \Omega)\right)}^2+\|v(\blue{\tau})\|_{L^2\left((t_0,t_1)_{\blue{\tau}}; H^s_h(\partial \Omega)\right)}^2 \right).
		\end{multline*}
	\end{proof}

    Our next estimate allows us to both microlocalize and work in higher regularity than Lemma~\ref{Hyperbolic estimate lemma}. The method for microlocalization was communicated from~\cite{GaLe:25}.
	\begin{lem}
		\label{Elliptic estimate lemma}
		Let $\e>0$, $\Lambda\in \Psi^1_{\tangent, h}(\Omega)$, $X,\tilde X\in \Psi^0_{\tangent,h}(\Omega)$ such that $\WF(X)\subset \Ellip(\tilde X)$ and $\WF(X)\subset \{(x,\xi'):\Im\sigma(\Lambda)(x,\xi')< -\e\langle \xi'\rangle \}$. Also, let $s\in \R$ and $0\leq t_0<t_1<t_2$. Then, there exists $h_0,c,C>0$ such that
		\begin{equation*}
			\begin{aligned}
				\label{Elliptic estimate eqn}
				&\|Xv(t_0)\|_{H^s_h(\partial \Omega)} +c h^{-\frac{1}{2}}\|Xv\|_{L^2\left((t_0,t_1);H_h^{s+\frac{1}{2}}(\partial \Omega)\right)}
				\\
				&\le Ch^{-\frac{1}{2}}\|\tilde{X}(hD_{x^1}-\Lambda)v\|_{L^2\left((t_0,t_2);H_h^{s-\frac{1}{2}}(\partial \Omega)\right)}+h^N\| v\|_{L^2\left((t_0,t_2);H_h^{-N}(\partial \Omega)\right)}\\
				&+h^N\| (hD_{x^1}-\Lambda)v\|_{L^2\left((t_0,t_2);H_h^{-N}(\partial \Omega)\right)}
			\end{aligned}
		\end{equation*}
		for all $0<h<h_0$.
	\end{lem}
	\begin{proof}
		Let $A=\langle hD' \rangle^s$. Exactly as in~\eqref{dA eqn}, one has
		\begin{equation}
			\label{dAX eqn}
			\frac{h}{2} \partial_{x^1}\|AXv\|_{L^2(\partial \Omega)}^2
			=-\Im\langle A(hD_{x^1} -\Lambda)Xv, AXv \rangle-\Im\langle \left(\Lambda +[\Lambda,A]A^{-1}\right)AXv, AXv \rangle.
		\end{equation}

		
		By the microlocal G\aa rding inequality \cite[Proposition E.34]{DyZw:19}, one has
		\begin{equation*}
			\Im\langle \Lambda AXv, AXv \rangle\le -\e \|Xv\|_{H^{s+\frac{1}{2}}_h(\partial \Omega)}^2+h^N\|v\|^2_{H^{-N}_h(\partial \Omega)}.
		\end{equation*}
		Therefore, plugging into \eqref{dAX eqn}, we arrive at
		\begin{equation}
			\label{e:plane}
			\frac{h}{2} \partial_{x^1}\|AXv\|_{L^2(\partial \Omega)}^2
			\ge - \|(hD_{x^1} -\Lambda)Xv\|_{H_h^{s-\frac{1}{2}}(\partial \Omega)}^2
			+ c\|Xv\|_{H_h^{s+\frac{1}{2}}(\partial \Omega)}^2-h^N\|v\|^2_{H^{-N}_h(\partial \Omega)}.
		\end{equation}
		
		We now claim that if $t_0<t_1<t_2$, and $X'\in \Psi^0_{\tangent,h}(\Omega)$ with $\WF(X)\subset \Ellip(X')$,  then, 
		\begin{equation}
			\begin{aligned}
				\label{e:ellip1}
				&\|Xv(t_0)\|_{H^s_h(\partial \Omega)}^2+c h^{-1}\|Xv\|_{L^2\left((t_0,t_1);H_h^{s+\frac{1}{2}}(\partial \Omega)\right)}^2
				\\
				&\le C h^{-1}\|X(hD_{x^1} -\Lambda)v\|_{L^2\left((t_0,t_2);H_h^{s-\frac{1}{2}}(\partial \Omega)\right)}^2
				+C\|X'v\|_{L^2\left((t_0,t_2);H_h^{s-\frac{1}{2}}(\partial \Omega)\right)}^2\\
				&\qquad +h^N\| v\|_{L^2\left((t_0,t_2);H_h^{-N}(\partial \Omega)\right)}.
			\end{aligned}
		\end{equation}
		
		To prove this, let $t_-<t_0$ and $\varphi\in C^\infty((t_-,t_2);[0,1])$ with $\varphi\equiv 1$ on $[t_0,t_1]$. Then by~\eqref{e:plane} with $v$, we have
		\begin{align*}
			\|Xv(t_0)\|_{H^s_h(\partial \Omega)}^2&=
			-\int_{t_0}^\infty \partial_{\tau }\left(\varphi(\tau) \| Xv(\tau )\|_{H^s_h(\partial \Omega)}^2\right) \der \tau
			\\
			&\le C h^{-1}\left(\|X(hD_{x^1} -\Lambda)v\|_{L^2\left((t_0,t_2);H_h^{s-\frac{1}{2}}(\partial \Omega)\right)}^2-c\|Xv\|_{L^2\left((t_0,t_2);H_h^{s+\frac{1}{2}}(\partial \Omega)\right)}^2\right)
			\\
			&+C\| Xv\|_{L^2\left((t_1,t_2);H_h^{s}(\partial \Omega)\right)}^2\\
			&+ Ch^{-1}\|[hD_{x^1}-\Lambda,X]v\|_{L^2\left((t_0,t_2);H_h^{s-\frac{1}{2}}(\partial\Omega)\right)}^2+Ch^N\| v\|_{L^2\left((t_0,t_2);H_h^{-N}(\partial \Omega)\right)},
		\end{align*}
		which implies~\eqref{e:ellip1} since $\| Xv\|_{L^2\left((t_1,t_2);H_h^{s}(\partial \Omega)\right)}^2$ can be absorbed by $\|Xv\|_{L^2\left((t_0,t_2);H_h^{s+\frac{1}{2}}(\partial \Omega)\right)}^2$ for sufficiently small $h$ and
		\begin{multline*}
			Ch^{-1}\|[hD_{x^1}-\Lambda,X]v\|_{L^2\left((t_0,t_2);H_h^{s-\frac{1}{2}}(\partial\Omega)\right)}^2+\| Xv\|_{L^2\left((t_0,t_2);H_h^{s-\frac{1}{2}}(\partial \Omega)\right)}^2\\\leq C\| X'v\|_{L^2\left((t_0,t_2);H_h^{s-\frac{1}{2}}(\partial \Omega)\right)}^2+Ch^N\| v\|_{L^2\left((t_0,t_2);H_h^{-N}(\partial \Omega)\right)}.
		\end{multline*}
		
		Now, we will prove by induction that for $t_0<t_1<t_2<t_3$, and $X'\in \Psi^0_{\tangent,h}(\Omega)$ with $\WF(X)\subset \Ellip(X')$,
		\begin{equation}
			\begin{aligned}
				\label{e:ellip2}
				&\|Xv(t_0)\|_{H^s_h(\partial \Omega)}^2+c h^{-1}\|Xv\|_{L^2\left((t_0,t_1);H_h^{s+\frac{1}{2}}(\partial \Omega)\right)}^2
				\\
				&\le C h^{-1}\|X'(hD_{x^1} -\Lambda)v\|_{L^2\left((t_0,t_3);H_h^{s-\frac{1}{2}}(\partial \Omega)\right)}^2
				+Ch^j\|X'v\|_{L^2\left((t_0,t_2);H_h^{s-j-\frac{1}{2}}(\partial \Omega)\right)}^2\\
				&\qquad +h^N\| v\|_{L^2\left((t_0,t_3);H_h^{-N}(\partial \Omega)\right)}+h^N\| (hD_{x^1}-\Lambda)v\|_{L^2\left((t_0,t_3);H_h^{-N}(\partial \Omega)\right)}.
			\end{aligned}
		\end{equation}
		By~\eqref{e:ellip1} we have~\eqref{e:ellip2} with $j=0$. 
		Suppose that~\eqref{e:ellip2} holds for some $j\geq 0$. Let $t_0<t_1<t_2'<t_2<t_3$ and $X'',X'\in \Psi^0_{\tangent,h}(\Omega)$ with $\WF(X)\subset \Ellip(X'')$ and $\WF(x'')\subset\Ellip(X')$. Then, by the induction hypothesis~\eqref{e:ellip2} holds. 
		\begin{equation}
			\begin{aligned}
				\label{e:ellip3}
				&\|Xv(t_0)\|_{H^s_h(\partial \Omega)}^2+c h^{-1}\|Xv\|_{L^2\left((t_0,t_1);H_h^{s+\frac{1}{2}}(\partial \Omega)\right)}^2
				\\
				&\le C h^{-1}\|X''(hD_{x^1} -\Lambda)v\|_{L^2\left((t_0,t_3);H_h^{s-\frac{1}{2}}(\partial \Omega)\right)}^2
				+Ch^j\|X''v\|_{L^2\left((t_0,t_2');H_h^{s-j-\frac{1}{2}}(\partial \Omega)\right)}^2\\
				&\qquad +h^N\| v\|_{L^2\left((t_0,t_3);H_h^{-N}(\partial \Omega)\right)}+h^N\| (hD_{x^1}-\Lambda)v\|_{L^2\left((t_0,t_3);H_h^{-N}(\partial \Omega)\right)}.
			\end{aligned}
		\end{equation}
		By~\eqref{e:ellip1} with $s$ replaced by $s-j$ and $(t_0,t_1,t_2)$ replaced by $(t_0,t_2',t_2)$, and $(X,X')$ replaced by $(X'',X')$,  we have 
		\begin{equation*}
			\begin{aligned}
				&\|X''v\|_{L^2\left((t_0,t_2');H_h^{s-j-\frac{1}{2}}(\partial \Omega)\right)}^2\\
				&\le C \|X''(hD_{x^1} -\Lambda)v\|_{L^2\left((t_0,t_2);H_h^{s-j-\frac{1}{2}}(\partial \Omega)\right)}^2
				+Ch\|X'v\|_{L^2\left((t_0,t_2);H_h^{s-j-\frac{1}{2}}(\partial \Omega)\right)}^2\\
				&\qquad +h^N\| v\|_{L^2\left((t_0,t_2);H_h^{-N}(\partial \Omega)\right)}.
			\end{aligned}
		\end{equation*}
		Using this in~\eqref{e:ellip3} then implies~\eqref{e:ellip2} with $j$ replaced by $j+1$ and hence completes the proof of the lemma.
	\end{proof}

    We also need an estimate analogous to Lemma~\ref{Elliptic estimate lemma} for estimating $v$ in the interior.
	\begin{lem}
		\label{l:lopa}
		Let $\e>0$, $\Lambda\in \Psi^1_{\tangent,h}(\Omega)$, $X,\tilde X\in \Psi^0_{\tangent,h}(\Omega)$ such that $\WF(X)\subset \Ellip(\tilde X)$ and $\WF(X)\subset \{(x,\xi'):\Im\sigma(\Lambda)(x,\xi')> \e\langle \xi'\rangle \}$. Also, let $s\in \R$ and $0\leq t_0<t_1<t_2$. Then, there exists $h_0,c,C>0$ such that
		\begin{equation}
			\begin{aligned}
				\label{Elliptic estimate eqna}
				& \|Xv\|_{L^2\left((t_0,t_1);H_h^{s+\frac{1}{2}}(\partial \Omega)\right)}\\
				&\le C\|\tilde{X}(hD_{x^1}-\Lambda)v\|_{L^2\left((t_0,t_2);H_h^{s-\frac{1}{2}}(\partial \Omega)\right)}+ch^{\frac{1}{2}}\|\tilde{X}v(t_0)\|_{H^s_h(\partial \Omega)}\\
				& \qquad+h^N\| v\|_{L^2\left((t_0,t_2);H_h^{-N}(\partial \Omega)\right)}+h^N\| (hD_{x^1}-\Lambda)v\|_{L^2\left((t_0,t_2);H_h^{-N}(\partial \Omega)\right)}\\
				& \qquad+h^N\| v(t_0)\|_{H_h^{-N}(\partial \Omega)},
			\end{aligned}
		\end{equation}
		for all $0<h<h_0$.
	\end{lem}
	\begin{proof}Let $A=\langle hD' \rangle^s$. We start again from~\eqref{dAX eqn} and use the
		%
		microlocal G\aa rding inequality \cite[Proposition E.34]{DyZw:19} to obtain
		\begin{equation*}
			\Im\langle \Lambda AXv, AXv \rangle\ge \e \|Xv\|_{H^{s+\frac{1}{2}}_h(\partial \Omega)}^2-h^N\|v\|^2_{H^{-N}_h(\partial \Omega)}.
		\end{equation*}
		
		Therefore, plugging into \eqref{dAX eqn}, we arrive at
		\begin{equation}
			\label{e:plane 1}
			\frac{h}{2} \partial_{x^1}\|AXv\|_{L^2(\partial \Omega)}^2
			\le \|(hD_{x^1} -\Lambda)Xv\|_{H_h^{s-\frac{1}{2}}(\partial \Omega)}^2
			- c\|Xv\|_{H_h^{s+\frac{1}{2}}(\partial \Omega)}^2+h^N\|v\|^2_{H^{-N}_h(\partial \Omega)}.
		\end{equation}
		
		We now claim that if $t_0<t_1<t_2$, and $X'\in \Psi^0_{\tangent,h}(\Omega)$ with $\WF(X)\subset \Ellip(X')$,  then, 
		\begin{equation}
			\begin{aligned}
				\label{e:ellip1a}
				&\|Xv\|_{L^2\left((t_0,t_1);H_h^{s+\frac{1}{2}}(\partial \Omega)\right)}^2\\
				&\le C \|X(hD_{x^1} -\Lambda)v\|_{L^2\left((t_0,t_2);H_h^{s-\frac{1}{2}}(\partial \Omega)\right)}^2 +h\|Xv(t_0)\|_{H^s_h(\partial \Omega)}^2\\
				&\qquad+Ch\|X'v\|_{L^2\left((t_0,t_2);H_h^{s-\frac{1}{2}}(\partial \Omega)\right)}^2+h^N\| v\|_{L^2\left((t_0,t_2);H_h^{-N}(\partial \Omega)\right)}.
			\end{aligned}
		\end{equation}
		
		As in previous Lemma, let $t_-<t_0$ and $\varphi\in C^\infty((t_-,t_2);[0,1])$ with $\varphi\equiv 1$ on $[t_0,t_1]$. Then by~\eqref{e:plane 1} with $v$, we have
		\begin{align*}
			\|Xv(t_0)\|_{H^s_h(\partial \Omega)}^2&=
			-\int_{t_0}^\infty \partial_{\tau }\left(\varphi(\tau) \| Xv(\tau )\|_{H^s_h(\partial \Omega)}^2\right) \der \tau
			\\
			&\ge -C h^{-1}\left(\|X(hD_{x^1} -\Lambda)v\|_{L^2\left((t_0,t_2);H_h^{s-\frac{1}{2}}(\partial \Omega)\right)}^2-c\|Xv\|_{L^2\left((t_0,t_2);H_h^{s+\frac{1}{2}}(\partial \Omega)\right)}^2\right)
			\\
			&-C\| Xv\|_{L^2\left((t_1,t_2);H_h^{s}(\partial \Omega)\right)}^2\\
			&- Ch^{-1}\|[hD_{x^1}-\Lambda,X]v\|_{L^2\left((t_0,t_2);H_h^{s-\frac{1}{2}}(\partial\Omega)\right)}^2-Ch^N\| v\|_{L^2\left((t_0,t_2);H_h^{-N}(\partial \Omega)\right)},
		\end{align*}
		which implies~\eqref{e:ellip1a} since $\| Xv\|_{L^2\left((t_1,t_2);H_h^{s}(\partial \Omega)\right)}^2$ can be absorbed by $\|Xv\|_{L^2\left((t_0,t_1);H_h^{s+\frac{1}{2}}(\partial \Omega)\right)}^2$ for sufficiently small $h$ and
		\begin{multline*}
			Ch^{-1}\|[hD_{x^1}-\Lambda,X]v\|_{L^2\left((t_0,t_2);H_h^{s-\frac{1}{2}}(\partial\Omega)\right)}^2+\| Xv\|_{L^2\left((t_0,t_2);H_h^{s-\frac{1}{2}}(\partial \Omega)\right)}^2\\\leq C\| X'v\|_{L^2\left((t_0,t_2);H_h^{s-\frac{1}{2}}(\partial \Omega)\right)}^2+Ch^N\| v\|_{L^2\left((t_0,t_2);H_h^{-N}(\partial \Omega)\right)}.
		\end{multline*}
		
		Now, we will prove by induction that for $t_0<t_1<t_2<t_3$, and $X'\in \Psi^0_{\tangent,h}(\Omega)$ with $\WF(X)\subset \Ellip(X')$,
		\begin{equation}
			\begin{aligned}
				\label{e:ellip2a}
				&\|Xv\|_{L^2\left((t_0,t_1);H_h^{s+\frac{1}{2}}(\partial \Omega)\right)}^2\le C \|X'(hD_{x^1} -\Lambda)v\|_{L^2\left((t_0,t_3);H_h^{s-\frac{1}{2}}(\partial \Omega)\right)}^2
				\\
				&\qquad\qquad\qquad+Ch\|X'v(t_0)\|_{H^s_h(\partial \Omega)}^2+Ch^{j+1}\|X'v\|_{L^2\left((t_0,t_2);H_h^{s-j-\frac{1}{2}}(\partial \Omega)\right)}^2\\
				&\qquad\qquad\qquad +h^N\| v\|_{L^2\left((t_0,t_3);H_h^{-N}(\partial \Omega)\right)}+h^N\| (hD_{x^1}-\Lambda)v\|_{L^2\left((t_0,t_3);H_h^{-N}(\partial \Omega)\right)}\\
				&\qquad\qquad\qquad +h^N\| v(t_0)\|_{H_h^{-N}(\partial \Omega)}.
			\end{aligned}
		\end{equation}
		By~\eqref{e:ellip1a} we have~\eqref{e:ellip2a} with $j=0$. 
		Suppose that~\eqref{e:ellip2a} holds for some $j\geq 0$. Let $t_0<t_1<t_2'<t_2<t_3$ and $X'',X'\in \Psi^0_{\tangent,h}(\Omega)$ with $\WF(X)\subset \Ellip(X'')$ and $\WF(X'')\subset\Ellip(X')$. Then, by the induction hypothesis~\eqref{e:ellip2a} holds. 
		\begin{equation}
			\begin{aligned}
				\label{e:ellip3a}
				&\|Xv\|_{L^2\left((t_0,t_1);H_h^{s+\frac{1}{2}}(\partial \Omega)\right)}^2\le C \|X''(hD_{x^1} -\Lambda)v\|_{L^2\left((t_0,t_3);H_h^{s-\frac{1}{2}}(\partial \Omega)\right)}^2
				\\
				&\qquad\qquad\qquad +h\|X''v(t_0)\|_{H^s_h(\partial \Omega)}^2+Ch^{j+1}\|X''v\|_{L^2\left((t_0,t_2');H_h^{s-j-\frac{1}{2}}(\partial \Omega)\right)}^2\\
				&\qquad\qquad\qquad+h^N\| v\|_{L^2\left((t_0,t_3);H_h^{-N}(\partial \Omega)\right)}+h^N\| (hD_{x^1}-\Lambda)v\|_{L^2\left((t_0,t_3);H_h^{-N}(\partial \Omega)\right)}\\
				&\qquad\qquad\qquad+h^N\| v(t_0)\|_{H_h^{-N}(\partial \Omega)}.
			\end{aligned}
		\end{equation}
		By~\eqref{e:ellip1a} with $s$ replaced by $s-j-1$ and $(t_0,t_1,t_2)$ replaced by $(t_0,t_2',t_2)$, and $(X,X')$ replaced by $(X'',X')$,  we have 
		\begin{equation*}
			\begin{aligned}
				&\|X''v\|_{L^2\left((t_0,t_2');H_h^{s-j-\frac{1}{2}}(\partial \Omega)\right)}^2\\
				&\le C \|X''(hD_{x^1} -\Lambda)v\|_{L^2\left((t_0,t_2);H_h^{s-j-\frac{3}{2}}(\partial \Omega)\right)}^2 +h\|X''v(t_0)\|_{H^{s-j-1}_h(\partial \Omega)}^2\\
				&\qquad+Ch\|X'v\|_{L^2\left((t_0,t_2);H_h^{s-j-\frac{3}{2}}(\partial \Omega)\right)}^2+h^N\| v\|_{L^2\left((t_0,t_2);H_h^{-N}(\partial \Omega)\right)}.
			\end{aligned}
		\end{equation*}
		Using this in~\eqref{e:ellip3a} then implies~\eqref{e:ellip2a} with $j$ replaced by $j+1$ and hence completes the proof of the lemma.
	\end{proof}

	\subsection{Estimates for the operator $P$}
	In this section, we use the factorisation, Lemma~\ref{Pdecompmicrolocal} together with the estimates from the previous subsection to prove estimates on solutions to $Pu=f$. 
	
    Let $E_\pm$ be the factorisation operators of $P$ defined in Lemma \ref{Pdecompmicrolocal}. We have the following microlocal estimates.
	\begin{lem}
		\label{DtN Hyperbolic estimate lemma}
		Let $X,\tilde X\in \Psi^0_{\tangent,h}(\Omega)$ such that $\WF(X)\subset \Ellip(\tilde X)$, then for $t_0<t_1$, we have
		\begin{multline*}
			\|X(hD_{x^1}+\rmi E_\pm )u|_{x^1=t_0}\|_{H^s_h(\partial\Omega)}
			\\
			\le C \left(h^{-1}\|XPu\|_{L^2\left((t_0,t_1); H^s_h(\partial \Omega)\right)}+ \|\tilde Xu\|_{H^1_h\left((t_0,t_1); H^s_h(\partial \Omega)\right)}\right)
			+h^N\|\tilde X u\|_{L^2\left((t_0,t_1);H_h^{-N}(\partial \Omega)\right)}.
		\end{multline*}
	\end{lem}
	\begin{proof}
		First note that
		\begin{multline}
			\label{DtN Hyperbolic estimate lemma eqn 1}
			(hD_{x^1}+ha-\rmi E_\pm)X(hD_{x^1}+\rmi E_\pm)u
			\\
			=XPu+([hD_{x^1}+ha-\rmi E_\pm,X])(hD_{x^1}+\rmi E_\pm)u+O(h^\infty)_{\Psi^{-\infty}_{\tangent,h}}(\tilde X u).
		\end{multline}
		Let $\Lambda=\rmi E_\pm-h a$, which satisfies the criteria in Lemma \ref{Hyperbolic estimate lemma}. Then by setting $v=X(hD_{x^1}+\rmi E_\pm)u$, we have
		$$
		\|v|_{x^1=t_0}\|_{H^s_h(\partial\Omega)}\le C \left(h^{-1} \|(hD_{x^1}-\Lambda )v\|_{L^2\left((t_0,t_1); H^s_h(\partial \Omega)\right)}+\|v\|_{L^2\left((t_0,t_1); H^s_h(\partial \Omega)\right)}\right).
		$$
		That is
		\begin{multline*}
			\|X(hD_{x^1}+\rmi E_\pm)u|_{x^1=t_0}\|_{H^s_h(\partial\Omega)}
			\\
			\le C \Big(h^{-1}\|XPu\|_{L^2\left((t_0,t_1); H^s_h(\partial \Omega)\right)}+ h^{-1}\|[hD_{x^1}+h a-\rmi E_\pm,X](hD_{x^1}+\rmi E_\pm)u\|_{L^2\left((t_0,t_1); H^s_h(\partial \Omega)\right)}
			\\
			+\|X(hD_{x^1}+\rmi E_\pm)u\|_{L^2\left((t_0,t_1); H^s_h(\partial \Omega)\right)}\Big)
			+h^N\|\tilde X u\|_{L^2\left((t_0,t_1);H_h^{-N}(\partial \Omega)\right)},
		\end{multline*}
		which proves the lemma.
	\end{proof}
	
	\begin{lem}
		\label{DtN Elliptic estimate lemma -}
		Let $\e>0$, $X,\tilde X\in \Psi^0_{\tangent,h}(\Omega)$ such that $\WF(X)\subset \Ellip(\tilde X)\subset \WF(\tilde X)\subset \{(x,\xi'):\Re\sigma(E_-)(x,\xi')< -\e \}$. Also, let $s\in \R$ and $t_0<t_1<t_2$. Then, there exists $h_0,c,C>0$ such that
		\begin{equation}
			\begin{aligned}
				\label{DtN Elliptic estimate eqn}
				&\|X(hD_{x^1}+\rmi E_-)u(t_0)\|_{H^s_h(\partial \Omega)} +c h^{-\frac{1}{2}}\|X(hD_{x^1}+\rmi E_-)u\|_{L^2\left((t_0,t_1);H_h^{s+\frac{1}{2}}(\partial \Omega)\right)}
				\\
				&\le Ch^{-\frac{1}{2}}\|\tilde XPu\|_{L^2\left((t_0,t_2);H_h^{s-\frac{1}{2}}(\partial \Omega)\right)}
				\\
				&\qquad+Ch^N\|u\|_{H^1_h\left((t_0,t_2);H_h^{-N}(\partial \Omega)\right)}+Ch^N\| Pu\|_{L^2\left((t_0,t_2);H_h^{-N}(\partial \Omega)\right)}
			\end{aligned}
		\end{equation}
		for all $0<h<h_0$. 
		If $\Re\sigma(E_-)(x,\xi')< 0$ on $[t_0,t_2]\times T^*\partial\Omega$, then we have $X=\tilde X=I$ and a better estimate
		\begin{equation}
			\label{DtN Elliptic estimate eqn better}
			\begin{aligned}
				&\|(hD_{x^1}+\rmi E_-)u(t_0)\|_{H^s_h(\partial \Omega)} +ch^{-\frac{1}{2}}\|(hD_{x^1}+\rmi E_-)u\|_{L^2\left((t_0,t_1);H_h^{s+\frac{1}{2}}(\partial \Omega)\right)}
				\\
				&\le Ch^{-\frac{1}{2}}\|Pu\|_{L^2\left((t_0,t_2);H_h^{s-\frac{1}{2}}(\partial \Omega)\right)}+Ch^N\|u\|_{L^2\left((t_0,t_2);H_h^{-N}(\partial \Omega)\right)}
			\end{aligned}
		\end{equation}
		for all $0<h<h_0$.
	\end{lem}
    \begin{rem}
    \label{Epm comparison}
    Note that Lemma \ref{DtN Hyperbolic estimate lemma} says that both $E_+$ and $E_-$ are good approximation to $h\partial_{x^1}$ if we allow some $H^1_h$-error of $u$. Lemma \ref{DtN Elliptic estimate lemma -} says $E_-$ is a better approximation than $E_+$ as the $H^1_h$-error of $u$ is reduced to $h^\infty$ small for $E_-$.
    \end{rem}

	An immediate corollary of Lemma \ref{DtN Elliptic estimate lemma -} is the following elliptic estimate.
	\begin{cor}
		\label{DtN Elliptic estimate cor -}
		Let $\omega_0\in\mathbb{R}$ and suppose that $|\omega_0-\omega|<Ch$ and $P$ be given as in \eqref{P(w;g,L)}, Then, there exists $h_0,C,\omega'>0$ such that for $0<h<h_0$,
		\begin{equation}
			\label{hDx_1 estimate 1}
			\begin{aligned}
				&\|hD_{x^1}u(t_0)\|_{H^s_h(\partial \Omega)}
				\le Ch^{-\frac{1}{2}}\|(P+\omega')u\|_{L^2\left((t_0,t_2);H_h^{s-\frac{1}{2}}(\partial \Omega)\right)}
				\\
				&\qquad+ \|u(t_0)\|_{H^{s+1}_h(\partial \Omega)}+Ch^N\|u\|_{L^2\left((t_0,t_2);H_h^{-N}(\partial \Omega)\right)},
			\end{aligned}
		\end{equation}
		and the following estimate holds,
		\begin{equation*}
			\label{hDx_1 estimate 2}
			\begin{aligned}
				&\|hD_{x^1}u\|_{L^2\left((t_0,t_1);H_h^{s+\frac{1}{2}}(\partial \Omega)\right)}
				\le C\|(P+\omega')u\|_{L^2\left((t_0,t_2);H_h^{s-\frac{1}{2}}(\partial \Omega)\right)}
				\\
				&\qquad+\|u\|_{L^2\left((t_0,t_1);H_h^{s+\frac{3}{2}}(\partial \Omega)\right)}+Ch^N\|u\|_{L^2\left((t_0,t_2);H_h^{-N}(\partial \Omega)\right)}.
			\end{aligned}
		\end{equation*}
	\end{cor}
	\begin{proof}[Proof of Corollary \ref{DtN Elliptic estimate cor -}]
		We only need to choose $\omega'$ large enough such that $P+\omega'$ makes 
        $$
        \Re\sigma(E_-)(x,\xi')< 0 \quad \text{on} \quad [t_0,t_2]\times T^*\partial\Omega,
        $$ 
        then estimate \eqref{DtN Elliptic estimate eqn better} holds for $E_-(\omega')$ and $P+\omega'$.
	\end{proof}
	\begin{proof}[Proof of Lemma \ref{DtN Elliptic estimate lemma -}]
		Let $t_0<t_1<t_2'<t_2$. We consider the $E_-$ factorisation of \eqref{DtN Hyperbolic estimate lemma eqn 1} in Lemma \ref{DtN Hyperbolic estimate lemma}. In other words, we set $\Lambda=\rmi E_--ha$ and substitute $v=(hD_{x^1}+\rmi E_-)u$. Let $X'\in \Psi_{\tangent,h}^0(\Omega)$ with $\WF(X)\subset\Ellip (X_1)$ and $\WF(X_1)\subset \Ellip(\tilde{X})$. Then replacing $(X,\tilde{X})$ by $(X,X_1)$ in Lemma \ref{Elliptic estimate lemma}, we obtain
		\begin{equation}
			\label{e:george}
			\begin{aligned}
				&\|X(hD_{x^1}+\rmi E_-)u(t_0)\|_{H^s_h(\partial \Omega)} +ch^{-\frac{1}{2}}\|X(hD_{x^1}+\rmi E_-)u\|_{L^2\left((t_0,t_1);H_h^{s+\frac{1}{2}}(\partial \Omega)\right)}\\
				&
				\le Ch^{-\frac{1}{2}}\|X_1Pu\|_{L^2\left((t_0,t_2');H_h^{s-\frac{1}{2}}(\partial \Omega)\right)}+Ch^N\|X_1(hD_{x^1}+\rmi E_-)u \|_{L^2\left((t_0,t_2');H_h^{-N}(\partial \Omega)\right)}\\
				&\quad+
				Ch^N\|Pu \|_{L^2\left((t_0,t_2');H_h^{-N}(\partial \Omega)\right)}.
			\end{aligned}
		\end{equation}
		To obtain~\eqref{DtN Elliptic estimate eqn} we simply estimate
		\begin{align*}
			&h^{-\frac{1}{2}}\|X_1Pu\|_{L^2\left((t_0,t_2');H_h^{s-\frac{1}{2}}(\partial \Omega)\right)}+Ch^N\|X_1(hD_{x^1}+\rmi E_-)u \|_{L^2\left((t_0,t_2');H_h^{-N}(\partial \Omega)\right)}\\
			&\le Ch^{-\frac{1}{2}}\|\tilde{X}Pu\|_{L^2\left((t_0,t_2');H_h^{s-\frac{1}{2}}(\partial \Omega)\right)}+Ch^N\|\tilde{X} u\|_{H_h^1\left((t_0,t_2');H_h^{-N}(\partial \Omega)\right)}\\
			&\quad+Ch^N\|Pu \|_{L^2\left((t_0,t_2');H_h^{-N}(\partial \Omega)\right)}.
		\end{align*}

		Now, we are left with the case $\Re\sigma(E_-)(x,\xi')< 0$ on $T^*\Omega$. Note that by setting $X=X_1=I$ in \eqref{e:george}, one has
		\begin{multline}
			\label{DtN Elliptic estimate simple}
			\|(hD_{x^1}+\rmi E_-)u(t_0)\|_{H^s_h(\partial \Omega)} +ch^{-\frac{1}{2}}\|(hD_{x^1}+\rmi E_-)u\|_{L^2\left((t_0,t_1);H_h^{s+\frac{1}{2}}(\partial \Omega)\right)}
			\\
			\le Ch^{-\frac{1}{2}}\|Pu\|_{L^2\left((t_0,t_2');H_h^{s-\frac{1}{2}}(\partial \Omega)\right)}
			+h^N\| u\|_{L^2\left((t_0,t_2');H_h^{-N}(\partial \Omega)\right)}
			\\
			+h^N\| (hD_{x^1}+\rmi E_-)u\|_{L^2\left((t_0,t_2');H_h^{-N}(\partial \Omega)\right)}.
		\end{multline}
		Since $\Re\sigma(E_-)(x,\xi')< 0$ on $[t_0,t_2]\times T^*\partial\Omega$ implies $P$ is elliptic in the neighborhood of $[t_1,t_2']\times \partial \Omega$, this means 
		$$
		\| (hD_{x^1}+\rmi E_-)u\|_{H_h^1\left((t_1,t_2');H_h^{-N}(\partial \Omega)\right)}\le C\|Pu\|_{L^2\left((t_0,t_2);H_h^{-N}(\partial \Omega)\right)}+\|u\|_{L^2\left((t_0,t_2);H_h^{-N}(\partial \Omega)\right)},
		$$
		which can be applied to \eqref{DtN Elliptic estimate simple} to complete the proof.
	\end{proof}
	
	The application of Lemma \ref{l:lopa} to $E_+$ is given by the following.
	
	\begin{lem}
		\label{DtN Elliptic estimate lemma +}
		Let $\e>0$, $X,\tilde X\in \Psi^0_{\tangent,h}(\Omega)$ such that $\WF(X)\subset \Ellip(\tilde X)\subset \WF(\tilde X)\subset \{(x,\xi'):\Re\sigma(E_+)(x,\xi')> \e \}$. Also, let $s\in \R$ and $t_0<t_1<t_2$. Then, there exists $h_0,c,C>0$ such that
		\begin{equation}
			\label{DtN Elliptic estimate eqn +}
			\begin{aligned}
				& \|X(hD_{x^1}+\rmi E_+)u\|_{L^2\left((t_0,t_1);H_h^{s+\frac{1}{2}}(\partial \Omega)\right)}\\
				&\le C\|\tilde{X}Pu\|_{L^2\left((t_0,t_2);H_h^{s-\frac{1}{2}}(\partial \Omega)\right)}+Ch^{\frac{1}{2}}\|\tilde{X}(hD_{x^1}+\rmi E_+)u(t_0)\|_{H^s_h(\partial \Omega)}\\
				& \qquad+h^N\| (hD_{x^1}+\rmi E_+)u\|_{L^2\left((t_0,t_2);H_h^{-N}(\partial \Omega)\right)}+h^N\| Pu\|_{L^2\left((t_0,t_2);H_h^{-N}(\partial \Omega)\right)}\\
				& \qquad+h^N\| (hD_{x^1}+\rmi E_+)u(t_0)\|_{H_h^{-N}(\partial \Omega)}.
			\end{aligned}
		\end{equation}
		for all $0<h<h_0$. 
		If $\Re\sigma(E_+)(x,\xi')> 0$ on $[t_0,t_2]\times T^*\partial\Omega$, then we have $X=\tilde X=I$ and a better estimate
		\begin{multline}
			\label{DtN Elliptic estimate eqn better +}
			\|(hD_{x^1}+\rmi E_+)u\|_{L^2\left((t_0,t_1);H_h^{s+\frac{1}{2}}(\partial \Omega)\right)}\\
			\\
			\le C\|Pu\|_{L^2\left((t_0,t_2);H_h^{s-\frac{1}{2}}(\partial \Omega)\right)}+Ch^{\frac{1}{2}}\|(hD_{x^1}+\rmi E_+)u(t_0)\|_{H^s_h(\partial \Omega)}
		\end{multline}
		for all $0<h<h_0$.
	\end{lem}
	
	\begin{proof}
		By setting $\Lambda=\rmi E_+-ha$ and $v=(hD_{x^1}+\rmi E_+)u$ in equation \eqref{Elliptic estimate eqna}, which gives equation \eqref{DtN Elliptic estimate eqn +}. Equation \eqref{DtN Elliptic estimate eqn better +} follows immediately from equation \eqref{DtN Elliptic estimate eqn +}.
	\end{proof}

	Combining Lemma~\ref{DtN Elliptic estimate lemma -} and Lemma~\ref{DtN Elliptic estimate lemma +} yields the following estimate.
	\begin{lem}
		\label{DtN microlocal Elliptic estimate lemma H1 -}
		Let $\e>0$, $B\in \Psi^1_h(\partial\Omega)$, $X,\tilde X\in \Psi^0_{\tangent,h}(\Omega)$ such that $\WF(X)\subset \Ellip(\tilde X)\subset \WF(\tilde X)\subset \{(x,\xi'):\Re\sigma(E_-)(x,\xi')< -\e \}$. Also, let $s\in \R$ and $t_0<t_1<t_2$. If $\Re\sigma(E_-)(x,\xi')< -\e$, then there exists $h_0,C>0$ such that
		\begin{equation}
			\label{DtN microlocal Elliptic estimate lemma -}
			\begin{aligned}
				& \|Xu\|_{H^1_h\left((t_0,t_1);H_h^{s}(\partial \Omega)\right)}+\|Xu\|_{L^2\left((t_0,t_1);H_h^{s+1}(\partial \Omega)\right)}\\
				&\le C\|\tilde XPu\|_{L^2\left((t_0,t_2);H_h^{s-1}(\partial \Omega)\right)}+Ch^{\frac{1}{2}}\|\tilde{X}u(t_0)\|_{H^{s+\frac{1}{2}}_h(\partial \Omega)}\\
				& \qquad+Ch^N\|u\|_{H^1_h((t_0,t_2);H_h^{-N}(\partial \Omega))}+Ch^N\| Pu\|_{L^2\left((t_0,t_2);H_h^{-N}(\partial \Omega)\right)}+Ch^N\|u(t_0)\|_{H_h^{-N}(\partial\Omega)}
			\end{aligned}
		\end{equation}
		for all $0<h<h_0$. 
	\end{lem}
	\begin{proof}
		Let $\WF(X)\subset \Ellip(X')\subset\WF(X')\subset \Ellip(\tilde X)\subset \WF(\tilde X)$. Combining \eqref{DtN Elliptic estimate eqn} and \eqref{DtN Elliptic estimate eqn +} yields
		\begin{equation*}
			\begin{aligned}
				&h^{\frac{1}{2}}\|X(hD_{x^1}+\rmi E_-)u(t_0)\|_{H^s_h(\partial \Omega)}+c\|X(hD_{x^1}+\rmi E_-)u\|_{L^2\left((t_0,t_1);H_h^{s+\frac{1}{2}}(\partial \Omega)\right)}
				\\
				&\le C\|X' Pu\|_{L^2\left((t_0,t_2);H_h^{s-\frac{1}{2}}(\partial \Omega)\right)}
				\\
				&\qquad+Ch^N\|u\|_{H^1_h((t_0,t_2);H_h^{-N}(\partial \Omega))}+Ch^N\| Pu\|_{L^2\left((t_0,t_2);H_h^{-N}(\partial \Omega)\right)}
			\end{aligned}
		\end{equation*}
		and
		\begin{equation*}
			\begin{aligned}
				& \|X(hD_{x^1}+\rmi E_+)u\|_{L^2\left((t_0,t_1);H_h^{s+\frac{1}{2}}(\partial \Omega)\right)}\\
				&\le C\|X'Pu\|_{L^2\left((t_0,t_2);H_h^{s-\frac{1}{2}}(\partial \Omega)\right)}+Ch^{\frac{1}{2}}\|X'(hD_{x^1}+\rmi E_+)u(t_0)\|_{H^s_h(\partial \Omega)}\\
				& \qquad+h^N\| (hD_{x^1}+\rmi E_+)u\|_{L^2\left((t_0,t_2);H_h^{-N}(\partial \Omega)\right)}+h^N\| Pu\|_{L^2\left((t_0,t_2);H_h^{-N}(\partial \Omega)\right)}\\
				& \qquad+h^N\| (hD_{x^1}+\rmi E_+)u(t_0)\|_{H_h^{-N}(\partial \Omega)}.
			\end{aligned}
		\end{equation*}
		With parallelogram law on Hilbert spaces, we have
		\begin{equation*}
			\begin{aligned}
				&\|XhD_{x^1}u\|_{L^2\left((t_0,t_1);H_h^{s+\frac{1}{2}}(\partial \Omega)\right)}+\|X(E_+-E_-)u\|_{L^2\left((t_0,t_1);H_h^{s+\frac{1}{2}}(\partial \Omega)\right)}
				\\
				&\le C\|X' Pu\|_{L^2\left((t_0,t_2);H_h^{s-\frac{1}{2}}(\partial \Omega)\right)}+Ch^{\frac{1}{2}}\|X'(hD_{x^1}+\rmi E_-)u(t_0)\|_{H^s_h(\partial \Omega)}
				\\
				&\qquad+Ch^{\frac{1}{2}}\|X'(E_+-E_-)u(t_0)\|_{H^s_h(\partial \Omega)}+Ch^N\|u\|_{H^1_h\left((t_0,t_2);H_h^{-N}(\partial \Omega)\right)}
				\\
				&\qquad+Ch^N\| Pu\|_{L^2\left((t_0,t_2);H_h^{-N}(\partial \Omega)\right)}+Ch^N\| (hD_{x^1}+\rmi E_+)u(t_0)\|_{H_h^{-N}(\partial \Omega)}.
			\end{aligned}
		\end{equation*}
		This shows
		\begin{equation*}
			\label{simple eqn 11}
			\begin{aligned}
				&\|hD_{x^1}Xu\|_{L^2\left((t_0,t_1);H_h^{s+\frac{1}{2}}(\partial \Omega)\right)}+\|Xu\|_{L^2\left((t_0,t_1);H_h^{s+\frac{3}{2}}(\partial \Omega)\right)}
				\\
				&\le C\|\tilde XPu\|_{L^2\left((t_0,t_2);H_h^{s-\frac{1}{2}}(\partial \Omega)\right)}+Ch^{\frac{1}{2}}\|\tilde X u(t_0)\|_{H^{s+1}_h(\partial \Omega)}
				\\
				&\qquad+Ch^N\|u\|_{H^1_h((t_0,t_2);H_h^{-N}(\partial \Omega))}+Ch^N\| Pu\|_{L^2\left((t_0,t_2);H_h^{-N}(\partial \Omega)\right)}
				\\
				& \qquad+Ch^N\|hD_{x^1}u(t_0)\|_{H_h^{-N}(\partial \Omega)}+Ch^N\|u(t_0)\|_{H_h^{-N}(\partial \Omega)}+h\|\tilde Xu\|_{L^2\left((t_0,t_1);H_h^{s+\frac{1}{2}}(\partial \Omega)\right)}.
			\end{aligned}
		\end{equation*}
		A similar bootstrap argument as used in the proof of Lemma \ref{Elliptic estimate lemma} will allow us to absorb the term $h\|\tilde Xu\|_{L^2\left((t_0,t_1);H_h^{s+\frac{1}{2}}(\partial \Omega)\right)}$. On the other hand, estimate \eqref{hDx_1 estimate 1} allows us to kill the term $Ch^N\|hD_{x^1}u(t_0)\|_{H_h^{-N}(\partial \Omega)}$. This proves the Lemma by replacing $s$ with $s-\frac{1}{2}$.
	\end{proof}
	
	A simple elliptic parametrix estimate then yields
	\begin{lem}
		Let $\e>0$, $B\in \Psi^1_h(\partial\Omega)$, $X,\tilde X\in \Psi^0_{\tangent,h}(\Omega)$ such that $\WF(X)\subset \Ellip(\tilde X)\subset \WF(\tilde X)\subset \{(x,\xi'):\Re\sigma(E)(x,\xi')< -\e \}\cap \{(x,\xi'):|\sigma(B)-\rmi \sigma(E)|(x,\xi')>0\}$. Also, let $s\in \R$ and $t_0<t_1<t_2$. Then, there exists $h_0,C>0$ such that
		\begin{equation}
			\label{DtN microlocal Elliptic estimate lemma - B}
			\begin{aligned}
				& \|Xu\|_{L^2\left((t_0,t_1);H_h^{s+1}(\partial \Omega)\right)}\\
				&\le C\|\tilde XPu\|_{L^2\left((t_0,t_2);H_h^{s-1}(\partial \Omega)\right)}+Ch^{\frac{1}{2}}\|\tilde{X}(hD_{x^1}+B)u(t_0)\|_{H^{s-\frac{1}{2}}_h(\partial \Omega)}\\
				& \qquad+Ch^N\|u\|_{H^1_h\left((t_0,t_2);H_h^{-N}(\partial \Omega)\right)}+Ch^N\| Pu\|_{L^2\left((t_0,t_2);H_h^{-N}(\partial \Omega)\right)}+Ch^N\|u(t_0)\|_{H_h^{-N}(\partial\Omega)}
			\end{aligned}
		\end{equation}
		for all $0<h<h_0$. 
	\end{lem}
	\begin{proof}
		Microlocal elliptic estimate says
		\begin{equation*}
			\begin{aligned}
				\|Xu(t_0)\|_{H^{s+\frac{1}{2}}_h(\partial \Omega)} &\le C\|X(B-\rmi E_-)u(t_0)\|_{H^{s-\frac{1}{2}}_h(\partial \Omega)}
				\\
				&\qquad +Ch\|X' u(t_0)\|_{H^{s-\frac{1}{2}}_h(\partial \Omega)}+Ch^N\|u(t_0)\|_{H^{-N}_h(\partial \Omega)}
				\\
				&\le  C\|X(hD_{x^1}+B)u(t_0)\|_{H^{s-\frac{1}{2}}_h(\partial \Omega)} +C\|X(hD_{x^1}+\rmi E_-)u(t_0)\|_{H^{s-\frac{1}{2}}_h(\partial \Omega)}
				\\
				&\qquad +Ch\|\tilde X u(t_0)\|_{H^{s-\frac{1}{2}}_h(\partial \Omega)}+Ch^N\|u(t_0)\|_{H^{-N}_h(\partial \Omega)}
				\\
				&\le Ch^{-\frac{1}{2}}\|\tilde XPu\|_{L^2\left((t_0,t_2);H_h^{s-1}(\partial \Omega)\right)}+C\|X(hD_{x^1}+B)u(t_0)\|_{H^{s-\frac{1}{2}}_h(\partial \Omega)}
				\\
				&\qquad+Ch^N\|u\|_{H^1_h((t_0,t_2);H_h^{-N}(\partial \Omega))}+Ch^N\| Pu\|_{L^2\left((t_0,t_2);H_h^{-N}(\partial \Omega)\right)}
				\\
				&\qquad +Ch\|\tilde X u(t_0)\|_{H^{s-\frac{1}{2}}_h(\partial \Omega)}+Ch^N\|u(t_0)\|_{H^{-N}_h(\partial \Omega)}.
			\end{aligned}
		\end{equation*}
		Using the bootstrapping argument as in Section \ref{Energy estimates for first order operators}, we can replace $Ch\|\tilde X u(t_0)\|_{H^{s-\frac{1}{2}}_h(\partial \Omega)}$ by $Ch^N\|u(t_0)\|_{H^{-N}_h(\partial \Omega)}$ and estimate \eqref{DtN microlocal Elliptic estimate lemma - B} follows immediately from \eqref{DtN microlocal Elliptic estimate lemma -}.
	\end{proof}

	Finally, we can combine the above estimates to obtain higher regularity in the normal variable. 
	\begin{lem}
		\label{l:gainNormal}
		Let $\e>0$, $B\in \Psi^1_h(\partial\Omega)$, $X,\tilde X\in \Psi^0_{\tangent,h}(\Omega)$ such that $\WF(X)\subset \Ellip(\tilde X)\subset \WF(\tilde X)\subset \{(x,\xi'):\Re\sigma(E)(x,\xi')< -\e \}\cap \{(x,\xi'):|\sigma(B)-\rmi \sigma(E)|(x,\xi')>0\}$. Also, let $k\in \N$, $k\ge 2$, $s\in \R$ and $t_0<t_1<t_2$. Then, there exists $h_0,C>0$ such that
		\begin{equation}
			\label{l:gainNormal eqn1}
			\begin{aligned}
				&\|Xu\|_{H_h^{k}\left((t_0,t_1);H_h^{s}\right)}+\|Xu\|_{H_h^{k-1}\left((t_0,t_1);H_h^{s+1}\right)}+\|Xu\|_{H_h^{k-2}\left((t_0,t_1);H_h^{s+2}\right)}\\
				&\qquad\leq C\sum_{j=0}^{k-2}\|\tilde XPu\|_{H_h^{k-2-j}\left((t_0,t_2);H_h^{s+j}(\partial \Omega)\right)}+Ch^{\frac{1}{2}}\|\tilde{X}u(t_0)\|_{H^{s+k-\frac{1}{2}}_h(\partial \Omega)}\\
				& \qquad+Ch^N\|u\|_{H^{1}_h\left((t_0,t_2);H_h^{-N}(\partial \Omega)\right)}+Ch^N\| Pu\|_{L^2\left((t_0,t_2);H_h^{-N}(\partial \Omega)\right)}+Ch^N\|u(t_0)\|_{H_h^{-N}(\partial \Omega)}
			\end{aligned}
		\end{equation}
		and
		\begin{equation}
			\label{l:gainNormal eqn2}
			\begin{aligned}
				&\|Xu\|_{H_h^{k}\left((t_0,t_1);H_h^{s}(\partial \Omega)\right)}+\|Xu\|_{H_h^{k-1}\left((t_0,t_1);H_h^{s+1}(\partial \Omega)\right)}+\|Xu\|_{H_h^{k-2}\left((t_0,t_1);H_h^{s+2}(\partial \Omega)\right)}\\
				&\qquad\leq C\sum_{j=0}^{k-2}\|\tilde XPu\|_{H_h^{k-2-j}\left((t_0,t_2);H_h^{s+j}(\partial \Omega)\right)}+Ch^{\frac{1}{2}}\|\tilde{X}(hD_{x^1}+B)u(t_0)\|_{H^{s+k-\frac{3}{2}}_h(\partial \Omega)}\\
				& \qquad+Ch^N\|u\|_{H^{1}_h\left((t_0,t_2);H_h^{-N}(\partial \Omega)\right)}+Ch^N\| Pu\|_{L^2\left((t_0,t_2);H_h^{-N}(\partial \Omega)\right)}+Ch^N\|u(t_0)\|_{H_h^{-N}(\partial \Omega)}
			\end{aligned}
		\end{equation}
		for all $0<h<h_0$.
	\end{lem}
	
	\begin{proof}
		First note that for $k\ge 2$
		\begin{equation*}
			\begin{aligned}
				(hD_{x^1})^{k}X&=(hD_{x^1})^{k-2}\left(X(hD_{x^1})^2+2[hD_{x^1},X]+[hD_{x^1},[hD_{x^1},X]]\right)\\
				&=(hD_{x^1})^{k-2}\left(X(P-A_2)+2[hD_{x^1},X]+[hD_{x^1},[hD_{x^1},X]]\right)
			\end{aligned}
		\end{equation*}
		where $A_2=ha(x)hD_{x^1}-R(x^1,x',hD_{x'})$ and $R \in \Psi^2_{\tangent,h}(\Omega)$ as defined in \eqref{LaplaceonBoundary}.
		Therefore
		\begin{equation}
			\label{l:gainNormal proof eqn 1}
			\begin{aligned}
				&\|(hD_{x^1})^{k}Xu\|_{L^2\left((t_0,t_1);H_h^{s}(\partial \Omega)\right)}
				\\
				&\qquad \le \|(hD_{x^1})^{k-2} XPu\|_{L^2\left((t_0,t_1);H_h^{s}(\partial \Omega)\right)}+\|(hD_{x^1})^{k-2} X u\|_{L^2\left((t_0,t_1);H_h^{s+2}(\partial \Omega)\right)}
				\\
				&\qquad +Ch \|\tilde X u\|_{H^{k-2}_h\left((t_0,t_1);H_h^{s+2}(\partial \Omega)\right)}+Ch^N \|u\|_{H^{k-2}_h\left((t_0,t_1);H_h^{-N}(\partial \Omega)\right)}
			\end{aligned}
		\end{equation}
		implies
		\begin{equation*}
			\begin{aligned}
				&\|(hD_{x^1})^{2}Xu\|_{L^2\left((t_0,t_1);H_h^{s}(\partial \Omega)\right)}
				\\
				&\qquad \le \|XPu\|_{L^2\left((t_0,t_1);H_h^{s}(\partial \Omega)\right)}+\|X u\|_{L^2\left((t_0,t_1);H_h^{s+2}(\partial \Omega)\right)}
				\\
				&\qquad +Ch \|\tilde X u\|_{L^2\left((t_0,t_1);H_h^{s+2}(\partial \Omega)\right)}+Ch^N \|u\|_{L^2\left((t_0,t_1);H_h^{-N}(\partial \Omega)\right)}.
			\end{aligned}
		\end{equation*}
		Then, together with Lemma \ref{DtN microlocal Elliptic estimate lemma H1 -}, we obtain \eqref{l:gainNormal eqn1} for the case $k=2$ and all $s\in \R$. Now, suppose \eqref{l:gainNormal eqn1} holds for $k$ and all $s\in \R$. Then \eqref{l:gainNormal proof eqn 1} says
		\begin{equation*}
			\begin{aligned}
				&\|(hD_{x^1})^{k+1}Xu\|_{L^2\left((t_0,t_1);H_h^{s}(\partial \Omega)\right)}
				\\
				&\qquad \le \|XPu\|_{H^{k-1}_h\left((t_0,t_1);H_h^{s}(\partial \Omega)\right)}+\|X u\|_{H^{k-1}_h\left((t_0,t_1);H_h^{s+2}(\partial \Omega)\right)}
				\\
				&\qquad +Ch \|\tilde X u\|_{H^{k-1}_h\left((t_0,t_1);H_h^{s+2}(\partial \Omega)\right)}+Ch^N \|u\|_{H^{k-1}_h\left((t_0,t_1);H_h^{-N}(\partial \Omega)\right)}.
			\end{aligned}
		\end{equation*}
		Moreover, \eqref{l:gainNormal eqn1} with $s$ replacing by $s+1$ says
		\begin{equation*}
			\begin{aligned}
				&\|Xu\|_{H_h^{k}\left((t_0,t_1);H_h^{s+1}(\partial \Omega)\right)}+\|Xu\|_{H_h^{k-1}\left((t_0,t_1);H_h^{s+2}(\partial \Omega)\right)}\\
				&\qquad\leq C\sum_{j=0}^{k-2}\|\tilde XPu\|_{H_h^{k-2-j}\left((t_0,t_2);H_h^{s+1+j}(\partial \Omega)\right)}+Ch^{\frac{1}{2}}\|\tilde{X}u(t_0)\|_{H^{s+k+\frac{1}{2}}_h(\partial \Omega)}\\
				& \qquad+Ch^N\|u\|_{H^{1}_h\left((t_0,t_2);H_h^{-N}(\partial \Omega)\right)}+Ch^N\| Pu\|_{L^2\left((t_0,t_2);H_h^{-N}(\partial \Omega)\right)}+Ch^N\|u(t_0)\|_{H_h^{-N}(\partial \Omega)}.
			\end{aligned}
		\end{equation*}
		This implies
		\begin{equation*}
			\begin{aligned}
				&\|Xu\|_{H_h^{k+1}\left((t_0,t_1);H_h^{s}(\partial \Omega)\right)}+\|Xu\|_{H_h^{k}\left((t_0,t_1);H_h^{s+1}(\partial \Omega)\right)}+\|Xu\|_{H_h^{k-1}\left((t_0,t_1);H_h^{s+2}(\partial \Omega)\right)}\\
				&\qquad\leq C\sum_{j=0}^{(k+1)-2}\|\tilde XPu\|_{H_h^{(k+1)-2-j}\left((t_0,t_2);H_h^{s+j}(\partial \Omega)\right)}+Ch^{\frac{1}{2}}\|\tilde{X}u(t_0)\|_{H^{s+(k+1)-\frac{1}{2}}_h(\partial \Omega)}\\
				& \qquad+Ch^N\|u\|_{H^{1}_h\left((t_0,t_2);H_h^{-N}(\partial \Omega)\right)}+Ch^N\| Pu\|_{L^2\left((t_0,t_2);H_h^{-N}(\partial \Omega)\right)}+Ch^N\|u(t_0)\|_{H_h^{-N}(\partial \Omega)},
			\end{aligned}
		\end{equation*}
		which completes the proof for \eqref{l:gainNormal eqn1} and \eqref{l:gainNormal eqn2} follows similarly.
	\end{proof}

    \begin{rem}
    \label{interior elliptic estimate}
    The classical interior elliptic estimate follows immediately from Lemma \ref{l:gainNormal}. That is, if $P$ is a classical second-order elliptic operator, then for any $V\Subset U$, we have
    \begin{equation*}
    \label{interior elliptic eqn}
        \|u\|_{H^2_h(V)}\le C (\|Pu\|_{L^2(U)}+\|u\|_{L^2(U)}).
    \end{equation*}
    \end{rem}

	\subsection{Exterior problem}
	\label{LU exterior problem}
	In this section, we apply the results of the previous sections to $P_{\exterior}-z^2$.
\begin{lem}
\label{l:dtNDescriptionExterior}
Let $\e_0>0$, $M>0$, $X,\blue{Y}\in \Psi_{\tangent,h}^{\comp}(\partial\Omega)$ with $\WF(\blue{Y}) \subset \{|\xi'|_{g_\exterior}>1+\e_0\}$ and $\WF(X)\cap \WF(I-\blue{Y})=\emptyset$. Let $U$ be a Fermi normal coordinate neighborhood of $\partial \Omega$ in $\Omega_{\exterior}$ with coordinates $(x^1,x')$, $E_-$ as in Proposition~\ref{Pdecompmicrolocal}. Then for all $\chi\in C_c^\infty(-1,1)$ with $\chi\equiv 1$ near $0$, $\e>0$ small enough, all $k,N\geq 0$, and $\psi\in C_c^\infty(\overline{\Omega}_{\exterior})$, there is $C>0$ such that for all $0<h<1$, $z\in [1-\e_0,1+\e_0]+\rmi[-Mh,Mh]$  and $u\in L^2(\partial\Omega)$ we have
$$
\|\psi \partial_z^k (\chi(\e^{-1}x^1)v-G_{\exterior}Xu)\|_{H_h^2(\Omega_{\exterior})}\leq Ch^N\|u\|_{L^2(\partial\Omega)},
$$
where $v$ satisfies
\begin{equation}
\label{e:too basic}
(hD_{x^1}-\Lambda)v=0,\qquad v|_{x^1}=Xu,\qquad \Lambda:=-\rmi(E_- \blue{Y}-(I-\blue{Y})\Op(\langle \xi'\rangle)).
\end{equation}
Moreover, for $X'\in \Psi_{\tangent,h}^{\comp}$ with $\WF(I-X')\cap \WF(\blue{Y})=\emptyset$,
$$
\|(I-X')\partial_z^k\chi v\|_{H_h^N(\Omega_\exterior)}\leq C_Nh^N\|u\|_{L^2(\partial\Omega)}.
$$
\end{lem}
\begin{proof}
Let $\e>0$ small enough so that $\{d(\partial\Omega,x)<3\e\}\subset U$. We will require a few microlocal cutoffs below. Let $X_j\in \Psi^{\comp}_{\tangent,h}(\Omega)$, $j=0,1,\dots, 4$ such that \blue{$X_0=X$, $X_2=\blue{Y}$, $X_4=X'$ and they satisfy}
\begin{equation*}
\begin{aligned}
&\WF(X_j)\subset \{|\xi'|_{g_\exterior}-1-\e_0>0\}& &\text{for} \quad j=0,1,\dots, 4,\\
&\WF(X_j)\cap \WF(I-X_{j'})=\emptyset & &\text{for} \quad j<j'.
\end{aligned}
\end{equation*}
Notice that, with $\Lambda:=-\rmi (E_-X_{2}-(I-X_{2})\Op(\langle\xi'\rangle))$, we have $\Im\sigma(\Lambda)>c_0\langle\xi'\rangle>0$. Define 
$$
\tilde{v}(x^1,x'):=e^{\frac{c_0x^1}{2h}}v(x^1,x'),
$$
so that $\tilde{v}$ satisfies
\begin{equation}
\label{e:tildeV}
(hD_{x^1}-(\Lambda-\rmi \tfrac{c_0}{2}))\tilde v=0,\qquad \tilde v|_{x^1=0}=Xu.
\end{equation}
We claim that for any $X_4\in\Psi_{\tangent,h}^\comp$ with $\WF(I-X_4)\cap \WF(X_2)=\emptyset$, any $k\geq 0$, and any $0<\e'<3\e$,
\begin{gather}
\|\partial_z^k\tilde{v}\|_{L^2\left((0,3\e);H_h^N(\partial\Omega)\right)}\leq Ch^{\frac{1}{2}}\|Xu\|_{L^2(\partial\Omega)},\label{e:basicL2}\\
\|(I-X_4)\partial_z^k\tilde{v}\|_{H_h^N\left((0,\e')\times \partial\Omega\right)}\leq C_Nh^N\|u\|_{L^2(\partial\Omega)}\label{e:microlocalK}.
\end{gather}
For $k=0$, using that $\Im \sigma(\Lambda-\rmi \tfrac{c_0}{2})>\frac{c_0}{2}\langle\xi'\rangle$, we have by Lemma~\ref{l:lopa} with $X$ replaced by $I$ that,
\begin{equation}
\label{e:basic}
\|\tilde{v}\|_{L^2\left((0,3\e);H_h^{s}(\partial \Omega)\right)}\leq Ch^{\frac{1}{2}}\|Xu\|_{L^2(\partial \Omega)}.
\end{equation}
Using Lemma~\ref{l:lopa} again, this time with $X$ and $\tilde X$ replaced by $(I-X_3)$ and $(I-X_1)$ respectively, we obtain that 
\begin{equation}
\label{l:outDtNSign eqn 1}
\begin{aligned}
&\|(I-X_3)\tilde{v}\|_{L^2\left((0,\e');H_h^s(\partial \Omega)\right)}
\\
&\qquad \leq C_Nh^N\left(\|\tilde{v}\|_{L^2\left((0,3\e);H^{-N}_h(\partial \Omega)\right)}+\|u\|_{L^2(\partial \Omega)}\right)\leq C_Nh^N\|u\|_{L^2(\partial \Omega)}.
\end{aligned}
\end{equation}
Hence, using~\eqref{e:tildeV} again, many times,  
\begin{equation*}
\begin{aligned}
&\|(I-X_4)(hD_{x^1})^k\tilde{v}\|_{L^2\left((0,\e');H_h^{N-k}(\partial \Omega)\right)}\\
&=\|(I-X_4)(\Lambda-\rmi\tfrac{c_0}{2})^k\tilde{v}\|_{L^2\left((0,\e');H_h^{N-k}(\partial \Omega)\right)}\\
&\leq\|(\Lambda-\rmi\tfrac{c_0}{2})^k(I-X_4)\tilde{v}\|_{L^2\left((0,\e');H_h^{N-k}(\partial \Omega)\right)}+Ch\|(I-X_3)\tilde{v}\|_{L^2\left((0,\e');H^{N-1}_h(\partial \Omega)\right)}.
\end{aligned}
\end{equation*}
That is 
\begin{equation*}
\label{l:outDtNSign eqn 2}
\begin{aligned}
&\|(I-X_4)\tilde v\|_{H^k_h\left((0,\e');H_h^{N-k}(\partial \Omega)\right)}
\\
&\qquad \le \|(I-X_4)\tilde{v}\|_{L^2\left((0,\e');H_h^{N}(\partial \Omega)\right)}+Ch\|(I-X_3)\tilde{v}\|_{L^2\left((0,\e');H^{N-1}_h(\partial \Omega)\right)}.
\end{aligned}
\end{equation*}
Combining with \eqref{l:outDtNSign eqn 1}, one has
\begin{equation}
\label{e:microlocal}
\|(I-X_4)\tilde{v}\|_{H_h^N\left((0,\e')\times \partial \Omega\right)}\leq C_Nh^N\|u\|_{L^2(\partial \Omega)}.
\end{equation}

Now, suppose that ~\eqref{e:basicL2} and~\eqref{e:microlocal} hold for $0\leq k\leq K-1$. 
Then, observe that 
$$
(hD_{x^1}-(\Lambda+\rmi \tfrac{c_0}{2}))\partial_z^K\tilde{v}=\sum_{0\leq j\leq K-1} A_j(z)\partial_z^j\tilde{v},\qquad \partial_z^K\tilde{v}|_{x^1=0}=0,
$$
where $A_j(z)\in \Psi_{\tangent,h}^{\comp}$ can be computed from derivatives of $\Lambda$ in $z$. 
Applying Lemma~\ref{l:lopa} as in \eqref{e:basic} and~\eqref{l:outDtNSign eqn 1}, we obtain respectively
$$
\|\partial_z^K\tilde{v}\|_{L^2\left((0,3\e);H_h^s(\partial \Omega)\right)}\leq C\sum_{j\leq K-1}\|\partial_z^j\tilde{v}\|_{L^2\left((0,3\e);H_h^{s-\frac{1}{2}}(\partial \Omega)\right)}\leq Ch^{\frac{1}{2}}\|Xu\|_{L^2(\partial\Omega)},
$$
and $\e'<\e''<3\e$,
\begin{equation*}
\begin{aligned}
\|(I-X_4)\partial_z^K\tilde{v}&\|_{L^2\left((0,\e');H_h^s(\partial \Omega)\right)} \leq C\sum_{0\leq j\leq K-1}\|(I-X_3)A_j\partial_z^j \tilde{v}\|_{L^2\left((0,\e'');H_h^{s-1}(\partial \Omega)\right)}\\
&+C_Nh^N\left(\|\partial_z^K\tilde{v}\|_{L^2\left((0,\e'');H_h^{-N}(\partial \Omega)\right)}+\sum_{0\leq j\leq K-1}\|A_j\partial_z^j \tilde{v}\|_{L^2\left((0,\e'');H_h^{-N}(\partial \Omega)\right)}\right)\\
&\leq C_Nh^N\|u\|_{L^2(\partial \Omega)}.
\end{aligned}
\end{equation*}
Hence, arguing as we did to obtain~\eqref{e:microlocal}, we have that~\eqref{e:basicL2} and~\eqref{e:microlocalK} hold for all $K$.

Now that we have~\eqref{e:basicL2} and~\eqref{e:microlocalK}, we finish the proof of the lemma by understanding $\partial_z^K(P_\exterior-z^2)\chi(x^1)v$. 
Observe that 
\begin{equation}
\label{e:quasimode}
\begin{aligned}
&\partial_z^K(P_{\exterior}-z^2)\chi(x^1)v
\\
&=\partial_z^KX_4(hD_{x^1}-h\tilde{a}-\rmi E_{-}(x,hD_{x'}))(hD_{x^1}+\rmi E_-(x,hD_{x'}))\chi(x^1)v\\
&\qquad+\partial_z^K(I-X_4)(P_{\exterior}-z^2)\chi(x^1)v+O(h^\infty)_{\Psi_{\tangent,h}^{-\infty}}\chi(x^1)v\\
&=\partial_z^KX_4(hD_{x^1}-h\tilde{a}-\rmi E_{-}(x,hD_{x'}))\chi(x^1)(hD_{x^1}-\Lambda )v\\
&\qquad-\rmi h X_4\partial_z^K(hD_{x^1}-h\tilde{a}-\rmi E_{-}(x,hD_{x'}))\chi'(x^1)v\\
&\qquad+(I-X_4)\partial_z^K(P_{\exterior}-z^2)\chi(x^1)(I-X_{3})v+O(h^\infty)_{\Psi_{\tangent,h}^{-\infty}}\chi(x^1)v.
\end{aligned}
\end{equation}
By~\eqref{e:microlocalK}, we have
\begin{equation}
\label{e:a}
\|(I-X_3)\partial_z^Kv\|_{H_h^N\left((0,2\e)\times\partial \Omega\right)}=\left\|e^{-\frac{c_0 x^1}{2h}}(I-X_3)\partial_z^K\tilde{v}\right\|_{H_h^N\left((0,2\e)\times\partial \Omega\right)}\leq C_Nh^N\|u\|_{L^2(\partial \Omega)},
\end{equation}
 by~\eqref{e:too basic} we have for any $j$, and $B_j\in\Psi_{\tangent,h}^{\comp}$,
\begin{equation}
\label{e:b}
\begin{aligned}
&\|X_4(hD_{x^1}-B_j))\chi'\partial_z^jv\|_{L^2(\Omega_\exterior)}
\\
&\qquad \le \|X_4\chi'hD_{x^1}\partial_z^jv\|_{L^2(\Omega_\exterior)}+h\|X_4\chi''\partial_z^jv\|_{L^2(\Omega_\exterior)}+\|\partial_z^jv\|_{L^2\left((\e,2\e);H^1_h(\partial \Omega)\right)}
\\
&\qquad \le C\sum_{k=0}^j\|\partial_z^k v\|_{L^2\left((\e,2\e);H^1_h(\partial \Omega)\right)}\le C\left\|e^{-\frac{c_0 x^1}{2h}}\partial_z^j\tilde{v}\right\|_{L^2\left((\e,2\e);H^1_h(\partial \Omega)\right)}
\\
&\qquad \le Ce^{-\frac{c}{h}}\|u\|_{L^2(\partial \Omega)},
\end{aligned}
\end{equation}
and by~\eqref{e:basic}, we have
\begin{equation}
\label{e:c}
\|\partial_z^jv\|_{L^2\left((0,3\e);H_h^{s}(\partial \Omega)\right)}=\left\|e^{-\frac{c_0 x^1}{2h}}\partial_z^j\tilde{v}\right\|_{L^2\left((0,3\e);H_h^{s}(\partial \Omega)\right)}\leq Ch^{\frac{1}{2}}\|u\|_{L^2(\partial \Omega)}.
\end{equation}
Hence, using~\eqref{e:a},~\eqref{e:b}, and~\eqref{e:c} in~\eqref{e:quasimode}, we obtain
\begin{equation}
\label{e:derivativeEstimate}
\|\partial_z^K(P_{\exterior}-z^2)\chi(x^1)v\|_{L^2(\Omega_\exterior)}\leq C_Nh^N\|u\|_{L^2(\partial\Omega)}.
\end{equation}
Taking $K=0$, using that $G_{\exterior}(z)Xu$ is outgiong, we first obtain for any $\psi \in C_c^\infty(\overline\Omega_{\exterior})$, 
$$
\|\psi(G_{\exterior}(z)Xu-\chi(x^1)v)\|_{H_h^2(\Omega_\exterior)}=\|\psi R_{\exterior}(z)(P_\exterior - z^2)\chi(x^1)v\|_{H_h^2(\Omega_\exterior)}\leq C_Nh^N\|u\|_{L^2(\partial \Omega)}.
$$
Now, suppose by induction that for $0\leq j\leq K-1$, and any $\psi\in C_c^\infty(\overline{\Omega}_{\exterior})$,
$$
\|\psi\partial_z^j(G_{\exterior}(z)Xu-\chi(x^1)v)\|_{H_h^2(\Omega_\exterior)}\leq C_Nh^N\|u\|_{L^2(\partial\Omega)}.
$$
Then, observe that $\partial_z^K(G_{\exterior}(z)Xu-\chi v)|_{\partial\Omega}=0$, and
$$
[(P_\exterior-z^2)\partial_z^K(G_{\exterior}(z)Xu-\chi v)]=\sum_{j=0}^{K-1}(a_j+b_jz)\partial_z^j(G_{\exterior}(z)Xu-\chi v)- \partial_z^K[(P_{\exterior}-z^2)\chi v],
$$
where $a_j,b_j\in\mathbb{C}$.  By~\eqref{e:derivativeEstimate} and the inductive hypothesis, we have for any $\tilde{\psi}\in C_c^\infty(\overline{\Omega}_{\exterior})$,
$$
\|\tilde{\psi} [(P_\exterior-z^2)\partial_z^K(G_{\exterior}(z)Xu-\chi v)]\|_{L^2(\Omega_\exterior)}\leq C_Nh^N\|u\|_{L^2(\partial\Omega)}.
$$
Hence, since $\partial_z^KG(z)Xu$ is outgoing, that for any $\psi\in C_c^\infty(\overline{\Omega}_\inner)$ with $\supp \psi\cap \supp(1-\tilde{\psi})=\emptyset$, 
$$
\|\psi\partial_z^K(G_{\exterior}(z)Xu-\chi v)]\|_{H_h^2(\Omega_{\exterior})}\leq C_Nh^N\|u\|_{L^2(\partial\Omega)}.
$$
\end{proof}

	\begin{prop}
		\label{symbol of DtN exterior}
		Let $\e_0\geq 0$, $M>0$, $X\in \Psi^0_{h}(\partial\Omega)$ with $\WF(X) \in \{(x,\xi'):|\xi'|_{g_{\exterior}}>1+\e_0\}$. Then for all $z\in [1-\e_0,1+\e_0]+\rmi[-Mh,Mh]$, we have  $X\Lambda_{\exterior}(z)\in \Psi^{1}_h(\partial\Omega)$ and 
		\begin{equation*}
			\label{prin and sub of XhDu}
				\sigma(X\Lambda_\exterior (z))=\sigma(X)\rho_{\exterior}\left(|\xi'|^2_{g_\Eucl}-(\Re z)^2\right)^{\frac{1}{2}}.\\
		\end{equation*}
            Moreover, for $X_c\in\Psi^{\comp}_h(\partial\Omega)$ with $\WF(X_c)\subset \{ |\xi'|_{g_{\exterior}}>1+\e_0\}$, and $k\geq 0$, $X_c\partial_z^k\Lambda_{\exterior}\in \Psi^{\comp}_h$ with symbol
            \begin{equation}
            \label{e:derivativesOut}\sigma(X_c\partial_z^k\Lambda_{\exterior})=\sigma(X_c)\rho_{\exterior}\partial_z^k\left(|\xi'|^2_{g_\Eucl}-z^2\right)^{\frac{1}{2}}.
            \end{equation}
	\end{prop}
	\begin{proof}
		In this case, we recall~\eqref{e:pExtRewrite}.
		In particular, by~\eqref{DtN Elliptic estimate eqn} and Lemma~\ref{l:gainNormal}
		\begin{equation*}
		\begin{aligned}
        &\|(X(hD_{x^1}+\rmi E_-)G_\exterior (z)u_0)(0)\|_{H^s_h(\partial \Omega)}
        \\
        &\qquad \leq Ch^N\|G_\exterior (z)u_0\|_{H^1_h\left((t_0,t_2);H_h^{-N}(\partial \Omega)\right)}\leq Ch^N\|u_0\|_{H_h^{-N}(\partial \Omega)},
		\end{aligned}
		\end{equation*}
		where we have used the non-trapping estimate of $\|\chi G_\exterior\|_{H^{-N}_h(\partial \Omega)\to H^2_h(\Omega_\exterior)}\le C h^{-1}$. Hence, since $\Lambda_\exterior (z)u_0 = - \rho_\exterior h\partial_{x^1}u|_{x^1=0} = -\rmi \rho_\exterior hD_{x^1}u|_{x^1=0}$
		$$
		X\Lambda_{\exterior}(z)u_0=\left(-X\rho_\exterior E_-+O(h^\infty)_{\Psi_h^{-\infty}}\right)u_0,
		$$
		and the first statement follows since $\sigma(E_-)=-\sqrt{|\xi'|_{g_{\exterior}}^2-(\Re z)^2}$.
        The second statement follows directly from Lemma~\ref{l:dtNDescriptionExterior}.

	\end{proof}

	\begin{lem}
    \label{l:outDtNSign}
		Let $\e_0>0$ and $M>0$. Then for $X\in \Psi^{\comp}_h(\partial\Omega)$, with $\WF(X)\subset \{|\xi'|_{g_{\exterior}}^2>1+\e_0\}$ and $z\in[1-\e_0,1+\e_0]+\rmi [-Mh,Mh]$, we have
		$$
		-\operatorname{sgn}(\Im z^2)\Im \langle \Lambda_{\exterior}(z)Xu,Xu\rangle_{L^2\left(\partial\Omega, \der \vol_{g_\exterior,\partial\Omega}\right)} \geq C|\Im z^2|\|Xu\|_{L^2(\partial \Omega)}^2 +O(h^\infty)\|u\|_{L^2(\partial \Omega)}^2.
		$$
        for some $C>0$.
	\end{lem}
	\begin{proof}
    First observe that, integration by parts on $B(0,R)\cap \Omega_{\exterior}$ yields
    \begin{equation*}
    \label{e:intByParts}
    \begin{aligned}
    &-h\Im \langle \Lambda_{\exterior}(z)Xu,Xu\rangle_{L^2(\partial\Omega, \der \vol_{g_\exterior,\partial\Omega})}\\
    &\quad=\Im z^2\|G_{\exterior}Xu\|_{L^2(B(0,R)\cap \Omega_{\exterior},\rho_\exterior\der \vol_{g_\exterior})}^2+h\Im \langle \rho_\exterior h\partial_{r}G_{\exterior}Xu,G_{\exterior}Xu\rangle_{L^2(\partial B(0,R), \der \vol_{g_\exterior,\partial B(0,R)})}.
    \end{aligned}
    \end{equation*}
To simplify the notation, we will omit the dependence on $\rho_\exterior$. In order to complete the proof, we need to understand $G_{\exterior}Xu$. 

We now apply Lemma~\ref{l:dtNDescriptionExterior}
\begin{align*}
-h\Im \langle \Lambda_\exterior(z)G_\exterior Xu,Xu\rangle_{L^2(\partial\Omega)}&=\Im z^2\|\chi(x^1)v\|_{L^2(\Omega_\exterior)}^2+O(h^\infty)\|u\|_{L^2(\partial \Omega)}^2.
\end{align*}

Finally, we have, letting $X'\in \Psi_{\tangent,h}^\comp$ with $\WF(I-X')\cap \WF(X)=\emptyset$
\begin{align*}
\|Xu\|_{L^2(\partial \Omega)}^2&=-\int_0^\infty \partial_{x^{1}}\|\chi(x^1)\tilde v(x^1)\|_{L^2(\partial \Omega)}^2\der x^1\\
&=-\int_0^\infty 2\Re \left\langle \partial_{x^1}(\chi(x^1)v(x^1)),\chi(x^1)v(x^1)\right\rangle_{L^2(\partial \Omega)}\der x^1\\
&\leq C h^{-1}\|\chi(x^1)hD_{x^1}v\|_{L^2\left((0,2\e) \times \partial \Omega\right)}\|\chi(x^1)v\|_{L^2\left((0,2\e) \times \partial \Omega\right)}+C\|v\|^2_{L^2\left((0,2\e) \times \partial \Omega\right)}\\
&=C h^{-1}\|\chi(x^1)\Lambda v\|_{L^2\left((0,2\e) \times \partial \Omega\right)}\|\chi(x^1)v\|_{L^2\left((0,2\e) \times \partial \Omega\right)}+C\|v\|^2_{L^2\left((0,2\e) \times \partial \Omega\right)}\\
&\le C h^{-1}\|\Lambda X'\chi(x^1)v\|_{L^2\left((0,2\e) \times \partial \Omega\right)}\|\chi(x^1)v\|_{L^2\left((0,2\e) \times \partial \Omega\right)}+C_Nh^N\|Xu\|_{L^2(\partial \Omega)}^2\\
&\leq Ch^{-1}\|\chi(x^1)v\|_{L^2(\Omega_\exterior)}^2+C_Nh^N\|u\|_{L^2(\partial \Omega)}^2,
\end{align*}
which completes the proof of the lemma.
	%
	\end{proof}

    \subsection{Application to the interior problem}
	\label{LU interior problem}
	Next, we apply our estimates to $P_{\inner}-z^2$.
    \begin{lem}
\label{l:dtNDescriptionInner}
Let $\e_0>0$, $M>0$, $U$ be a Fermi normal coordinate neighborhood of $\partial \Omega$ in $\Omega_{\exterior}$ with coordinates $(x^1,x')$, $E_-$ as in Proposition~\ref{Pdecompmicrolocal} with $P_{\inner}-z^2=-P(-z^2,g_{\inner},L_{\inner})$ (as in~\eqref{e:pInRewrite}). Then there for all $\chi\in C_c^\infty(-1,1)$ with $\chi\equiv 1$ near $0$, $\e>0$ small enough, and $k,N\geq 0$ there is $C>0$ such that for all $0<h<1$, $z\in[1-\e_0,1+\e_0]+\rmi[-Mh,Mh]$, and $u\in L^2(\partial\Omega)$,  we have
$$
\| \partial_z^k (\chi(\e^{-1}x^1)v-G_{\inner}u)\|_{H_h^2(\Omega_{\inner})}\leq Ch^N\|u\|_{L^2(\partial\Omega)},
$$
where $v$ satisfies
\begin{equation*}
\label{e:too basic interior}
(hD_{x^1}+\rmi E_-)v=0,\qquad v|_{x^1}=u.
\end{equation*}
\end{lem}
\begin{proof}
The proof of this lemma is nearly identical to Lemma~\ref{l:dtNDescriptionExterior} with the caveat that all cutoffs can be taken to be the identity, which simplifies the proof substantially. 
\end{proof}
    
	\begin{prop}
		\label{symbol of DtN interior}
		Let $\e_0>0$, $M>0$. Then for all $z\in [1-\e_0,1+\e_0]+\rmi[-Mh,Mh]$, we have $\Lambda_{\inner}(z)\in \Psi_h^1(\partial\Omega)$ with principal symbol 
        \begin{equation*}
			\label{prin and sub of calV}
			\sigma(\Lambda_{\inner}(z))=\rho_\inner\left(|\xi'|^2_{g_\inner}+(\Re z)^2\right)^{\frac{1}{2}}.
		\end{equation*}
        Moreover
        \begin{equation}
        \label{e:derivativesIn}
        \sigma(\partial_z^\alpha \Lambda_{\inner}(z))= \rho_{\inner}\partial_z^\alpha\left(|\xi'|^2_{g_\inner}+z^2\right)^{\frac{1}{2}}.
        \end{equation}
	\end{prop}
	\begin{proof}
    The Proposition follows directly from Lemma~\ref{l:dtNDescriptionInner} once we calculate the symbol of $E_-$.
        Recall that $\sigma(E_-)=-\sqrt{|\xi'|_{g_{\inner}}^2+(\Re z)^2}.$ Recall that $\nu_{g_\inner}$ is the outward normal with respect to the metric $g_\inner$ and we have $-\partial_{ x^1_\inner}=\partial_{\nu_{g_\inner}}$
		In particular, the DtN map with respect to $g_\inner$ is given by $\Lambda_{\inner}(z)u_0=\rho_{\inner}h\partial_{\nu_\inner}u_0|_{\partial \Omega_\inner}$, which can be written as
		$$
		\Lambda_{\inner}(z)u_0=  \left(-\rho_{\inner}E_-+O(h^\infty)_{\Psi_h^{-\infty}}\right)u_0.
		$$

	\end{proof}
	
	\begin{lem}\label{l:inDtNSign}
		 Let $\e_0\geq 0$ and $M>0$. Then for $X\in \Psi^{k}_h(\partial\Omega)$, and $z\in[1-e_0,1+\e_0]+\rmi[-Mh,Mh]$, we have
		$$
		\operatorname{sgn}(\Im z^2)\Im \langle \Lambda_{\inner}Xu,Xu\rangle_{L^2\left(\partial\Omega, \der \vol_{g_\inner,\partial\Omega}\right)} \geq C |\Im z^2|\|Xu\|_{L^2(\partial \Omega)}^2,
		$$
        or equivalently,
        $$
		\operatorname{sgn}(\Im z^2)\Im \langle \tau \Lambda_{\inner}Xu,Xu\rangle_{L^2\left(\partial\Omega, \der \vol_{g_\exterior,\partial\Omega}\right)} \geq C |\Im z^2|\|Xu\|_{L^2(\partial \Omega)}^2,
		$$
		for some constant $C>0$.
	\end{lem}
	\begin{proof}
		Observe that, from Section \ref{Reformulation of the problem as an exterior problem}, one has
		\begin{equation}
        \label{e:intByPartsInner}
		\begin{aligned}
        &h\Im \langle \tau \Lambda_{\inner}Xu,Xu\rangle_{L^2\left(\partial\Omega, \der \vol_{g_\exterior,\partial\Omega}\right)}=h\Im \langle \Lambda_{\inner}Xu,Xu\rangle_{L^2\left(\partial\Omega, \der \vol_{g_\inner,\partial\Omega}\right)}\\
        &= \Im \langle z^2G_{\inner}(z)Xu,G_{\inner}(z)Xu\rangle_{L^2\left(\Omega_\inner,\rho_\inner\der \vol_{g_\inner}\right)} =\Im z^2\|G_{\inner}(z)Xu\|_{L^2\left(\Omega_\inner,\rho_\inner\der \vol_{g_\inner}\right)}^2.
		\end{aligned}
        \end{equation}
		We will omit the dependence of $\rho_\inner$ to ease the notations. Now, $(P_{\inner}-z^2)G_{\inner}(z)g=0\text{ in }\Omega_{\inner}$, $G_{\inner}(z)g|_{\partial\Omega}=g$ and hence, using the factorization~\eqref{Laplace factorisation}, we have by Lemma \ref{l:gainNormal}  with $X=\tilde{X}=I$,
        \begin{align*}
        \|G_{\inner}(z)Xu\|_{H_h^1\left((0,\e);L^2(\partial\Omega)\right)}&\leq Ch^N\|G_{\inner}(z)Xu\|_{H_h^1\left((0,2\e);H_h^{-N}(\partial\Omega)\right)}+Ch^{\frac{1}{2}}\|Xu\|_{L^2(\partial\Omega)}.
        \end{align*}
        
        Subtracting part of the first term on the left to the right-hand side and using local elliptic regularity for $P_{\inner}-z^2$ and applying~\eqref{e:innerEstimates} we have
        \begin{equation}
        \label{l:inDtNSign eqn 1}
        \begin{aligned}
         \|G_{\inner}(z)Xu\|_{H_h^1\left((0,\e);L^2(\partial\Omega)\right)}&\leq Ch^N\|G_{\inner}(z)Xu\|_{H_h^1\left((\e,2\e);H_h^{-N}(\partial\Omega)\right)}+Ch^{\frac{1}{2}}\|Xu\|_{L^2(\partial\Omega)}\\
         &\leq  Ch^N\|G_{\inner}(z)Xu\|_{L^2(\Omega_{\inner})}+Ch^{\frac{1}{2}}\|Xu\|_{L^2(\partial\Omega)},
        \end{aligned}
        \end{equation}
        where we have used that for $U\Subset \Omega_{\inner}$, one has the interior elliptic estimate
        \begin{equation}
        \label{l:inDtNSign eqn 2}
        \|G_{\inner}(z)Xu\|_{H_h^N(U)}\leq Ch^N\|G_{\inner}(z)Xu\|_{L^2(\Omega_\inner)}.
        \end{equation}
        Combining \eqref{l:inDtNSign eqn 1} and \eqref{l:inDtNSign eqn 2}, one obtains
        \begin{equation}
        \label{l:inDtNSign eqn 3}
         \|G_{\inner}(z)Xu\|_{H_h^1\left((0,2\e);L^2(\partial \Omega)\right)}\leq  Ch^N\|G_{\inner}(z)Xu\|_{L^2(\Omega_{\inner})}+Ch^{\frac{1}{2}}\|Xu\|_{L^2(\partial\Omega)}.
        \end{equation}
        Hence, letting $\varphi\in C_c^\infty([0,2\e))$ with $\varphi\equiv 1$ on $[0,\e]$, 
        \begin{align*}
        &\|Xu\|^2_{L^2(\partial\Omega)}=-\int_0^\infty \partial_{x^{1}}\left(\varphi(x^1)\|G_{\inner}(z)Xu(x^1)\|_{L^2(\partial \Omega_\inner)}^2\right)\der x^1\\
        &\leq C\|G_{\inner}(z)Xu\|_{L^2\left(\Omega_{\inner},\rho_\inner\der \vol_{g_\inner}\right)}^2+C h^{-1}\int_0^\infty \|hD_{x^1}G_{\inner}(z)Xu\|_{L^2(\partial \Omega)}\|G_{\inner}(z)Xu\|_{L^2(\partial \Omega)} \der x^1\\
        &\leq Ch^{-1}\left(\delta\|G_{\inner}(z)Xu\|^2_{H_h^1\left((0,2\e);L^2(\partial \Omega)\right)}+(1+\delta^{-1})\|G_{\inner}(z)Xu\|_{L^2\left(\Omega_{\inner},\rho_\inner\der \vol_{g_\inner}\right)}\right)\\
        &\leq Ch^{-1}\left(\delta h\|Xu\|_{L^2(\partial\Omega)}^2+(2+\delta^{-1})\|G_{\inner}(z)Xu\|_{L^2(\Omega_{\inner})}\right),
        \end{align*}
        where we used \eqref{l:inDtNSign eqn 3} in the last step. Now, choosing $\delta$ small enough in the above estimate, we have
        $$
        \|Xu\|_{L^2(\partial\Omega)}^2\leq Ch^{-1}\|G_{\inner}(z)Xu\|_{L^2(\Omega_{\inner})}^2.
        $$
        The lemma now follows by combining this with~\eqref{e:intByPartsInner}.
	\end{proof}
	
	\subsection{Combination of interior and exterior problems}
	\label{LU combination problem}
	
    Thanks to the ellipticity of the interior problem, we have an accurate representation of $\Lambda_{\inner}(z)$  and so we work with $\rho_{\inner}h\partial_{\nu_{\inner}}u_{\inner}$ replaced by $\Lambda_{\inner}(z)u_{\inner}$. However, the exterior Dirichlet-to-Neumann map can only be accurately approximated microlocally in $|\xi'|_{g_{\exterior}}>1$ (this is done in Proposition \ref{symbol of DtN exterior}). Therefore, the difference of the exterior Dirichlet-to-Neumann and the interior Dirichlet-to-Neumann can only be directly analyzed microlocally on microlocally in $|\xi'|_{g_{\exterior}}>1$. 
	\begin{prop}
    \label{p:imagPart}
		Let $e_0\geq 0$, $M>0$, and $X\in \Psi^\comp_h(\partial\Omega)$, with $\WF(X)\subset\{|\xi'|_{g_{\exterior}}>1+\e_0\}$.  Then, for all $N\in \N$, there exists $C, C_N, h_0>0$ such that for $0<h<h_0$ and $z\in[1-\e_0,1+\e_0]+\rmi[-Mh,Mh]$,
		\begin{equation*}
			\label{Xn control via DtN difference eqn}
			|\Im z^2|\|X u\|^2_{L^2(\partial\Omega)}-C_Nh^N\|u\|_{L^2(\partial\Omega)}^2\le C\left(\left|\langle ( \Lambda_{\exterior}-\tau\Lambda_{\inner}) X u, X u \rangle_{L^2(\partial\Omega, \der \vol_{g_\exterior,\partial\Omega})}\right|\right).
		\end{equation*}
	\end{prop}
	\begin{proof}
        Observe that by Lemmas~\ref{l:inDtNSign} and~\ref{l:outDtNSign}
        $$
      \left|\Im\langle (\tau \Lambda_{\inner}-\Lambda_{\exterior})Xu,Xu\rangle_{L^2\left(\partial\Omega, \der \vol_{g_\exterior,\partial\Omega}\right)} \right|\geq C(|\Im z^2|\|Xu\|_{L^2(\partial\Omega)}^2-C_Nh^N\|u\|_{L^2(\partial\Omega)}^2),
        $$
        which proves the proposition once relabeling the constant $C$.
        \end{proof}

\section{Proofs of Theorems~\ref{t:noStates},~\ref{t:states}, and~\ref{t:plasmonic}}
\label{s:proofOfTheorems}

In this section, we prove the first three main theorems on our article.  
\subsection{Microlocal estimates for boundary traces}
Before proceeding to the proofs of the theorems, we require some microlocalized apriori estimates on boundary traces.
	\begin{lem}
		\label{l:intByPartsEstimate}
		Let $\Lambda\in \Psi^1_h(\partial\Omega)$,$X,\tilde{X}\in \Psi_h^0(\partial\Omega)$ with 
		$
		\WF(X)\subset \Ellip  ( \sigma(R+\Lambda^*\Lambda))\cap \Ellip(\tilde{X}),
		$
		and 
		$$
		P:=(hD_{x^1})^2+hahD_{x^1}-R(x,hD_{x'})
		$$
		be formally self-adjoint.  Then for all $s\in\mathbb{R}$, $\e>0$, there is $C>0$ such that 
		\begin{align*}
			& \|Xu|_{x^1=0}\|_{H_h^{s+1}(\partial \Omega)}+\|XhD_{x^1}u|_{x^1=0}\|_{H_h^s(\partial\Omega)}\\
			&\leq C \|u\|_{L^2\left((0,\e); H^{s-2}_h(\partial\Omega)\right) }+Ch^{-1}\|Pu\|_{L^2\left((0,\e);H_h^s(\partial\Omega)\right)}+C\|\tilde{X}(hD_{x^1}-\Lambda )u|_{x^1=0}\|_{H_h^{s}(\partial\Omega)}\\
			&\qquad+Ch^N\|u|_{x^1=0}\|_{H_h^{-N}(\partial\Omega)}+Ch^N\|(hD_{x^1}-\Lambda )u|_{x^1=0}\|_{H_h^{-N}(\partial\Omega)}.
		\end{align*}
	\end{lem}
    \blue{\begin{rem}
        In Lemma~\ref{l:intByPartsEstimate}, $\Lambda$ is general. To study the resonances of our transmission problem, we will take $\Lambda=\rm i \frac{\tau}{\rho_{\mathcal{O}}}\Lambda_{_{\mathcal{I}}}(z)$ so that the terms $hD_{x^1}-\Lambda$ are negligible when we use the estimate.
    \end{rem}}
	\begin{proof}
		We first claim that for any $X,\tilde{X}\in \Psi_h^0(\partial\Omega)$ with 
		$
		\WF(X)\subset \Ellip  ( \sigma(R+\Lambda^*\Lambda))\cap \Ellip(\tilde{X}),
		$
		we have
		\begin{equation}
			\label{e:plane1}
			\begin{aligned}
				\|Xu\|_{H_h^{s+1}(\partial \Omega)}^2
				&\leq C\|u\|_{L^2\left((0,\e); H^{s-2}_h(\partial\Omega)\right) }^2+Ch^{-2}\|Pu\|_{L^2\left((0,\e);H_h^s(\partial\Omega)\right)}^2+Ch^{j}\|\tilde{X}u|_{x^1=0}\|_{H_h^{s+1-\frac{j}{2}}(\partial\Omega)}^2\\
				&+\|\tilde{X}(hD_{x^1}-\Lambda )u\|_{H_h^{s}(\partial \Omega)}^2+C_Nh^N\|u\|_{H_h^{-N}(\partial \Omega)}^2+O(h^\infty)\|(hD_{x^1}-\Lambda )u\|_{H_h^{-N}(\partial \Omega)}^2.
			\end{aligned}
		\end{equation}
		
		Since $\sigma( R+\Lambda^*\Lambda)$ is real valued, we may assume without loss of generality that 
		$$\WF(X)\subset \{ \pm \sigma(R+\Lambda^*\Lambda)>0\}$$
		for some choice of $\pm$. 
		Let $E_0\in \Psi_h^s(\partial\Omega)$, with $\WF(X)\subset \Ellip(E_0)$ and 
		$$
		\WF(E_0)\subset \{ \pm \sigma(R+\Lambda^*\Lambda)>0\}.
		$$ 
        Also, let $\chi\in C_c^\infty(\mathbb{R})$ with $\chi \equiv 1 $ near $0$ and set $E=\chi(x^1)E^*_0E_0$, and assume that $X_{1,2}\in \Psi_h^0(\partial\Omega)$ with $\WF(E_0)\subset \Ellip(X_1)\subset \WF(X_1)\subset \Ellip(X_2)$. Then, define
		\begin{equation*}
			\label{Q eqn}
			\begin{aligned}
				&Q(u;E):=\langle E_0^*E_0 hD_{x^1}u,hD_{x^1}u\rangle_{L^2(\partial \Omega)}+\langle E_0^*E_0Ru,u \rangle_{L^2(\partial \Omega)}+h\langle [a,E^*_0E_0]hD_{x^1}u,u\rangle_{L^2(\partial\Omega)}\\
				&=\langle E_0 (hD_{x^1}-\Lambda)u,E_0(hD_{x^1}-\Lambda)u\rangle_{L^2(\partial \Omega)}+2\Re \langle E_0 \Lambda u, E_0(hD_{x^1}-\Lambda)u\rangle_{L^2(\partial\Omega)}\\
				&\qquad\langle (E_0^*E_0R+\Lambda^*E_0^*E_0\Lambda +h[a,E_0^*E_0]\Lambda)u,u\rangle_{L^2(\partial\Omega)} +h\langle [a,E^*_0E_0](hD_{x^1}-\Lambda)u,u\rangle_{L^2(\partial\Omega)}\\
				&=\langle E_0 (hD_{x^1}-\Lambda)u,E_0(hD_{x^1}-\Lambda)u\rangle_{L^2(\partial \Omega)}+2\Re \langle ( \Lambda E_0+[E_0,\Lambda])u, E_0(hD_{x^1}-\Lambda)u\rangle_{L^2(\partial\Omega)}\\
				&\qquad\langle (E_0^*E_0R+\Lambda^*E_0^*E_0\Lambda +h[a,E_0^*E_0]\Lambda)u,u\rangle_{L^2(\partial\Omega)} +h\langle [a,E^*_0E_0](hD_{x^1}-\Lambda)u,u\rangle_{L^2(\partial\Omega)},
			\end{aligned}
		\end{equation*}
        where we used the identity 
        $$ \langle (A-B)u,(A-B)u\rangle+\langle Bu,Bu\rangle +2\Re \langle (A-B)u,Bu\rangle  = \langle Au,Au\rangle. $$
		Next, notice that 
		\begin{align*}
			&\langle (E_0^*E_0R+\Lambda^*E_0^*E_0\Lambda)u,u\rangle_{L^2(\partial\Omega)} \\
			&=\langle (E_0^*(RE_0+[E_0,R])+ (E_0^*\Lambda^*+[\Lambda^*,E_0^*])(\Lambda E_0+[E_0,\Lambda]) )u,u\rangle_{L^2(\partial\Omega)} \\
			&=\langle (R+\Lambda^*\Lambda)E_0u,E_0u\rangle + O(h)\|X_1 u\|_{H_h^{s+\frac{1}{2}}(\partial\Omega)}^2+O(h^\infty)\|u\|_{H_h^{-N}(\partial\Omega)}^2. 
		\end{align*}
		Then, using the microlocal G\aa rding inequality, we obtain 
        \begin{equation}
        \label{E_0 eqn 1}
            \begin{aligned}
            \|E_0u\|_{H_h^1(\partial\Omega)}^2& \leq |Q(u;E)| +Ch\|X_1 u\|_{H_h^{s+\frac{1}{2}}(\partial\Omega)}^2+C_Nh^N\|u\|_{H_h^{-N}(\partial\Omega)}^2
            \\
            &\qquad +\|X_1 (hD_{x^1}-\Lambda )u\|_{H_h^{s}(\partial\Omega)}^2+C_Nh^N\|(hD_{x^1}-\Lambda )u\|_{H_h^{-N}(\partial\Omega)}^2.
            \end{aligned}
        \end{equation}
		
		Next, we have

        \begin{align*}
			&|Q(u;E)|
			=\Big|-\frac{\rmi}{h}\langle [P,EhD_{x^1}]u,u \rangle_{L^2(\Omega)}
			-\frac{2}{h} \Im\left(\langle EhD_{x^1}u, Pu \rangle_{L^2(\Omega)}\right)\\
			&\qquad\qquad+\frac{\rmi}{h}\langle Pu, (EhD_{x^1}-(hD_{x^1})^*E)u \rangle_{L^2(\Omega)}\Big|\\
			&\leq \left(\|X_1 u\|_{H_h^2\left((0,\e);H_h^{s-1}(\partial \Omega)\right)}+\|X_1 u\|_{H_h^1\left((0,\e);H_h^{s}(\partial \Omega)\right)}+\|X_1 u\|_{L^2\left((0,\e);H_h^{s+1}(\partial \Omega)\right)}\right)\\
            &\qquad\times \|X_1 u\|_{L^2\left((0,\e);H^{s}(\partial \Omega)\right)}+Ch^{-1}\|X_1 u\|_{H_h^1\left((0,\e);H^s_h(\partial \Omega)\right)}\|X_1 Pu\|_{L^2\left((0,\e);H_h^s(\partial \Omega)\right)}\\
            &\qquad+Ch^N\left(\|u\|_{H_h^2\left((0,\e);H_h^{-N}(\partial \Omega)\right)}+\|u\|_{H_h^1\left((0,\e);H_h^{-N}(\partial \Omega)\right)}+\|u\|_{L^2\left((0,\e);H_h^{-N}(\partial \Omega)\right)}\right).
		\end{align*}
		Now Lemma \ref{l:gainNormal} says
        \begin{equation*}
        \begin{aligned}
        &\|X_1 u\|_{H_h^2\left((0,\e);H_h^{s-1}(\partial \Omega)\right)}+\|X_1 u\|_{H_h^1\left((0,\e);H_h^{s}(\partial \Omega)\right)}+\|X_1 u\|_{L^2\left((0,\e);H_h^{s+1}(\partial \Omega)\right)}
        \\
        &\qquad \le \|\tilde XPu\|_{L^2\left((0,\e);H_h^{s-1}(\partial \Omega)\right)}+C\|\tilde Xu\|_{L^2\left((0,\e);H_h^{s-1}(\partial \Omega)\right)}+Ch^{\frac{1}{2}}\|\tilde{X}u|_{x^1=0}\|_{H^{s+\frac{1}{2}}_h(\partial \Omega)}\\
        &\qquad+Ch^N\|u\|_{H^{1}_h\left((0,\e);H_h^{-N}(\partial \Omega)\right)}+Ch^N\| Pu\|_{L^2\left((0,\e);H_h^{-N}(\partial \Omega)\right)}+Ch^N\|u|_{x^1=0}\|_{H_h^{-N}(\partial \Omega)},
        \end{aligned}
        \end{equation*}
        and
        \begin{equation*}
        \begin{aligned}
        &\|u\|_{H_h^2\left((0,\e);H_h^{-N}(\partial \Omega)\right)}+\|u\|_{H_h^1\left((0,\e);H_h^{-N}(\partial \Omega)\right)}+\|u\|_{L^2\left((0,\e);H_h^{-N}(\partial \Omega)\right)}
        \\
        &\qquad \le\|Pu\|_{L^2\left((0,\e);H_h^{-N}(\partial \Omega)\right)}^2+C\|u\|_{L^2\left((0,\e);H_h^{-N}(\partial \Omega)\right)}+Ch\|u|_{x^1=0}\|_{H^{-N+2}_h(\partial \Omega)}^2\\
        &\qquad+Ch^M\|u|_{x^1=0}\|_{H_h^{-M}(\partial \Omega)}.
        \end{aligned}
        \end{equation*}
		Hence, the estimate for $|Q(u;E)|$ becomes
        \begin{equation}
        \label{E_0 eqn 2}
        \begin{aligned}
			|Q(u;E)|&\leq Ch^{-2}\|\tilde XPu\|_{L^2\left((0,\e);H_h^{s}(\partial \Omega)\right)}^2+C\|X_2 u\|_{L^2\left((0,\e);H_h^{s}(\partial \Omega)\right)}^2+Ch\|\tilde{X}u|_{x^1=0}\|_{H^{s+\frac{1}{2}}_h(\partial \Omega)}^2\\
            &\qquad+Ch^N\|u|_{x^1=0}\|_{H_h^{-N}(\partial \Omega)}.
		\end{aligned}
        \end{equation}
		Now apply Lemma \ref{DtN microlocal Elliptic estimate lemma H1 -} to $P+\omega'$ for sufficiently large $\omega'$, we have
        \begin{align*}
        &\|X_2 u\|_{H^1_h\left((0,\e);H_h^{s-1}(\partial \Omega)\right)}+\|X_2 u\|_{L^2\left((0,\e);H_h^{s}(\partial \Omega)\right)}\\
		&\qquad \le C\|\tilde X Pu\|_{L^2\left((t_0,t_2);H_h^{s-2}(\partial \Omega)\right)}+C_{\omega'}\|\tilde X u\|_{L^2\left((0,\e);H_h^{s-2}(\partial \Omega)\right)}+Ch^{\frac{1}{2}}\|\tilde X u|_{x^1=0}\|_{H^{s-\frac{1}{2}}_h(\partial \Omega)}\\
        &\qquad+Ch^N\|u\|_{H^1_h((0,\e);H_h^{-N}(\partial \Omega))}+Ch^N\| Pu\|_{L^2\left((0,\e);H_h^{-N}(\partial \Omega)\right)}+Ch^N\|u|_{x^1=0}\|_{H_h^{-N}(\partial\Omega)}
        \end{align*}
        and
        \begin{align*}
        &\|u\|_{H^1_h\left((0,\e);H_h^{s-1}(\partial \Omega)\right)}+\|u\|_{L^2\left((0,\e);H_h^{s}(\partial \Omega)\right)}\\
		&\qquad \le C\|Pu\|_{L^2\left((t_0,t_2);H_h^{s-2}(\partial \Omega)\right)}+C_{\omega'}\|u\|_{L^2\left((t_0,t_2);H_h^{s-2}(\partial \Omega)\right)}+Ch^{\frac{1}{2}}\|u(t_0)\|_{H^{s-\frac{1}{2}}_h(\partial \Omega)},
        \end{align*}
        which implies
        \begin{equation}
        \label{E_0 eqn 3}
        \begin{aligned}
        &\|X_2 u\|_{L^2\left((0,\e);H_h^{s}(\partial \Omega)\right)}\\
		&\qquad \le C\|Pu\|_{L^2\left((t_0,t_2);H_h^{s-2}(\partial \Omega)\right)}+C_{\omega'}\|u\|_{L^2\left((0,\e);H_h^{s-2}(\partial \Omega)\right)}+Ch^{\frac{1}{2}}\|\tilde X u|_{x^1=0}\|_{H^{s-\frac{1}{2}}_h(\partial \Omega)}\\
        &\qquad+Ch^N\|u|_{x^1=0}\|_{H_h^{-N}(\partial\Omega)}.
        \end{aligned}
        \end{equation}
        Plugging \eqref{E_0 eqn 2} and \eqref{E_0 eqn 3} into \eqref{E_0 eqn 1}, we have
		\begin{align*}
			&\|E_0u\|_{H_h^1(\partial \Omega)}^2 \leq 
            Ch^{-2}\|Pu\|_{L^2\left((0,\e);H_h^{s}(\partial \Omega)\right)}^2+C\|u\|_{L^2\left((0,\e);H_h^{s-2}(\partial \Omega)\right)}^2+Ch\|\tilde{X}u|_{x^1=0}\|_{H^{s+\frac{1}{2}}_h(\partial \Omega)}^2\\
            &\quad+\|\tilde X (hD_{x^1}-\Lambda )u\|_{H_h^{s}(\partial\Omega)}^2+Ch^N\|u|_{x^1=0}\|_{H_h^{-N}(\partial \Omega)}+C_Nh^N\|(hD_{x^1}-\Lambda )u\|_{H_h^{-N}(\partial\Omega)}^2,
		\end{align*}
        from which~\eqref{e:plane1} with $j=1$ follows.
		Next, suppose that~\eqref{e:plane1} holds for some $J\geq 1$. Then, let $X'\in\Psi^0_h(\partial\Omega)$ with 
		$$
		\WF(X)\subset \Ellip(X'),\qquad \WF(X')\subset \Ellip(R+\Lambda^*\Lambda).
		$$ 
		Then, using~\eqref{e:plane1} with $j=J$, and $(X,\tilde{X})$ replaced by $(X,X')$
		\begin{equation}
			\label{e:plane2}
			\begin{aligned}
				\|Xu\|_{H_h^{s+1}(\partial \Omega)}^2
				&\leq \|u\|_{L^2\left((0,\e);H^{s-2}_h(\partial \Omega)\right) }^2+Ch^{-2}\|Pu\|_{L^2\left((0,\e);H_h^{s}(\partial\Omega)\right)}^2+Ch^J\|X'u\|_{H_h^{s+1-\frac{J}{2}}(\partial\Omega)}^2\\
				&+\|X'(hD_{x^1}-\Lambda )u\|_{H_h^{s}(\partial\Omega)}^2+C_Nh^N\|u\|_{H_h^{-N}}^2+O(h^\infty)\|(hD_{x^1}-\Lambda )u\|_{H_h^{-N}(\partial \Omega)}^2.
			\end{aligned}
		\end{equation}
		Then, applying~\eqref{e:plane1} with $j=1$, $(X,\tilde{X})$ replaced by $(X',\tilde{X})$, and  $s$ replaced by $s-\frac{J}{2}$, we obtain
		\begin{equation}
			\label{e:plane3}
			\begin{aligned}
				\|X' u\|_{H_h^{s+1-\frac{J}{2}}(\partial \Omega)}^2
				&\leq \|u\|_{L^2\left((0,\e);H^{s-2-\frac{J}{2}}_h(\partial \Omega)\right) }^2+Ch^{-2}\|Pu\|_{L^2\left((0,\e);H_h^{s-\frac{J}{2}}(\partial\Omega)\right)}^2+Ch\|\tilde{X}u\|_{H_h^{s+\frac{1}{2}-\frac{J}{2}}(\partial\Omega)}^2\\
				&+\|\tilde{X}(hD_{x^1}-\Lambda )u\|_{H_h^{s}(\partial\Omega)}^2+C_Nh^N\|u\|_{H_h^{-N}(\partial\Omega)}^2+O(h^\infty)\|(hD_{x^1}-\Lambda )u\|_{H_h^{-N}(\partial\Omega)}^2.
			\end{aligned}
		\end{equation}
		Inserting~\eqref{e:plane3} in~\eqref{e:plane2} then implies~\eqref{e:plane1} with $j=J+1$. The proof of the lemma is then completed by the fact that 
		$$
		\|XhD_{x^1}u\|_{H_h^s(\partial\Omega)}\leq \|\tilde{X}u\|_{H_h^{s+1}(\partial\Omega)}+\|\tilde{X}(hD_{x^1}-\Lambda)u\|_{H_h^s(\partial\Omega)}+C_Nh^N\|u|_{x^1=0}\|_{H_h^{-N}(\partial\Omega)}.
		$$
	\end{proof}

	\noindent{\bf{Estimates for the boundary traces of the transmission problem}}
    We are now in a position to obtain apriori estimates for the problem~\eqref{e:transmission}. We start in the simpler situation when~\eqref{e:noPlasmons} holds.
	\begin{lem}
		\label{thm easy}
		 Suppose that~\eqref{e:noPlasmons} holds. Then for all $M>0$, $s\in \mathbb{R}$, and $\e>0$, there are $C,h_0>0$ such that for all $0<h<h_0$, $|z-1|\leq Ch$ 
        and $u\in L^2\left((0,\e);H_h^{s-2}(\partial \Omega)\right)$ solutions to~\eqref{e:transmission}, we have
		\begin{equation*}
			\label{thm easy eqn}
			\| u\|_{H^{s+1}_h(\partial \Omega)}+\|hD_{x^1} u\|_{H^s_h(\partial \Omega)}\le C\left(\| u\|_{L^2\left((0,\e);H^{s-2}_h(\partial \Omega)\right)}+\|g\|_{H^s_h(\partial \Omega)}\right)
		\end{equation*}
		for $0<h<h_0$.
	\end{lem}
	\begin{proof}
        Recall that in Fermi normal coordinates
		$$
		P_{\exterior}-z^2=(hD_{x^1})^2+ha(x)hD_{x^1}-R(x,hD_{x'})
		$$
		with $\sigma(R)=1-|\xi'|_{g_\exterior}^2$.
		
		Let $\Lambda= \rmi \frac{\tau}{\rho_{\exterior}}\Lambda_{\inner}(z)$ and recall that $\sigma(\Lambda_{\inner}(z))=\rho_{\inner}\sqrt{|\xi'|_{g_\inner }^2+1}$. Then,
		$$
		(hD_{x^1}-\Lambda)u|_{x^1=0}=-\rmi\left(h\partial_{\nu}+\tfrac{\tau}{\rho_{\exterior}}\Lambda_{\inner}(z)\right)u=\rmi\tfrac{\tau}{\rho_{\exterior}}g,
		$$
		and 
		$$
		\sigma(R+\Lambda^*\Lambda)=1-|\xi'|_{g_{\exterior}}^2+\left(\tau\rho_{\exterior}^{-1}\right)^2\rho_{\inner}^2(|\xi'|_{g_{\inner}}^2+1)=\rho_{\exterior}^{-2}\left(\rho_{\exterior}^2+\tau^2\rho_{\inner}^2+\tau^2\rho_{\inner}^2|\xi'|_{g_{\inner}}^2-\rho_{\exterior}^2|\xi'|_{g_{\exterior
        }}^2\right).
		$$
		In particular, $(\tau\rho_\inner)^2|\xi'|_{g_\inner}^2>\rho_\exterior^2|\xi'|_{g_\exterior}^2 \quad \text{for all $\xi'\in T^*\partial \Omega$}$ implies that there exists a positive constant, $c_1$ such that 
        $$
        c_1|\xi'|_{g_\exterior}<(\tau \rho_{\inner})^2|\xi'|^2_{g_{\inner}}-\rho_{\exterior}^2|\xi'|^2_{g_{\exterior}}, \quad \text{for all $\xi'\in T^*\partial \Omega$}.
        $$
        Now we have, for some constant $c_2$,
		$$
		\sigma(R+\Lambda^*\Lambda)>c_2\langle |\xi'|_{g_\exterior}\rangle^2>0.
		$$
		Hence, Lemma~\ref{l:intByPartsEstimate} yields
		\begin{align*}
			\|u|_{x^1=0}\|_{H_h^{s+1}(\partial \Omega)}+\|hD_{x^1}u|_{x^1=0}\|_{H_h^s(\partial\Omega)}&
			\leq C \|u\|_{L^2\left((0,\e);H^{s-2}_h(\partial \Omega)\right)}+C\|g\|_{H_h^{s}(\partial\Omega)}\\
			&\qquad+C_Nh^N\|u|_{x^1=0}\|_{H_h^{-N}(\partial\Omega)}+c_Nh^N\|g\|_{H_h^{-N}(\partial\Omega)}\\
			&\leq C\|u\|_{L^2\left((0,\e);H^{s-2}_h(\partial \Omega)\right)}+C\|g\|_{H_h^s(\partial\Omega)}.
		\end{align*}
	\end{proof}

	Next, we consider the case of~\eqref{e:plasmons}.
	\begin{lem}
		\label{thm harder}
	 Suppose that~\eqref{e:plasmons} holds. Then for all $M>0$, $s\in \mathbb{R}$, $\e>0$, and$X,\tilde{X}\in\Psi^0_h(\partial\Omega)$ satisfying 
		\begin{equation*}
        \label{e:goodWavefront}
		\WF(X)\cap \left\{\rho_{\exterior}^2|\xi'|_{g_\exterior }^2-\tau^2\rho_{\inner}^2|\xi'|_{g_{\inner}}^2=\rho_{\exterior}^2+\tau^2\rho_{\inner}^2\right\}=\emptyset
		\end{equation*}
		and $\WF(X)\subset \Ellip(\tilde{X})$, there are $C,h_0>0$ such that for all $0<h<h_0$, $|z-1|\leq Mh$ 
        and  all $u\in L^2\left((0,\e);H_h^{s-2}(\partial \Omega)\right)$ solutions to~\eqref{e:transmission} we have
		\begin{equation*}
		\label{thm harder eqn}
		\begin{aligned}
		    &\| Xu\|_{H^{s+1}_h(\partial \Omega)}+\|XhD_{x^1} u\|_{H^s_h(\partial \Omega)}
            \\
            &\qquad \le C\left(\|u\|_{L^2\left((0,\e);H_h^{s-2}(\partial \Omega)\right)}+\|\tilde{X}g\|_{H^s_h(\partial \Omega)}+C_Nh^N\|g\|_{H_h^{-N}(\partial \Omega)}+C_Nh^N\|u\|_{H_h^{-N}(\partial\Omega)}\right)
		\end{aligned}
		\end{equation*}
		for $0<h<h_0$.
	\end{lem}
	\begin{proof}
		As before, we need only consider $R+\Lambda^*\Lambda$ with $\Lambda=
    \rmi\tau\rho_{\exterior}^{-1} \Lambda_{\inner}(z)$. 
		Observe that if
		$$
		\sigma(R+\Lambda^*\Lambda)=\rho_{\exterior}^{-2}\left(\rho_{\exterior}^2+\tau^2\rho_{\inner}^2+\tau^2\rho_{\inner}^2|\xi'|_{g_{\inner}}^2-\rho_{\exterior}^2|\xi'|_{g_{\exterior
        }}^2\right)=0,
		$$
		then
		$$
		\rho_{\exterior}^2|\xi'|_{g_\exterior }^2-\tau^2\rho_{\inner}^2|\xi'|_{g_{\inner}}^2=\rho_{\exterior}^2+\tau^2\rho_{\inner}^2.
		$$
		Hence, Lemma~\ref{l:intByPartsEstimate} yields
		\begin{align*}
			\|Xu|_{x^1=0}\|_{H_h^{s+1}(\partial \Omega)}+\|XhD_{x^1}u|_{x^1=0}\|_{H_h^s(\partial\Omega)}&
			\leq C \|u\|_{L^2\left((0,\e);H_h^{s-2}(\partial \Omega)\right)}+C\|\tilde{X}g\|_{H_h^{s}(\partial\Omega)}\\
			&\qquad+C_Nh^N\|u|_{x^1=0}\|_{H_h^{-N}(\partial\Omega)}+c_Nh^N\|g\|_{H_h^{-N}(\partial\Omega)}.
		\end{align*}
	\end{proof}

Finally, we need an estimate on the high frequencies of a solution to $(P_\exterior-z^2)u=0$ in terms of the traces of $u$ on the boundary.
\begin{lem}
\label{l:highFreq}
Let $M>0$, $N>0$,  $\chi_0,\chi_1\in C_c^\infty(\mathbb{R}^d)$ with $\chi_0\equiv 1$ near $\partial\Omega$, $\supp \chi_0\cap \supp (1-\chi_1)=\emptyset$, $\phi \in C_{c}^\infty(\mathbb{R})$ with 
    \begin{equation}
    \label{e:ellipticCut}
    \supp (1-\phi)\cap \{ |\xi|_{g_{\exterior}}\,:\, \exists x\in \Omega_{\exterior}\text{ such that }|\xi|_{g_{\exterior}}^2\leq 2\} = \emptyset,
    \end{equation}
    and define $\Phi:=\Op(\phi(|\xi|_{g_{\exterior}}))$. Then there are $C,h_0>0$ such that for all $0<h<h_0$, $|z-1|<Mh$ 
    and $u\in L^2_{\loc}(\Omega)$ satisfying
$$
(P_{\exterior}-z^2)u=0,
$$
we have
$$
\|(1-\Phi)\chi_0 u\|_{L^2}\leq C\left(h^{\frac{1}{2}}\left(\|u\|_{L^2(\partial\Omega)}+\|hD_{\nu_{\exterior}}u\|_{L^2(\partial\Omega)}\right)+h^N\|\chi_1 u\|_{L^2(\Omega)}\right).
$$
\end{lem}
\begin{proof}
Let $\tilde u\in L^2_{\loc}(\mathbb{R}^d)$ $\tilde{u}:=1_{\Omega}u$. Then gives
	\begin{equation}
		\label{extension eqn}
		(P_{\exterior}-z^2) \tilde u= h^2\partial_{\nu_{\exterior}}^* \delta_{\partial \Omega} \otimes \left(\rho_{\exterior}u|_{\partial \Omega}\right)-h \delta_{\partial \Omega} \otimes \left(\rho_{\exterior}h\partial_{\nu_{\exterior}} u|_{\partial \Omega}\right),
	\end{equation}
	where $\langle \partial_{\nu_{\exterior}}^* \delta_{\partial \Omega}\otimes \left(u|_{\partial \Omega}\right), \varphi\rangle=\int_{\partial \Omega} u\partial_{\nu_{\exterior}}\varphi \der S$ and $\langle \delta_{\partial \Omega}\otimes \left(h\partial_{\nu_{\exterior}} u|_{\partial \Omega}\right),\varphi\rangle=\int_{\partial \Omega} \left(h\partial_{\nu_{\exterior}} u\right) \varphi \der S$ for $\varphi \in C_{c}^\infty(\R^d)$. 

    Let $\tilde{\chi}\in C_c^\infty(\mathbb{R}^d)$ with $\supp (1-\tilde{\chi})\cap \supp \chi_0=\supp \tilde{\chi}\cap \supp (1-\chi_1)=\emptyset$.
      Since $\WF(I-\Phi)\subset \Ellip(P_{\exterior}-z^2)$, there is $E\in \Psi^{-2}_h(\mathbb{R}^d)$ with $\WF(E)\cap \supp (1-\chi_1)=\emptyset$ such that 
	\begin{align*}
	(I-\Phi) \chi_0 \tilde u&=E(P_{\exterior}-z^2)\tilde{\chi}\tilde u+O(h^\infty)_{\Psi_h^{-\infty}}\tilde{\chi}\tilde u\\
    &= E\tilde{\chi}(P_{\exterior}-z^2)\tilde{u} +E[P,\tilde{\chi}]\chi_1u +O(h^\infty)_{\Psi_h^{-\infty}}\chi_1\tilde{u}\\
    &=E\tilde{\chi}(P_{\exterior}-z^2)\tilde{u}+O(h^\infty)_{\Psi_h^{-\infty}}\chi_1\tilde{u}.
	\end{align*}
 Since $ E\in \Psi^{-2}_h(\R^d)$, one has
	\begin{equation}
		\label{boundary defect measure thm eqn 1}
		\begin{aligned}\|(1-\Phi)\chi_0\tilde u\|_{L^2(\R^d)}
        &\leq \| (P_{\exterior}-z^2)\tilde u\|_{H_h^{-2}(\R^d)}+C_Nh^N\|\chi_1 u\|_{L^2(\Omega)}.
        \end{aligned}
	\end{equation}
	Using \eqref{extension eqn}, we know that
	\begin{equation}
		\label{boundary defect measure thm eqn 2}
		\begin{aligned}\|(P_{\exterior}-z^2)\tilde u\|_{H^{-2}_h(\R^d)}&\le Ch\left(\left\|h\partial_{\nu_\exterior}^* \delta_{\partial \Omega} \otimes \left(u|_{\partial \Omega}\right)\right\|_{H^{-2}_h(\R^d)}+\left\|\delta_{\partial \Omega} \otimes \left(h\partial_{\nu_{\exterior}} u|_{\partial \Omega}\right)\right\|_{H^{-2}_h(\R^d)}\right)
		\\
		&\le C h^{\frac{1}{2}}\left(\|u\|_{L^2(\partial \Omega)}+\|h\partial_{\nu_{\exterior}} u\|_{L^2(\partial \Omega)}\right).
        \end{aligned}
	\end{equation}
	Combining with \eqref{boundary defect measure thm eqn 1} and \eqref{boundary defect measure thm eqn 2}, one has
	\begin{equation*}
		\label{boundary defect measure thm eqn 3}
		\|(1-\Phi)\chi_0 \tilde u\|_{L^2(\R^d)}\le Ch^{\frac{1}{2}}\left(\|u\|_{L^2(\partial \Omega)}+\|hD_{\nu_{\exterior}} u\|_{L^2(\partial \Omega)}\right)+C_Nh^N\|\chi_1 u\|_{L^2(\Omega)},
	\end{equation*}
    which completes the proof.
\end{proof}
\subsection{Resolvent estimates - the absence of plasmon resonances}
\label{Resolvent estimates no plasmon}
This section will prove Theorem~\ref{t:states}. In particular, we obtain the desired resolvent estimates under the condition \eqref{e:noPlasmons}.

We start with a lemma that we use repeatedly to prove our estimates. It applies the relevant propagation of defect measures results to obtain estimates.
\begin{lem}
\label{l:basicPropagation}
Let $X,\tilde{X}\in \Psi^{0}_h(\partial\Omega)$ with
$$
\left\{\rho_{\exterior}^2(|\xi'|_{g_{\exterior}}^2-1)-\tau^2(|\xi'|_{g_{\inner}}^2+1)=0\right\} \cap \WF(X)=\emptyset,
$$
and $\WF(X)\cap (\WF(I-\tilde{X}))=\emptyset$. Then, for any $M>0$, $N>0$ and $\chi\in C_c^\infty(\overline{\Omega}_{\exterior})$, there are $h_0>0$ and $C>0$ such that for all $0<h<h_0$, $|z-1|<Mh$, and $u\in L^2_{\loc}(\Omega)$ satisfying
\begin{equation}
\label{e:uSolve}
		\left\{
		\begin{aligned}
			&(P_{\exterior}-z^2)u=0  &\text{in} \quad \Omega,\\
			&\rho_{\exterior}hD_{\nu_{\exterior}}u -\tau \Lambda_{\inner}(z) u=g  &\text{on} \quad \partial \Omega,\\
			&\text{$u$ is $z/h$ outgoing},
		\end{aligned}
		\right.
	\end{equation}
	 we have 
	\begin{equation*}
    \begin{aligned}
		&\|hD_{\nu_{\exterior}}u\|_{H_h^{\frac{1}{2}}(\partial\Omega)}+\|u\|_{H_h^{\frac{3}{2}}(\partial\Omega)}+\|\chi u\|_{H_h^2(\Omega_\exterior)}
        \\
        &\le C \left(\|\tilde{X}g\|_{H^\frac{1}{2}_h(\partial \Omega)}+\|(I-X)u\|_{H_h^{\frac{3}{2}}(\partial\Omega)}+h^N\|g\|_{H_h^{-N}(\partial\Omega)}\right).
        \end{aligned}
	\end{equation*}
\end{lem}
\begin{rem}
\label{r:rescale}
Notice that for $a>0$, we have $\Lambda_{\inner,h}(az)=a^{-1}\Lambda_\inner(z,ah)$ and hence, by rescaling $h$, we see that, provided 
$$
\{\rho_{\exterior}^2(|\xi'|_{g_{\exterior}}^2-z^2)-\tau^2(|\xi'|_{g_{\inner}}^2+z^2)=0\} \cap \WF(X)=\emptyset,
$$
for $z\in[1-\e_0,1+\e_0]$, Lemma~\ref{l:basicPropagation} continues to hold for $z\in[1-\e_0,1+\e_0]+\rmi [-Mh,Mh]$.
\end{rem}
\begin{proof}
	 We first claim it is enough to show that for any $\chi\in C_c^\infty(\overline{\Omega}_{\exterior})$,  
    \begin{equation}
    \label{e:toContradict1}
    \|\chi u\|_{L^2(\Omega_\exterior)}\leq C\left(\|\tilde{X}g\|_{H^{\frac{1}{2}}_h(\partial\Omega)}+\|(I-X)u\|_{L^2(\partial\Omega)}\right)+h^N\|g\|_{H_h^{-N}(\partial\Omega)}.
    \end{equation}

    Indeed, let $\chi_1\in C_c^\infty(\overline{\Omega}_\exterior)$ with $\supp \chi\cap \supp (1-\chi_1)=\emptyset$. 
     By Lemmas~\ref{thm easy} and Lemma~\ref{thm harder}, we have
	\begin{equation*}
        \begin{aligned}
		&\|XhD_{\nu_{\exterior}} u\|_{H_h^{\frac{1}{2}}(\partial\Omega)}+\|Xu\|_{H_h^{\frac{3}{2}}(\partial\Omega)}\\
        &\leq C\left(\|\tilde{X}g\|_{H_h^{\frac{1}{2}}(\partial\Omega)}+\|\chi_1 u\|_{L^2(\Omega)}+C_Nh^N\|g\|_{H_h^{-N}(\partial\Omega)}+C_Nh^N\|u\|_{H_h^{-N}(\partial\Omega)}\right).
        \end{aligned}
	\end{equation*}
    Hence, using Lemma~\ref{l:neumannBounds} and Proposition~\ref{G bound thm} to control the normal derivative and $\chi u$, we have
    \begin{equation}
		\label{e:bTrace1}
        \begin{aligned}
		&\|hD_{\nu_{\exterior}} u\|_{H_h^{\frac{1}{2}}(\partial\Omega)}+\|u\|_{H_h^{\frac{3}{2}}(\partial\Omega)}+\|\chi u\|_{H_h^2(\Omega)}\\
        &\leq C\left(\|\tilde{X}g\|_{H_h^{\frac{1}{2}}(\partial\Omega)}+C\|(I-X)u\|_{H_h^{\frac{3}{2}}(\partial \Omega)}+\|\chi_1 u\|_{L^2(\Omega_\exterior)}+C_Nh^N\|g\|_{H_h^{-N}(\partial\Omega)}\right).
        \end{aligned}
	\end{equation}
    
    We will prove~\eqref{e:toContradict1} by contradiction.
    Suppose that inequality \eqref{e:toContradict1} is false. That is, there exist sequences of solutions $u_j=u(h_j)$ and $z_j$ such that
	\begin{equation}
		\label{contradiction eqn n<1}
        \begin{gathered}
\|\chi u_j\|_{L^2(\Omega_\exterior)}=1,\quad
			, \text{ and }|z_j-1|<Mh_j,\\
            \|\tilde{X}g_j\|_{H^\frac{1}{2}_h(\partial \Omega)}+\|(I-X)u_j\|_{H_h^{\frac{3}{2}}(\partial\Omega)}+h_j^{-N}\|(I-\tilde{X})g_j\|_{H_h^{-N}(\partial\Omega)}=o(1).
            \end{gathered}
	\end{equation}

    Let $\chi_0,\chi_1\in C_c^\infty(\overline{\Omega}_\exterior)$ with $\chi_0\equiv 1$ in a neighborhood of $\partial\Omega$ and $\supp \chi_0\cap \supp (1-\chi)=\supp \chi\cap \supp (1-\chi_1)=\emptyset$.  Observe that 
    $$
    (P_\exterior -z^2)(1-\chi_0)u_j=[\chi_0,P_{\exterior}]\chi u_j,
    $$
    and hence 
    $$
    (1-\chi_0)u_j=R_{\exterior}[\chi_0,P_{\exterior}] \chi u_j.
    $$
    So that, by~\eqref{e:exteriorResolve}
    $$
    \|(1-\chi_0)u_j\|_{L^2(\Omega_\exterior)}\leq C\|[\chi_0,P_{\exterior}]\chi u_j\|_{L^2(\Omega_\exterior)}\leq C\|\chi u_j\|_{L^2(\Omega_\exterior)}\leq C,
    $$
    where in the second-to-last inequality, we have used that by elliptic regularity,
    $$
    \|u_j\|_{H_h^1(\supp \partial \chi_0)}\leq \|\chi u_j\|_{L^2(\Omega_\exterior)}.
    $$
    In particular, 
    \begin{equation}
    \label{e:slightlyBigger}
    \|\chi_1 u_j\|_{L^2(\Omega_\exterior)}\leq C.
    \end{equation}

	From Section \ref{Semiclassical defect measures} and the first condition in \eqref{contradiction eqn n<1}, up to extracting a subsequence, we may assume that there is a defect measure $\mu$ associated with $u_j$ (See \eqref{interior measure}). Furthermore, by~\eqref{e:bTrace1}, the boundary measures also exist. Since $u_j$ is outgoing, we have
	\begin{equation*}
    \label{e:outgoing}
		\WF(u_j)\cap\{(x,\xi):|x|\ge r_0\}\subset S_+:=\{(x,\xi):|x|\ge r_0,\langle x, \xi \rangle> 0\}
	\end{equation*}
	for some $r_0>0$.

	That is
	\begin{equation}
		\label{contradiction eqn -}
		\mu(\chi^2 \Psi_-)=0,
	\end{equation}
	where $\supp(\Psi_-)\subset S_-:=\{(x,\xi):|x|\ge r_0,\langle x, \xi \rangle< 0\}$.

	By the second condition of \eqref{contradiction eqn n<1}  and Theorem \ref{measure invariant thm}, we have
	\begin{equation*}
		\pi_*\mu(q_{(x_0,\xi_0)}\circ \varphi_t)=\pi_*\mu(q_{(x_0,\xi_0)}),
	\end{equation*}
	where $q\in C^\infty_{c}({}^bT^*\Omega_\exterior;\R)$ and $\WF(q_{(x_0,\xi_0)})\subset B_{(x_0,\xi_0)}(\delta)\cap S_-$ with $B_{(x_0,\xi_0)}(\delta)$ being the ball centered at $(x_0,\xi_0)$ with radius $\delta$. Since $\partial \Omega$ is non-trapping, there exists $t\ge0$ and $\varphi_t(x_0,\xi_0)=(x_1,\xi_1)\in S_+$. This shows
	\begin{equation*}
		\label{contradiction eqn +}
		\mu(\chi^2 \Psi_+)=0.
	\end{equation*}
	Together with \eqref{contradiction eqn -}, we have
	\begin{equation}
		\label{contradiction eqn -+}
		\mu(\chi^2)=0.
	\end{equation}

    Let $\phi\in C_c^\infty(\mathbb{R}^d)$ satisfy~\eqref{e:ellipticCut} and $\Phi=\phi(hD)$. Then, by~\eqref{e:slightlyBigger} and~\eqref{e:bTrace1} together with Lemma~\ref{l:highFreq}, we have
	\begin{equation*}
		\|(1-\Phi) \chi(u1_{\Omega_\exterior})\|_{L^2(\R^d)}\le Ch^{\frac{1}{2}}+C_Nh^N.
	\end{equation*}
	In particular,
    using~\eqref{contradiction eqn -+}, 
	\begin{equation*}
    \begin{aligned}
		1=\lim_{j\to \infty}\|\chi (u_j 1_{\Omega_\exterior})\|_{L^2(\R^d)}&\le\lim_{j\to \infty}\|\Phi\chi (u_j 1_{\Omega_\exterior})\|_{L^2(\R^d)}+\lim_{j\to \infty}\|(I-\Phi)\chi (u_j 1_{\Omega_\exterior})\|_{L^2(\R^d)}\\
        &=\mu(\chi^2\phi^2)\leq \mu(\chi^2)=0,
        \end{aligned}
	\end{equation*}
	which is a contradiction.
    \end{proof}

The next theorem gives the estimates~\eqref{e:noStatesToProve} and hence proves Theorem~\ref{t:noStates}.
\begin{theorem}
	\label{boundary defect measure thm n<1}
	Suppose that~\eqref{e:noPlasmons} holds. Then for any $M>0$ and $\chi \in C_c^\infty(\overline{\Omega}_{\exterior})$, there are $h_0, C>0$ such that for all $0<h<h_0$, $|z-1|<Mh$,  and $u\in L^2_\loc(\Omega)$ satisfying~\eqref{e:uSolve}
	 we have 
	\begin{equation*}
		\label{L2 control of u n<1}
		\|hD_{\nu_{\exterior}}u\|_{H_h^{\frac{1}{2}}(\partial\Omega)}+\|u\|_{H_h^{\frac{3}{2}}(\partial\Omega)}+\|\chi u\|_{H_h^2(\Omega_\exterior)}\le C \|g\|_{H^\frac{1}{2}_h(\partial \Omega)}.
	\end{equation*}
\end{theorem}
\begin{proof}
The theorem follows from Lemma~\ref{l:basicPropagation} with $X,\tilde{X}=I$. 
\end{proof}

 \subsection{Resolvent estimates and surface plasmons}
\label{Resolvent estimates}
In this subsection, we prove Theorem~\ref{t:states}. In particular, we obtain the estimates~\eqref{e:statesToProve} under the condition \eqref{e:plasmons} and hence prove Theorem~\ref{t:states}.
\begin{theorem}
	\label{boundary defect measure thm n>1}
	Suppose that~\eqref{e:plasmons} holds. For all $M,N>0$, $\chi\in C_c^\infty(\overline{\Omega}_\exterior)$, $X\in \Psi^{\comp}_h(\partial \Omega)$ with 
	\begin{equation}
    \label{e:nAway}
	\WF(I-X)\cap \left\{\rho_{\exterior}^2|\xi'|_{g_\exterior }^2-\tau^2\rho_{\inner}^2|\xi'|_{g_{\inner}}^2=\rho_{\exterior}^2+\tau^2\rho_{\inner}^2\right\}=\emptyset,\qquad \WF(X)\subset \{|\xi'|_{g_\exterior }>1\}
	\end{equation}
	there are $C, h_0>0$ such that for all $0<h<h_0$, $|z-1|<Mh$, $\Im z<-h^N$, and $u\in L^2_{\loc}(\Omega)$ satisfying~\eqref{e:uSolve}
	we have
	\begin{equation*}
		\label{invertible eqn n>1}
		\|\chi u\|_{H^2_{h}(\Omega_\exterior)}\le C|\Im z|^{-1}\|X g\|_{L^2(\partial \Omega)}
		+C\|(I- X)g\|_{H^\frac{1}{2}_h(\partial \Omega)}.
	\end{equation*}
\end{theorem}
\begin{proof}


    Let $X_i\in \Psi^{\comp}_h(\partial\Omega)$, $i=0,1,2$ with $\WF(X_i)\cap \WF(I-X_{i+1})=\emptyset$, $i=0,1$ and
	\begin{equation*}
	\WF(I-X_0)\cap \left\{\rho_{\exterior}^2|\xi'|_{g_\exterior }^2-\tau^2\rho_{\inner}^2|\xi'|_{g_{\inner}}^2=\rho_{\exterior}^2+\tau^2\rho_{\inner}^2\right\}=\emptyset,\qquad \WF(X_2)\subset \Ellip(X).
	\end{equation*}

    Using that $X_2\Lambda_{\exterior},\Lambda_{\exterior}X_2\in \Psi^{\comp}_h(\partial\Omega)$, and the wavefront set properties of $X_1$, 
	$$
	\WF([\Lambda_{\exterior}-\tau\Lambda_{\inner},X_1])\subset \Ellip(X(\Lambda_\exterior -\tau \Lambda_{\inner}))
	$$
	and hence there is $E\in h\Psi^\comp_h(\partial\Omega)$ such that 
	$$
	[\Lambda_{\exterior}-\tau\Lambda_{\inner},X_1]=EX_2(\Lambda_\exterior -\tau\Lambda_{\inner}))+O(h^\infty)_{\Psi_h^{-\infty}}. 
	$$

    	Thus, by Proposition~\ref{p:imagPart}
	\begin{align*}
		&|\Im z|\|X_1u\|_{L^2(\partial \Omega)}^2\\
        &\leq |\langle( \Lambda_{\exterior}-\tau\Lambda_{\inner})X_1u,X_1u\rangle|+C_Nh^N\|u\|_{L^2(\partial\Omega)}^2\\
		&\leq|\langle X_1g,X_1u\rangle |+|\langle [\Lambda_{\exterior}-\tau\Lambda_{\inner},X_1]u,X_1nu\rangle|+C_Nh^N\|u\|_{L^2(\partial\Omega)}^2\\
		&=|\langle X_1 g,X_1u\rangle| +|\langle EX_2 g+O(h^\infty)_{\Psi_h^{-\infty}}u,X' u\rangle|+C_Nh^N\|u\|_{L^2(\partial\Omega)}^2\\
		&\leq \left(\|X_2g\|_{L^2(\partial \Omega)}+O(h^\infty)\left(\|g\|_{H_h^{-N}(\partial\Omega)}+\|u\|_{H_{h}^{-N}(\partial\Omega)}\right)\right)\|X_1u\|_{L^2(\partial \Omega)}+C_Nh^N\|u\|_{L^2(\partial\Omega)}^2.
	\end{align*}
	Hence, using that $X_1\in\Psi_h^{\comp}(\partial\Omega)$ for the first inequality
	\begin{equation}
		\label{e:here}
        \begin{aligned}
		|\Im z|\|X_1u\|_{H_h^{\frac{3}{2}}(\partial\Omega)}&\leq C|\Im z|\|X_1u\|_{L^2(\partial\Omega)}+C_Nh^N\|u\|_{L^2(\partial\Omega)}\\
        &\leq C\|X_2g\|_{L^2(\partial\Omega)} +C_Nh^N\left(\|u\|_{L^2(\partial\Omega)}+\|g\|_{H_h^{-N}(\partial\Omega)}\right).
        \end{aligned}
	\end{equation}

    Now, by Lemma~\ref{l:basicPropagation} with $X=I-X_1$ and $\tilde{X}=I-X_0$, 
\begin{align*}
\|hD_{\nu_{\exterior}}u\|_{H_h^{\frac{1}{2}}(\partial\Omega)}&+\|u\|_{H_h^{\frac{3}{2}}(\partial\Omega)}+\|\chi u\|_{H_h^2(\Omega_\exterior)}\\
&\le C \left(\|(I-X_0)g\|_{H^\frac{1}{2}_h(\partial \Omega)}+\|X_1 u\|_{H_h^{\frac{3}{2}}(\partial\Omega)}+h^N\|g\|_{H_h^{-N}(\partial\Omega)}\right)\\
&\leq C \left(\|(I-X_0)g\|_{H^\frac{1}{2}_h(\partial \Omega)}+|\Im z|^{-1}\|X_2g\|_{H_h^{\frac{1}{2}}(\partial\Omega)}\right)\\
&\leq C \left(\|(I-X)g\|_{H_h^{\frac{1}{2}}(\partial\Omega)}+|\Im z|^{-1}\|Xg\|_{H_h^{\frac{1}{2}}(\partial\Omega)}\right),
\end{align*}
which completes the proof.
\end{proof}

\subsection{The plasmonic nature of resonances}. 
In this subsection, we show that all resonances close to the real axis are surface plasmons. In particular, we prove Theorem~\ref{t:plasmonic}.
\begin{lem}
	\label{l:plasmonic}
	Suppose that~\eqref{e:plasmons} holds. Then for all $M>0$, $\chi \in C_c^\infty(\overline{\Omega}_{\exterior})$ with $\chi=1$ in a neighborhood of $\partial \Omega$ and $\psi \in C_c^\infty(\Omega_{\exterior})$ (i.e.  $\supp \psi\cap \partial\Omega=\emptyset$) the following holds. There is $c>0$ such that for all $|z(h)-1|\leq Mh$ and $u=u(h)\in L^2_\loc(\Omega)$ satisfies
	\begin{equation*}
		\begin{cases}
			(P_{\exterior}-z^2)u=0  &\text{in } \Omega_{\exterior},\\			
            (\rho_{\exterior}h\partial_{\nu_{\exterior}} -\tau \Lambda_{\inner}(z)) u=0  &\text{on } \partial \Omega,\\
            \| u\|_{L^2(\partial \Omega)}=1,\\
            \text{$u$ is $z/h$ outgoing},
		\end{cases}
	\end{equation*}
then 
$$
ch^{\frac{1}{2}}\leq\|\chi u\|_{H_h^2(\Omega_{\exterior})}=O\left(h^{\frac{1}{2}}\right), \quad \text{and} \quad \|\psi u\|_{H_h^2(\Omega_{\exterior})}=O(h^\infty).
$$
\end{lem}
\begin{proof}
Let $X_j\in\Psi_h^{\comp}(\partial\Omega)$, $j=0,1,2,3$ satisfy~\eqref{e:nAway} with $\WF(I-X_{j+1})\cap \WF(X_j)=\emptyset.$
Then from Proposition \ref{symbol of DtN exterior}, and the elliptic parametrix construction, there is $E\in \Psi^{\comp}_h(\partial\Omega)$ such that 
$$
(X_2-X_0)=E(\Lambda_0-\tau\Lambda_\inner)+O(h^\infty)_{\Psi_h^{-\infty}}.
$$
Hence, using that 
$$
\WF([\Lambda_{\exterior}-\tau\Lambda_\inner,X_1])\subset \WF(X_1)\cap \WF(I-X_1)\subset \Ellip(X_2-X_0),
$$
we have
\begin{equation*}
\begin{aligned}
\|(\rho_{\exterior}h\partial_{\nu_{\exterior}}-\tau \Lambda_{\inner})&G_\exterior X_1u\|_{H_h^N(\partial \Omega)}=\|(\Lambda_\exterior-\tau \Lambda_{\inner})X_1u\|_{H_h^N(\partial \Omega)}\\
& =\|[(\Lambda_\exterior-\tau\Lambda_{\inner}),X_1]u\|_{H_h^N(\partial \Omega)}\leq Ch\|(X_2-X_0)u\|_{H_h^N(\partial \Omega)}+O(h^\infty)\\
&=Ch\|E(\Lambda_{\exterior}-\tau\Lambda_{\inner})u\|_{H_h^N(\partial \Omega)}+O(h^\infty)=O(h^\infty).
\end{aligned}
\end{equation*}

Define $w:=u-G_{\exterior}X_1 u$.
Then, we have
\begin{equation*}
		\begin{cases}
			(P_{\exterior}-z^2)
            w=0  &\text{in } \Omega_{\exterior},\\			
            g:=\rho_{\exterior}h\partial_{\nu_{\exterior}}w -\tau\Lambda_{\inner}(z) w=O(h^\infty) &\text{on } \partial \Omega,\\
            w=(I-X_1)u&\text{on } \partial\Omega,\\
            \text{$w$ is $z/h$ outgoing}.
		\end{cases}
	\end{equation*}
    Now, by Lemma~\ref{l:basicPropagation} with $X=I-X_0$ and $\tilde{X}=I$, 
    \begin{align*}
    		\|hD_{\nu_{\exterior}}w\|_{H_h^{\frac{1}{2}}(\partial\Omega)}+\|w\|_{H_h^{\frac{3}{2}}(\partial\Omega)}+\|\chi w\|_{H_h^2(\Omega_{\exterior})}&\le C \left(\|g\|_{H^\frac{1}{2}_h(\partial \Omega)}+\|X_0w\|_{H_h^{\frac{3}{2}}(\partial\Omega)}\right)\\
            &=C \left(\|g\|_{H^\frac{1}{2}_h(\partial \Omega)}+\|X_0(I-X_1)u\|_{H_h^{\frac{3}{2}}(\partial\Omega)}\right)\\
            &= O(h^\infty).
            \end{align*}
    
    Using Lemma~\ref{l:dtNDescriptionExterior} to bound $\|\chi G_{\exterior}X_1u\|_{H_h^2(\Omega_{\exterior})}$, we obtain
    $$
    \|\chi u\|_{L^2(\Omega_{\exterior})}\leq C\|\chi w\|_{H_h^2(\Omega_{\exterior})}+\|\chi G_{\exterior}X_1u\|_{H_h^2(\Omega_{\exterior})}\leq Ch^{\frac{1}{2}}.
    $$
    Next, using Lemma~\ref{l:dtNDescriptionExterior} again, observe that $\psi\in C_c^\infty(\Omega_{\exterior})$, one has
    $$
    \|\psi G_{\exterior}X_1u\|_{H_h^2(\Omega_{\exterior})}=O(h^\infty).
    $$
    Finally, observe that
    $$
    \|\psi u\|_{H_h^2(\Omega_{\exterior})}\leq \|\psi w\|_{H_h^2(\Omega_{\exterior})}+\|\psi G_{\exterior}X_1u\|_{H_h^2(\Omega_{\exterior})}=O(h^\infty),
    $$
    which completes the proof.
\end{proof}

We can now complete the proof of Theorem~\ref{t:plasmonic}
\begin{proof}[Proof of Theorem~\ref{t:plasmonic}]
Let $\psi\in C_c^\infty(\mathbb{R}^d)$ with $\supp \psi\cap \partial\Omega=\emptyset$. 
Suppose that $\lambda_j\in\mathcal{R}(P)$ with $\Re \lambda_j\to \infty$ and $|\Im\lambda_j|\leq C$ and $u_{\lambda_j}$ satisfies~\eqref{e:mainProblem} with $f_{\inner}=f_{\exterior}=0$, and $\|u_{\lambda_j}\|_{L^2(\partial\Omega)}=1$.

Set $h_j=\Re \lambda_j^{-1}$. Then Lemma~\ref{l:plasmonic} applies to $u_{\lambda_j,\exterior}$ and hence
$$
\|\psi u_{\lambda_j,\exterior}\|_{H_h^2(\Omega_{\exterior})}=O(h^\infty),\qquad \|u_{\lambda_j}\|_{H_h^{\frac{3}{2}}(\partial\Omega)}\leq C. 
$$
To finish the proof of Theorem~\ref{t:plasmonic} it suffices apply Lemma~\ref{l:dtNDescriptionInner} to see that
$$
\|\psi G_{\inner}u_{\lambda_j,\inner}\|_{H_h^2(\Omega_{\inner})}=O(h^\infty).
$$
\end{proof}

\section{Counting of Surface Plasmons}
\label{s:counting}

In this section, we prove Theorem~\ref{t:count}. We start by finding an operator that is uniformly invertible near the real axis and approximates $(\Lambda_{\exterior}(z)-\tau\Lambda_{\inner}(z))^{-1}$ well.
\begin{lem}
\label{l:approxInverse}
Suppose that $Q\in\Psi^{\comp}_h(\partial\Omega)$ satisfy
\begin{equation}
\label{e:qRequirement}
\left\{\rho_{\exterior}^2|\xi'|_{g_\exterior }^2-\tau^2\rho_{\inner}^2|\xi'|_{g_{\inner}}^2=\rho_{\exterior}^2+\tau^2\rho_{\inner}^2\right\}\subset \Ellip(Q),\qquad \WF(Q)\subset \left\{|\xi'|_{g_{\exterior}}>1\right\}.
\end{equation}
Then, there is $\e>0$ such that for all $M>0$ there are $h_0,C>0$ such that for $0<h<h_0$, $z\in [1-2\e,1+2\e]+\rmi [-Mh,Mh],$
$$
R_Q(z):=(\Lambda_{\exterior}(z)-\tau\Lambda_{\inner}(z)-\rmi Q)^{-1}
$$
exists and satisfies
$$
\|R_Q(z)\|_{H_h^{\frac{1}{2}}(\partial\Omega)\to H_h^{\frac{3}{2}}(\partial\Omega)}\leq C.
$$
\end{lem}
\begin{proof}
Let $X_0, X_1 \in \Psi^{\comp}_h(\partial\Omega)$ such that $\WF(X_0)\subset \WF(I-X_1)$, $\WF(X_1)\subset\{ |\xi'|_{g_\exterior}>1+2\e\}$, $\WF(Q)\cap \WF(I-X_0)=\emptyset,$ and
\begin{gather*}
\WF( X_0) \subset \Ellip(X_1(\Lambda_{\exterior}-\tau \Lambda_{\inner}-\rmi Q)),\\ 
\WF(I-X_0)\cap \left\{\rho_{\exterior}^2|\xi'|_{g_\exterior }^2-\tau^2\rho_{\inner}^2|\xi'|_{g_{\inner}}^2=\rho_{\exterior}^2z^2+\tau^2\rho_{\inner}^2z^2\right\}=\emptyset,\quad z\in\blue{[1-2\e,1+2\e]}.
\end{gather*}

Then, the elliptic parametrix construction implies
\begin{align*}
&\| X_0u\|_{H_h^{s}(\partial\Omega)}\le C \| (\Lambda_{\exterior}(z)-\tau\Lambda_{\inner}(z)-\rmi Q) u\|_{H_h^{s-1}(\partial \Omega)}+Ch^N\|u\|_{H^{-N}_h(\partial \Omega)}.
\end{align*}
Hence, by Lemma~\ref{l:basicPropagation} (together with Remark~\ref{r:rescale}) with $X=I-X_0$ and $\tilde{X}=I$, we have
\begin{align*}
\|hD_{\nu_{\exterior}}u\|_{H_h^{\frac{1}{2}}(\partial\Omega)}&+\|u\|_{H_h^{\frac{3}{2}}(\partial\Omega)}+\|\chi u\|_{H_h^2(\Omega_\exterior)}\le C \left(\|Qu\|_{H^\frac{1}{2}_h(\partial \Omega)}+\|X_0u\|_{H_h^{\frac{3}{2}}(\partial\Omega)}\right)\\
&\le C \left(\|X_0 u\|_{H_h^{\frac{3}{2}}(\partial\Omega)}+h^N\|u\|_{H_h^{-N}(\partial\Omega)}\right)\\
&\le C \left(\| (\Lambda_{\exterior}(z)-\tau\Lambda_{\inner}(z)-\rmi Q) u\|_{H_h^{\frac{1}{2}}(\partial \Omega)} +h^N\|u\|_{H_h^{-N}(\partial\Omega)}\right) ,
\end{align*}
which completes the proof after absorbing the last term on the right-hand side.

\end{proof}

We now fix $Q$ satisfying~\eqref{e:qRequirement}, let $\e>0$ as in Lemma~\ref{l:approxInverse} and are interested in the number of resonances in 
$$
V_{\e}(h):=[1-\e,1+\e]+\rmi [-h,h].
$$
Define 
$$\mathcal{Z}_{\e}(h):=\{z\in V_{\e}(h)\,:\, (\Lambda_{\exterior}(z)-\tau\Lambda_{\inner}(z))\text{ is not invertible}\}.$$

The next lemma reduces counting the number of resonances in $V_\e(h)$ to counting the number of zeros of an analytic function and gives a crude upper bound on how many zeros there may be. 
\begin{lem}
\label{l:basicUpper}
There is $h_0>0$ such that for $0<h<h_0$, 
$$
\mathcal{Z}_{\e}=\{ z\in V_{\e}(h)\,: F(z)=0\},
$$
where 
$$
F(z):=\det( I+\rmi R_Q(z)Q).
$$
Moreover,there is $C>0$ such that 
\begin{equation}
\label{e:exponentialEstimates}
|F(z)|\leq \exp\left(Ch^{-d+1}\right),\qquad z\in V_{2\e}
\end{equation}
and
$$
N(h):=\#\mathcal{Z}_{\e}(h)\leq Ch^{-d-1}.
$$
\end{lem}
\begin{proof}
Observe that 
$$
\Lambda_{\exterior}-\tau\Lambda_{\inner}=(\Lambda_{\exterior}-\tau\Lambda_{\inner}-\rmi Q)(I+\rmi R_Q(z)Q).
$$
Therefore, since $(\Lambda_{\exterior}-\tau\Lambda_{\inner}-\rmi Q)^{-1}$ exists for all $z\in V_\e(h)$, $\Lambda_{\exterior}-\tau\Lambda_{\inner}$ is invertible if and only if $I+\rmi R_Q(z)Q$ is invertible. Since $Q\in \Psi^{\comp}_h$, $R_QQ$ is trace class and hence $I+\rmi R_Q(z)Q$ is invertible if and only if $F(z)\neq 0$. \blue{Let $\|\cdot\|_{\tr}$ denote the trace norm.}

Now, observe that 
\begin{equation}
\label{e:upper}
|F(z)|\leq \exp(\|R_Q(z)Q\|_{\tr})\leq \exp \left(Ch^{-d+1}\right),\quad z\in V_{2\e} (h).
\end{equation}
On the other hand, set $z_s=s+ih$, then
$$
(I+\rmi R_Q(z_s)Q)^{-1}=(\Lambda_{\exterior}(z_s)-\tau\Lambda_{\inner}(z_s))^{-1}(\Lambda_{\exterior}(z_s)-\tau\Lambda_{\inner}(z_s)-\rmi Q)=I-\rmi (\Lambda_{\exterior}(z_s)-\tau\Lambda_{\inner}(z_s))^{-1}Q.
$$
Therefore, for $s\in[1-\e,1+\e]$, using that $\|(\Lambda_{\exterior}(z_s)-\tau\Lambda_{\inner}(z_s)\|\leq Ch^{-1}$, we have
\begin{equation}
\label{e:lower}
|F(z_s)|^{-1}\leq \exp \left(\|(\Lambda_\exterior(z_s)-\tau\Lambda_{\inner}(z_s))^{-1}Q\|_{\tr}\right)\leq \exp \left(Ch^{-d}\right).
\end{equation}
Using~\eqref{e:upper} and~\eqref{e:lower} together with~\cite[(D1.11)]{DyZw:19}
$$
\#\{z\in [s-h,s+h]+\rmi[-h,h]\,:F(z)=0\}\leq Ch^{-d},\qquad s\in[1-\e,1+\e].
$$
Hence, 
$$
N(h)\leq Ch^{-d-1},
$$
as claimed.
\end{proof}

With the crude estimate on the number of resonances, in hand, we can now write an effective formula for counting zeros. 
\begin{lem}
\label{l:formula}
Let $\chi \in C_c^\infty((1-\e,1+\e))$. Then, 
\begin{equation}
\label{e:integralFormula}
\sum_{z_j\in \mathcal{Z}_\e}\chi(\Re z_j)=\frac{1}{2\pi \rmi}\int_{\partial V_{\e,N}} \tilde{\chi}(z)\frac{\partial_zF(z)}{F(z)}\der z+O(h^\infty),
\end{equation}
where $\tilde{\chi}$ is an almost analytic extension of $\chi$.
\end{lem}
\begin{proof}
 First, as in~\cite[page 375]{Dy:15} by~\cite[Lemma $\alpha$, Section 3.9]{Ti:86} and the estimate~\eqref{e:exponentialEstimates}
(splitting the region $V_\e(h)$ into  $h$ by $h$ squares and applying Lemma $\alpha$ to each square, transformed into the unit disk using the Riemann Mapping Theorem), we have
$$
\frac{\partial_z F(z)}{F(z)}=\sum_{z_j\in\mathcal{Z}_\e} \frac{1}{z-z_j}+G(z),\qquad |G(z)|\leq Ch^{-N},\qquad z\in V_\e(h)\cap \supp \tilde{\chi}.
$$
 Hence, applying Stokes formula in
$$
V_{\e}(h)  \setminus \cup_{z_j\in\mathcal{Z}(h)}B(z_j,r),
$$
sending $r\to 0$, and using that $\supp \tilde{\chi}\subset \{ \Re z\in(1-\e,1+\e)\}$, we obtain 
\begin{align*}
\frac{1}{2\pi \rmi}\int_{\partial V_\e} \tilde{\chi}(z)\frac{\partial_zF(z)}{F(z)}\der z&= \sum_{z_j\in \mathcal{Z}_\e(h)}\tilde{\chi}(z_j)+\frac{1}{2\pi \rmi}\int_{V_\e(h)} \partial_{\bar{z}}\tilde{\chi}\frac{\partial_z F(z)}{F(z)}d\bar{z}\wedge \der z\\
&=\sum_{z_j\in \mathcal{Z}_\e(h)}\tilde{\chi}(z_j)+O(h^\infty),
\end{align*}
where the last equality follows from the bound $N(h)\leq Ch^{-d-1}$, $|G(z)|\leq Ch^{-N}$, and $\partial_{\bar{z}}\tilde{\chi}=O(|\Im z|^{\infty}).$

Finally, since by Theorem~\ref{boundary defect measure thm n>1}, $|\Im z_j|=O(h^\infty)$, the lemma follows.
\end{proof}

In order to obtain an asymptotic formula for the integral in~\eqref{e:integralFormula}, we will need to have an accurate description of $(\Lambda_{\exterior}(z)-\tau\Lambda_{\inner}(z))^{-1}X$, for $z \in\Gamma^{\pm}_{\e}(h):=[1-\e,1+\e]\pm \rmi h$, for any $X\in \Psi^{\comp}_h(\partial\Omega)$ with $\WF(X)\subset \{|\xi'|_{g_{\exterior}}>(1+\e)^2\}$. 
\begin{lem}
\label{l:inverse}
Let $X\in \Psi^{\comp}_h(\partial\Omega)$ with $\WF(X)\subset \{|\xi'|_{g_{\exterior}}>\blue{(1+\e)^2}\}$. Then,
$$
(\Lambda_{\exterior}(z)-\tau\Lambda_{\inner}(z))^{-1}X= -\frac{\rmi}{h}\int_0^{\pm \infty} W^*U(t,z)WXe^{-\rmi tz^2/h}\der t,\qquad \pm \Im z^2\leq -h^N, 
$$
where, $W\in \Psi^{\comp}_h(\partial \Omega)$ and for some $\tilde \chi\in C_c^\infty(\{|\xi'|_{g_{\exterior}}>(1+\e)^2\};[0,1])$ with $\supp (1-\tilde \chi)\cap \WF(X)=\emptyset$, 
$$
\sigma(W)= \sqrt{\frac{\rho_{\exterior}\sqrt{|\xi'|_{g_{\exterior}}^2-z^2}+\tau\rho_{\inner}\sqrt{|\xi'|_{g_{\inner}}^2+z^2}}{\rho_\exterior^2+\tau^2 \rho_\inner^2}}\tilde{\chi},
$$
there is $c>0$ such that for any \blue{$\alpha \in \N$}
$$
\|D_z^\alpha U(\pm t,z)e^{-\rmi tz^2/h}\|_{L^2(\partial\Omega)\to L^2(\partial\Omega)}\leq C_\alpha e^{-c|\Im z|^2t/h},\qquad \pm \Im z^2\leq-h^N, t\geq 0,
$$
and 
$$
(hD_t-B(z))U(t,z)=0,\qquad U(0,z)=I
$$
for $B\in \Psi^2_h(\partial \Omega)$ satisfying 
$$
\sigma(B)=\frac{\rho_{\exterior}^2|\xi'|^2_{g_{\exterior}}-\tau^2\rho_{\inner}^2|\xi'|_{g_{\inner}}^2}{\rho_\exterior^2+\tau^2 \rho_\inner^2}\chi^2,
$$
where $ \chi\in C_c^\infty(\{|\xi'|_{g_{\exterior}}>(1+\e)^2\};[0,1])$ with $\supp (1-\tilde \chi)\cap \supp(\chi)=\emptyset$.
\end{lem}
\begin{proof}
Since the analysis only happens at the boundary, we will omit the argument of Sobolev spaces. Define
$$
W:=\Op\left( \sqrt{\frac{\rho_{\exterior}\sqrt{|\xi'|_{g_{\exterior}}^2-z^2}+\tau\rho_{\inner}\sqrt{|\xi'|_{g_{\inner}}^2+z^2}}{\rho_\exterior^2+\tau^2 \rho_\inner^2}}\tilde{\chi}\right).
$$
Now, let $X_1\in \Psi^{\comp}_h(\partial\Omega)$ with $\WF(X_1)\cap \supp (1-\chi)=\emptyset$ and $\WF(X)\cap \WF(I-X_1)=\emptyset$.   

Set 
$$
B(z):=\Op(\chi)^*\left(W(\Lambda_{\exterior}-\tau\Lambda_{\inner})W^*+z^2\right)\Op(\chi).
$$
Then,
$$
X_1W(\Lambda_{\exterior}-\tau\Lambda_{\inner})W^*= X_1\left(B(z)-z^2\right)+O(h^\infty)_{\Psi_h^{-\infty}}, 
$$
and, by~\eqref{e:derivativesIn} and~\eqref{e:derivativesOut},
$$
\sigma(B(z))=\frac{\rho_{\exterior}^2|\xi'|^2_{g_{\exterior}}-\tau^2\rho_{\inner}^2|\xi'|_{g_{\inner}}^2}{\rho_\exterior^2+\tau^2 \rho_\inner^2}\chi^2,\qquad \partial_z^\alpha B(z)\in h\Psi^{\comp}_h.
$$

Moreover, using Proposition~\ref{p:imagPart},
$$
\begin{aligned}
&-\sgn(\Im z^2)\Im \left\langle (B(z)-z^2)u,u\right\rangle\\
&=-\sgn(\Im z^2)\Im \left\langle \left[\Op(\chi)^*W(\Lambda_{\exterior}-\tau\Lambda_{\inner})W^*\Op(\chi)-z^2\left(1-\Op(\chi)^*\Op(\chi)\right)\right]u,u\right\rangle\\
&\geq c_0\left|\Im z^2\right|\|W^*\Op(\chi)u\|_{L^2}^2+\left|\Im z^2\right|\langle (1-\Op(\chi)^*\Op(\chi))u,u\rangle-C_Nh^N\|u\|_{L^2}^2\\
&= \left|\Im z^2\right|\left\langle \left(1-\Op(\chi)^*\Op(\chi)+c\Op(\chi)^*W W^*\Op(\chi)\right)u,u\right\rangle-C_Nh^N\|u\|_{L^2}^2\\
&\geq \left(c\left|\Im z^2\right|-C_Nh^N\right)\|u\|_{L^2}^2,
\end{aligned}
$$
where the last line follows from G\aa rding's inequality and they fact that 
$$
\sigma(1-\Op(\chi)^*\Op(\chi)+c_0\Op(\chi)^*W W^*\Op(\chi))=1-\chi^2(1-c_0 W^2)\geq 2c>0.
$$

Next, observe that, if $(B(z)-z^2 )^{-1}$ exists (and is polynomially bounded in $h$), then for any $A\in \Psi^0_h$, with $\WF(I-A)\cap \WF(X)=\emptyset$, 
$$
(B(z)-z^2)^{-1}X=A(B(z)-z^2)^{-1}X+O(h^\infty)_{\Psi_h^{-\infty}},
$$
and hence, since $W$ is elliptic on $\WF(X_1)$, 
\begin{align*}
(\Lambda_\exterior-\tau \Lambda_{\inner})W^*(B(z)-z^2)^{-1}WX&=X_1W(\Lambda_\exterior-\tau \Lambda_{\inner})W^*(B(z)-z^2)^{-1}WX+O(h^\infty)_{\Psi_h^{-\infty}}\\
&=X+O(h^\infty)_{\Psi^{-\infty}}.
\end{align*}
In particular, since $(\Lambda_{\exterior}-\tau\Lambda_{\inner})^{-1}$ is tempered and bounded in $h$, 
$$
(\Lambda_{\exterior}-\tau\Lambda_{\inner})^{-1}X=W^*(B(z)-z^2)^{-1}WX+O(h^\infty)_{\Psi_h^{-\infty}}, \qquad |\Im z|\geq h^N.
$$

Therefore, we only need to invert $B(z)-z^2$. For this, define $U(t,z)$ by
$$
(hD_t-B(z))(U(t,z))=0,\qquad U(0,z)=I. 
$$
Then, observe that on 
\begin{align*}
&-\sgn(\Im z^2)h\partial_t\|U(t,z)e^{-\rmi tz^2/h}u_0\|_{L^2}^2\\
&=-2\sgn(\Im z^2)\Re\langle h\partial_t (Ue^{-\rmi tz^2/h} u_0),Ue^{-\rmi tz^2/h}u_0\rangle \\
&= 2\sgn(\Im z^2)\Im \langle (B(z)-z^2)Ue^{-\rmi tz^2/h} u_0,Ue^{-\rmi tz^2/h} u_0\rangle \\
&\leq -\left(c\left|\Im z^2\right|-C_Nh^N\right)\left\|U(t,z)e^{-\rmi tz^2/h} u_0\right\|_{L^2}^2.
\end{align*}
In particular, 
$$
\left\|U(\pm t,z)e^{-\rmi tz^2/h}u_0\right\|_{L^2}^2\leq Ce^{-c\left|\Im z^2\right|t/h}\|u_0\|_{L^2}^2,\qquad \pm\Im z^2\leq -h^N,\, t\geq 0.
$$

Thus, we have
$$
(B(z)-z^2)^{-1}= -\frac{\rmi}{h}\int_0^{\pm \infty} U(t,z)e^{-\rmi tz^2/h}\der t,\qquad \pm \Im z^2\leq -h^N.
$$

Finally, observe that 
$$
(hD_t-B)D_zU=(D_zB)U,\qquad D_zU(0)=0.
$$
So that for $\Im z^2\leq -h^N$, and $t\geq 0$, 
\begin{align*}
\left\|D_zU(t,z)e^{-\rmi tz^2/h}\right\|_{L^2\to L^2}&\leq\frac{1}{h}\int_0^t\left\|U(t-s,z)e^{-\rmi (t-s)z^2/h}(D_zB)U(s,z)e^{-\rmi sz^2/h}\right\|_{L^2\to L^2}\der s\\
&\leq \frac{t}{h}e^{-ct|\Im z^2|/h}\leq C_1 e^{-ct|\Im z^2|/h}.
\end{align*}
Now, suppose 
\begin{equation*}
\left\|D_z^kU(t,z)e^{-\rmi tz^2/h}\right\|_{L^2\to L^2}\le C_k e^{-ct|\Im z^2|/h}
\end{equation*}
is true for some $C_k$ and $k\le J$. Now, we have
$$
(hD_t-B)D_z^{J+1}U=\sum_{k=1}^{J+1}\binom{J+1}{k}D_z^k B D_z^{J+1-k}U,\qquad D_z^{J+1}U(0)=0.
$$
Then
\begin{align*}
&\left\|D_z^{J+1}U(t,z)e^{-\rmi tz^2/h}\right\|_{L^2\to L^2}
\\
&\quad \leq\frac{1}{h}\sum_{k=1}^{J+1}\binom{J+1}{k}\int_0^t\left\|U(t-s,z)e^{-\rmi (t-s)z^2/h}(D_z^k B) D_z^{J+1-k}U(s,z)e^{-\rmi sz^2/h}\right\|_{L^2\to L^2}\der s\\
&\quad  \leq \frac{C_B}{h}\sum_{k=1}^{J+1}\binom{J+1}{k} C_{J+1-k}e^{-ct|\Im z^2|/h}\leq C_{J+1} e^{-ct|\Im z^2|/h},
\end{align*}
which proves the induction step and hence the Lemma.

\end{proof}

We can now prove Theorem~\ref{t:count}.
\begin{proof}[Proof of Theorem~\ref{t:count}]
We start by computing, 
$$
N_\chi:=\frac{1}{2\pi \rmi}\int_{\Gamma_{\pm}} \tilde{\chi}(z)\frac{\partial_zF(z)}{F(z)}\der z=\sum_{z_j\in\mathcal{Z}_\e}\chi(\Re z_j)+O(h^\infty),
$$
where 
$$
\Gamma_{\pm}:=[1-\e,1+\e]\pm \rmi h,
$$
oriented to the left and right respectively. Here we have used Lemma~\ref{l:formula} to obtain the second equality.

Observe first that, using cyclicity of the trace and the fact that $R_QQ\in \Psi^{\comp}_h$, we obtain
\begin{align*}
\frac{\partial_zF(z)}{F(z)}&=\tr (I+\rmi R_Q(z)Q)^{-1}\rmi \partial_z R_Q(z)Q\\
&=-\rmi\tr (\Lambda_{\exterior}-\tau \Lambda_{\inner})^{-1}\partial_z(\Lambda_{\exterior}-\tau\Lambda_{\inner})R_Q(z)Q\\
&=-\rmi\tr X(\Lambda_{\exterior}-\tau \Lambda_{\inner})^{-1}X\partial_z(\Lambda_{\exterior}-\tau\Lambda_{\inner})XR_Q(z)Q +O(h^\infty),
\end{align*}
where $X\in \Psi^{\comp}_h$ with $\WF(Q)\cap \WF(I-X)=\emptyset. $

We can now use Lemma~\ref{l:inverse} to write 
\begin{align*}
N_\chi:=\sum_{\pm}\frac{\rmi}{2\pi h} \int_{\Gamma_{\mp}}\tilde{\chi}(z)\int_0^{\pm\infty}\tr XW^*U(t,z)WXe^{-\rmi tz^2/h}\partial_z(\Lambda_{\exterior}-\tau\Lambda_{\inner})XR_Q(z)Q\der t\der z+O(h^\infty).
\end{align*}
We may now integrate by parts using $-hD_z/(2z)e^{-\rmi tz^2/h}=e^{-\rmi tz^2/h}$ to see that for $\rho\in C_c^\infty$ with $1\notin \supp (1-\rho)$, 
$$
N_\chi=\sum_{\pm}\frac{\rmi}{2\pi h} \int_{\Gamma_{\mp}}\tilde{\chi}(z)\int_0^{\pm\infty}\rho(t)\tr XW^*U(t,z)WXe^{-\rmi tz^2/h}\partial_z(\Lambda_{\exterior}-\tau\Lambda_{\inner})XR_Q(z)Q\der t \der z+O(h^\infty).
$$
Applying Stokes theorem on $[1-\e,1+\e]\times \rmi [0,h]$for the integral over $\Gamma_+$ and on $[1-\e,1+\e]\times -\rmi[0,h]$, we obtain
\begin{equation}
\label{e:finalForm}
\begin{aligned}
N_\chi&=\sum_{\pm}\mp\frac{\rmi}{2\pi h} \int_{\mathbb{R}}\chi(z)\int_0^{\pm\infty}\rho(t)\tr XW^*U(t,z)WXe^{-\rmi tz^2/h}\partial_z(\Lambda_{\exterior}-\tau\Lambda_{\inner})XR_Q(z)Q \der t \der z+O(h^\infty)\\
&=\frac{\rmi}{2\pi h} \int_{\mathbb{R}}\chi (z)\int\rho(t)\tr XW^*U(t,z)WXe^{-\rmi tz^2/h}\partial_z(\Lambda_{\exterior}-\tau\Lambda_{\inner})XR_Q(z)Q \der t \der z+O(h^\infty).
\end{aligned}
\end{equation}

We now use~\cite[Theorem 10.4]{Zw:12} to write in local coordinates
$$
U(t,z)e^{-\rmi tz^2/h}Xu=\frac{1}{(2\pi h)^{d-1}}\int e^{\frac{\rmi}{h}(\varphi(t,x,\eta)-\langle y,\eta\rangle -tz^2)}a(t,x,\eta,z)\der \eta (Xu)(y)\der y +O\left(h^\infty\|u\|_{H_h^{-N}(\partial \Omega)}\right), 
$$
where, with $b=\sigma(B)$,
$$
\partial_t\varphi(x,\eta)=b(x,\partial_x \varphi),\qquad \varphi(0,x,\eta)=\langle x,\eta\rangle,\qquad a(0,x,\eta,z)=1.
$$
We can then perform stationary phase in $(t,z)$~\eqref{e:finalForm} with critical point $t=0$, and $z_c(x,\eta)=\sqrt{b(x,\eta)}$. Or equivalently,
$$
z_c^2(x,\eta)=\frac{\rho_{\exterior}^2|\eta|^2_{g_{\exterior}}-\tau^2\rho_{\inner}^2|\eta|_{g_{\inner}}^2}{\rho_\exterior^2+\tau^2 \rho_\inner^2}.
$$
Moreover, $\sigma\left(\partial_z(\Lambda_{\exterior}-\tau\Lambda_{\inner})\right)=-z\left(\frac{\rho_\exterior}{\sqrt{|\xi'|_{g_\exterior}^2-z^2}}+\frac{\tau \rho_\inner}{\sqrt{|\xi'|_{g_\inner}^2+z^2}}\right)$
We then compute the trace by restricting to the diagonal and integrating. This calculations yields
$$
N_{\chi}=\frac{1}{(2\pi h)^{d-1}}\int \chi(\sqrt{b(x,\eta)})\der x \der\eta.
$$
Taking a sequence of $\chi$ approximating $1_{[1-\e,1+\e]}$ shows that 
$$
\#\{z_j\in\mathcal{Z}_\e\}=(2\pi h)^{1-d}\vol_{T^*\partial\Omega}\left\{(x,\xi)\,:\, (1-\e)^2\leq\frac{\rho_{\exterior}^2|\xi'|_{g_{\exterior}}^2-\tau^2\rho_{\inner}^2|\xi'|_{g_{\inner}}^2}{ \rho_{\exterior}^2+\tau^2\rho_{\inner}^2}\leq (1+\e)^2\right\} +o(h^{1-d}).
$$

Now, set $\alpha=\frac{1+\e}{1-\e}$, $h_j:=(1+\e)\lambda^{-1}\alpha^{j}$ and observe that 
\begin{align*}
&\#\{ \lambda_j\in \calR(P)\,:\, 0< \Re \lambda_j \leq \lambda\,:\Im \lambda_j\geq -M\}\\
&=\sum_{j=0}^{\lfloor\log_{\alpha}\lambda\rfloor}\#\{ \lambda_j\in \calR(P)\,:\, \alpha^{-j-1}\lambda< \Re \lambda_j \leq \alpha^{-j}\lambda,:\Im \lambda_j\geq -M\}+O(1)\\
&=\sum_{j=0}^{\lfloor\log_{\alpha}\lambda\rfloor}(2\pi h_j)^{1-d}\vol_{T^*\partial\Omega}\left\{(x,\xi)\,:\, (1-\e)^2\leq\frac{\rho_{\exterior}^2|\xi'|_{g_{\exterior}}^2-\tau^2\rho_{\inner}^2|\xi'|_{g_{\inner}}^2}{ \rho_{\exterior}^2+\tau^2\rho_{\inner}^2}\leq (1+\e)^2\right\}+o(h_j^{1-d})\\
&=o(\lambda^{d-1})+\vol_{T^*\partial\Omega}(\mathcal{V})\left((1+\e)^{d-1}-(1-\e)^{d-1}\right)\lambda^{d-1}(1+\e)^{1-d}(2\pi)^{1-d}\sum_{j=0}^{\lfloor\log_{\alpha}\lambda\rfloor} \alpha^{j(1-d)} \\
&= (2\pi)^{1-d}\vol_{T^*\partial\Omega}(\mathcal{V})\lambda^{d-1}+o(\lambda^{d-1}),
\end{align*}
which completes the proof.
\end{proof}

\appendix
\section{Properties of the operator $P$}
\label{a:blackBox}

In this section, we show that $P$ with domain~\eqref{e:domain1} is a black-box Hamiltonian.

We begin with a technical lemma.
\begin{lem}
\label{l:index0}
Suppose that $M$ is a smooth, closed manifold (compact without boundary), $m\in\mathbb{R}$, and $A\in \Psi_1^m(M)$ (i.e. a classical pseudodifferential operator of order $m$) with 
\begin{equation}
\label{e:index0condition}
|\sigma(A)(x,\xi)|\geq c| \xi|^m,\qquad -[0,\infty)\cap \{ \sigma(A)(x,\xi)\,:\, (x,\xi)\in S^*M\}=\emptyset.
\end{equation}
Then $A$ is a Fredholm operator with index $0$.
\end{lem}
\begin{proof}
Let $a=\sigma(A)$. (Recall that $a$ is homogeneous degree $m$ in $\xi$.)
$$
A-\operatorname{Op}_1(a)=:R\in \Psi^{m-1}(M).
$$
Let $\chi_i\in C_c^\infty(\mathbb{R};[0,1])$, $i=1,2$ with $\chi_i\equiv 1$ near $0$ and $\supp \chi_1\cap \supp(1-\chi_2)=\emptyset$. 
For $h>0$, let 
$$
A_h:=\operatorname{Op}_1(a(x,h\xi)(1-\chi_1(h|\xi|))+\chi_2(h|\xi|)).
$$
Observe that by~\eqref{e:index0condition}
$$
|a(x,h\xi)(1-\chi_1(h|\xi|))+\chi_2(h|\xi|)|\geq c\langle h\xi\rangle^m.
$$
Moreover, 
$$
A_h=\Op_h(a(x,\xi)(1-\chi_1(|\xi|))+\chi_2(|\xi|))\in \Psi_h^m(M).
$$
Thus, there is $E\in \Psi_h^{-m}$ such that 
$$
EA_h=I+hR_{-1},\qquad R_{-1}\in \Psi_h^{-1}(M).
$$
In particular, since
$$
\|R_{-1}\|_{H_h^s(M)\to H_h^s(M)}\leq C_sh,
$$
for $h$ small enough $A_h$ is invertible and $A_h^{-1}\in \Psi_h^{-m}(M)$.

Now, 
\begin{align*}
A_h^{-1}h^mA&=A_h^{-1}(\operatorname{Op}_{1}(a(x,h\xi))+h^mR)\\
&=A_h^{-1}((A_h+\operatorname{Op}_{1}(\chi_1(h|\xi|)a(x,h|\xi|)-\chi_2(h|\xi|))+h^mR)=I+K,
\end{align*}
where $K:H^s(M)\to H^{s+1}(M)$ and hence is compact. 
In particular, 
$$
A=h^{-m}A_h +\tilde{K},
$$
where $\tilde{K}:H^{s+m}(M)\to H^{s+m+1}(M)\hookrightarrow H^{s}(M)$ is compact and hence, since $A_h:H^{s+m}(M)\to H^s(M)$ is invertible, $A$ is Fredholm with index 0.
\end{proof}

We now use Lemma~\ref{l:index0} to study the operator $P$. 
\begin{lem}
The operator $P$ is self-adjoint.
\end{lem}
\begin{proof}
We start by showing that $P$ is symmetric. Notice that, integration by parts implies that for $u,v\in H^2(\Omega_{\exterior})$,
    \begin{align*}
    &\langle \Delta_{g_\exterior,\rho_\exterior}u,v\rangle_{L^2(\Omega_{\exterior}, \rho_\exterior \der \vol_{g_\exterior})}\\
    &\qquad= \langle u,\Delta_{g_\exterior,\rho_\exterior}v\rangle_{L^2(\Omega_{\exterior}, \rho_\exterior \der \vol_{g_\exterior})} + \langle \rho_\exterior\partial_{\nu_{g_\exterior}}u,v\rangle_{L^2(\partial\Omega, \der \vol_{g_\exterior,\partial\Omega})}-\langle  u,\rho_\exterior\partial_{\nu_{g_\exterior}}v\rangle_{L^2(\partial\Omega, \der \vol_{g_\exterior,\partial\Omega})}.
    \end{align*}
 In addition, for $u,v\in H^2(\Omega_{\inner})$, 
    \begin{align*}
    &\langle \Delta_{g_\inner,\rho_\inner}u,v\rangle_{L^2(\Omega_{\inner}, \rho_\inner \der \vol_{g_\inner})}\\
    &\qquad= \langle u,\Delta_{g_\inner,\rho_\inner}v\rangle_{L^2(\Omega_{\inner}, \rho_\inner \der \vol_{g_\inner})} + \langle \tau \rho_\inner\partial_{\nu_{g_\inner}}u,v\rangle_{L^2(\partial\Omega, \der \vol_{g_\inner,\partial\Omega})}-\langle  u,\tau \rho_\inner\partial_{\nu_{g_\inner}}v\rangle_{L^2(\partial\Omega, \der \vol_{g_\exterior,\partial\Omega})}.
    \end{align*}
    In particular, the operator $P$ is symmetric. 

    We next show that there is $z$ with $\Im z>0$ such that $(P-z^2):\mathcal{D}(P)\to L^2$  and $(P-\overline{z^2}):\mathcal{D}(P)\to L^2$ are surjective. This then implies that $P-\Re (z^2)$ and hence also $P$ is self-adjoint. 
    
    To do this, recall the definitions of $R_{\exterior}(z)$, $G_{\exterior}(z)$, $R_{\inner}(z)$ and $G_{\inner}(z)$ from Sections~\ref{Reformulation of the problem as an exterior problem} and~\ref{s:interiorResolve}, and note that $R_{\exterior}:L^2(\Omega_{\exterior})\to H^2(\Omega_{\exterior})$, $G_{\exterior}:H^{\frac{3}{2}}(\partial\Omega)\to H^2(\Omega_{\exterior})$, $R_{\inner}:L^2(\Omega_{\inner})\to H^2(\Omega_{\inner})$, and $G_{\inner}:H^{\frac{3}{2}}(\partial\Omega)\to H^2(\Omega_{\inner})$ are analytic families of operators in $\Im z>0$. 
    
    In particular, $\Lambda_{\exterior}(z)-\tau\Lambda_{\inner}(z):H^{\frac{3}{2}}(\partial\Omega)\to H^{\frac{1}{2}}(\partial\Omega)$ is an analytic family of operators in $\Im z>0$. Moreover, since $\Lambda_{\exterior}-\tau\Lambda_{\inner}\in \Psi^1_1(\partial \Omega)$ (i.e. is a non-semiclassical pseudodifferential operator of order 1) has real principal symbol is elliptic in this class Lemma~\ref{l:index0} implies that $\Lambda_{\exterior}-\tau\Lambda_{\inner}$, is an analytic family of Fredholm operators with index 0 in $\Im z>0$. Thus, by the analytic Fredholm theorem, $(\Lambda_{\exterior}(z)-\tau\Lambda_{\inner}(z))^{-1}$ is a meromorphic family of Fredholm operators with index 0 in $\Im z>0$.  

    Using this, we have, in $\Im z>0$
    $$
    (P-z^2)\left(I+\begin{pmatrix} G_{\inner}(z)\\G_{\exterior}(z)\end{pmatrix}(\Lambda_{\exterior}(z)-\tau\Lambda_{\inner}(z))^{-1}\begin{pmatrix}\tau\rho_{\inner}\partial_{\nu_{\inner}}&-\rho_{\exterior}\partial_{\nu_{\exterior}}\end{pmatrix}\right)\begin{pmatrix} R_{\inner}(z)&0\\0&R_{\exterior}(z)\end{pmatrix}=I_{L^2\to L^2}.
    $$
    In particular, since the poles of $(\Lambda_{\exterior}(z)-\tau \Lambda_{\inner}(z) )^{-1}$ form a discrete set in $\Im z>0$, one can find $z$ such that $(P-z^2):\mathcal{D}(P)\to L^2$ and \blue{$(P-{\overline{z^2}}):\mathcal{D}(P)\to L^2$} are surjective. 
    \end{proof}

\bibliographystyle{amsalpha}
\bibliography{biblio}

\end{document}